\documentclass[reqno,12pt,letterpaper]{amsart}
\usepackage{amsmath,amssymb,amsthm,graphicx,mathrsfs,url}
\usepackage[usenames,dvipsnames]{color}
\usepackage[colorlinks=true,linkcolor=Red,citecolor=Green]{hyperref}

\def\smallsection#1{\smallskip\noindent\textbf{#1}.}

\newtheorem{theo}{Theorem}
\newtheorem{prop}{Proposition}[section]
\newtheorem{defi}[prop]{Definition}

\newtheorem{lemm}[prop]{Lemma}

\numberwithin{equation}{section}

\DeclareMathOperator{\Res}{Res}

\DeclareMathOperator{\comp}{comp}
\DeclareMathOperator{\even}{even}

\let\Im=\Imag
\DeclareMathOperator{\loc}{loc}
\DeclareMathOperator{\Op}{Op}

\let\Re=\Real

\DeclareMathOperator{\supp}{supp}
\DeclareMathOperator{\Vol}{Vol}
\DeclareMathOperator{\WF}{WF}
\DeclareMathOperator{\Tr}{Tr}
\def\WFh{\WF_h}

\def\II{I_{\comp}(\Lambda^\circ)}

\title[Resonance projectors and asymptotics]%
{Resonance projectors and asymptotics\\
for $r$-normally hyperbolic trapped sets}
\author{Semyon Dyatlov}
\email{dyatlov@math.berkeley.edu}
\address{Department of Mathematics, University of California,
Berkeley, CA 94720, USA}

\begin{document}

\begin{abstract}
We prove a Weyl law for scattering resonances
in a strip near the real axis when the trapped set is $r$-normally hyperbolic with $r$ large
and a pinching condition on the normal expansion rates holds. Our
dynamical assumptions are stable under smooth perturbations and
are motivated by wave dynamics for black holes.  The key step is a construction
of a Fourier integral operator which microlocally projects onto the resonant states.
In addition to the Weyl law, this operator provides new information
about microlocal properties of resonant states.
\end{abstract}

\maketitle

\addtocounter{section}{1}
\addcontentsline{toc}{section}{1. Introduction}

For a Schr\"odinger operator $h^2\Delta_g+V(x)$, $V\in
C^\infty(X;\mathbb R)$, on a compact Riemannian manifold $(X,g)$ the
Weyl law (see for example~\cite[Theorem~10.1]{d-sj}) provides an
asymptotic for the number of eigenvalues (bound states) $\lambda_j(h)$ as $h\to 0$:
\begin{equation}
  \label{e:weyl-compact}
\#(\lambda_j(h)\in [\alpha_0,\alpha_1])=(2\pi h)^{-n}\big(\Vol_{\sigma}(p_V^{-1}([\alpha_0,\alpha_1]))+\mathcal O(h)\big).
\end{equation}
Here $n$ is the dimension of $X$, $p_V(x,\xi)=|\xi|_{g}^2+V(x)$ is the
(semiclassical) principal symbol of the Schr\"odinger operator,
defined on the cotangent bundle $T^*X$, and $\Vol_\sigma$ is the
symplectic volume on $T^*X$.

Scattering resonances are a natural generalization of bound states to
noncompact manifolds; they are the poles of the meromorphic
continuation of the resolvent to the lower half-plane $\{\Im\omega\leq
0\}\subset \mathbb C$, see~\eqref{e:r-v}
and~\S\S\ref{s:framework-schrodinger}, \ref{s:framework-ah}. However,
there are very few results giving Weyl asymptotics of resonances in
the style of~\eqref{e:weyl-compact}.  The first one is probably due to
Regge~\cite{regge}, with some of the following results
including~\cite{zw-u,sj-v,sj-z-bands,sj-new,f-t}~-- see the discussion
of related work below.

This paper provides a new Weyl asymptotic formula for resonances,
under the assumption that the trapped set is \emph{$r$-normally
hyperbolic} and expansion rates satisfy a pinching condition~-- see
Theorems~\ref{t:gaps} and~\ref{t:weyl-law}. These dynamical
assumptions are motivated by the study of black holes, see~\cite{k-s};
this continues the previous work of the author~\cite{skds,xpd,zeeman},
and the application to stationary perturbations of Kerr--de Sitter
black holes is given in~\cite{thesis}. See also~\cite{GSWW} and~\cite[Remark~1.1]{n-z-new} for
applications of normally hyperbolic trapping to molecular dynamics.
Since the imaginary part of a resonance can be interpreted as the
exponential decay rate of the corresponding linear wave, we study
\emph{long-living resonances}, that is those in strips of size $Ch$
around the real axis. More precisely, we establish an asymptotic
formula for the number of resonances in a band located between two
resonance free strips.

\smallsection{Setup}
To illustrate the results, we consider
semiclassical Schr\"odinger operators on $X=\mathbb R^n$,
studied in detail in~\S\ref{s:framework-schrodinger}:
\begin{equation}
  \label{e:P-V}
P_V:=h^2\Delta+V(x),\quad V\in C_0^\infty(\mathbb R^n;\mathbb R).
\end{equation}
Here $\Delta=-\sum_j \partial_{x_j}^2$ is the Euclidean Laplacian.
The results apply under the more general assumptions of~\S\S\ref{s:framework-assumptions}
and~\ref{s:dynamics}, in particular in the setting of even asymptotically hyperbolic manifolds~--
see~\S\ref{s:framework-ah} and Appendix~\ref{s:example}.
\emph{Resonances} are the poles of the meromorphic continuation of the resolvent
\begin{equation}
  \label{e:r-v}
R_V(\omega)=(P_V-\omega^2)^{-1}:L^2(\mathbb R^n)\to H^2(\mathbb R^n),\quad
\Im\omega>0,
\end{equation}
across the ray $(0,\infty)\subset\mathbb C$, as a family of operators $L^2_{\comp}(\mathbb R^n)\to H^2_{\loc}(\mathbb R^n)$.
For the proofs, it is convenient to consider a different operator with the same set of poles
\begin{equation}
  \label{e:r-cal-omega}
\mathcal R(\omega)=\mathcal P(\omega)^{-1}:\mathcal H_2\to\mathcal H_1,
\end{equation}
where $\mathcal H_1=H^2_h(\mathbb R^n)$ is a semiclassical Sobolev space,
$\mathcal H_2=L^2(\mathbb R^n)$, and $\mathcal P(\omega):\mathcal H_1\to\mathcal H_2$
is constructed from $P_V$ using the method of complex scaling (see~\S\ref{s:framework-schrodinger}).

To formulate dynamical assumptions, let $p_V(x,\xi)=|\xi|^2+V(x)$, fix energy intervals
$[\alpha_0,\alpha_1]\Subset [\beta_0,\beta_1]\subset (0,\infty)$,
put $p=\sqrt{p_V}$ on $p_V^{-1}([\beta_0^2,\beta_1^2])$
(see~\eqref{e:p} for the general case)
and define the \emph{incoming/outgoing tails} $\Gamma_\pm$ and the
\emph{trapped set} $K$ as
$$
\Gamma_\pm:=\{\rho\in p_V^{-1}([\beta_0^2,\beta_1^2])\mid
\exp(tH_p)(\rho)\not\to \infty\text{ as }t\to\mp\infty\},\quad
K:=\Gamma_+\cap\Gamma_-.
$$
Here $\exp(tH_{p})$ denotes the Hamiltonian flow of $p$.
We assume that (see~\S\ref{s:dynamics} for details)
$\Gamma_\pm$ are sufficiently smooth
codimension one submanifolds intersecting transversely at $K$, which is symplectic,
and the flow is \emph{$r$-normally hyperbolic} for large $r$ in the sense
that the minimal expansion rate $\nu_{\min}$ of the flow $\exp(tH_p)$ in the directions transverse to $K$ is much greater
than the maximal expansion rate $\mu_{\max}$ along $K$~-- see~\eqref{e:nu-min}, \eqref{e:mu-max}, \eqref{e:r-nh}.
These assumptions are stable under small smooth perturbations
of the symbol $p$, using the results of~\cite{HPS}~-- see~\S\ref{s:stability}.

\begin{figure}
\includegraphics{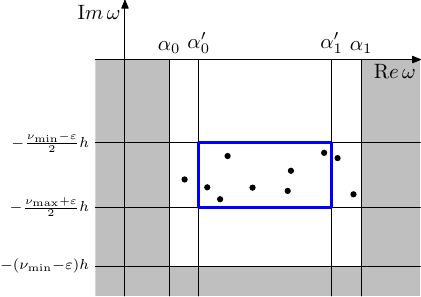}
\qquad
\includegraphics{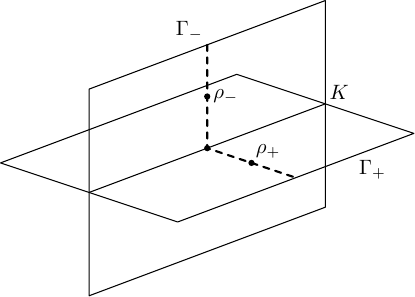}
\hbox to\hsize{\hss (a)\hss\hss (b)\hss}
\caption{(a) An illustration of Theorem~\ref{t:weyl-law}, with~\eqref{e:weyl-law}
counting resonances in the outlined box. The unshaded regions above
and below the box are the resonance-free regions of Theorem~\ref{t:gaps}.
(b) The canonical relation $\Lambda^\circ$, with the flow lines of $\mathcal V_\pm$ dashed.}
\label{f:Lambda}
\end{figure}

\smallsection{Distribution of resonances}
Let $\nu_{\max}$ be the maximal expansion
rate of the flow $\exp(tH_p)$ in the directions transverse to the trapped set,
see~\eqref{e:nu-max}. The following theorem provides a resonance free region
with a polynomial resolvent bound:
\begin{theo}
  \label{t:gaps}
Let the assumptions of~\S\S\ref{s:framework-assumptions} and~\ref{s:dynamics}
hold and fix $\varepsilon>0$. Then for
\begin{equation}
  \label{e:omega-gaps}
\Re\omega\in [\alpha_0,\alpha_1],\quad
\Im\omega\in [-(\nu_{\min}-\varepsilon)h,0]\setminus
\textstyle{1\over 2}(-(\nu_{\max}+\varepsilon)h,-(\nu_{\min}-\varepsilon)h),
\end{equation}
$\omega$ is not a resonance and we have the bound%
\footnote{The estimate~\eqref{e:gaps} implies, in the case~\eqref{e:P-V}, cutoff resolvent
bounds $\|\chi R_V(\omega) \chi\|_{L^2\to H^2_h}=\mathcal O(h^{-2})$
for any fixed $\chi\in C_0^\infty(\mathbb R^n)$. This explicit bound improves slightly
the bounds on the decay of correlations in~\cite[Theorem~1]{n-z-new}.}
\begin{equation}
  \label{e:gaps}
\|\mathcal R(\omega)\|_{\mathcal H_2\to\mathcal H_1}\leq Ch^{-2}.
\end{equation}
\end{theo}
In particular, we get a resonance free strip $\{\Im\omega>-{\nu_{\min}-\varepsilon\over 2}h\}$,
recovering in our situation the results of~\cite{ge-sj,w-z,n-z-new}.

Under the pinching condition
\begin{equation}
  \label{e:pinching}
\nu_{\max}<2\nu_{\min},
\end{equation}
we get a \emph{second} resonance free strip $\{\Im\omega\in [-(\nu_{\min}-\varepsilon)h,-(\nu_{\max}+\varepsilon)h/2]\}$.
We can then count the resonances in the band between the two strips, see Figure~\ref{f:Lambda}(a):
\begin{theo}
  \label{t:weyl-law}
Let the assumptions of~\S\S\ref{s:framework-assumptions} and~\ref{s:dynamics}
and the condition~\eqref{e:pinching} hold.
Fix $\varepsilon>0$ such that $\nu_{\max}+\varepsilon<2(\nu_{\min}-\varepsilon)$.
Then, with $\Res$ denoting the set of resonances counted with multiplicities (see~\eqref{e:multiplicity}),
\begin{equation}
  \label{e:weyl-law}
\begin{gathered}
\#\big(\Res\cap \{\Re \omega\in [\alpha'_0,\alpha'_1],\ \Im\omega\in{\textstyle{1\over 2}}[-(\nu_{\max}+\varepsilon)h,-(\nu_{\min}-\varepsilon)h]\}\big)\\
=(2\pi h)^{1-n}(\Vol_\sigma(K\cap p^{-1}([\alpha'_0,\alpha'_1]))+o(1)),
\end{gathered}\end{equation}
as $h\to 0$, for every $[\alpha'_0,\alpha'_1]\subset(\alpha_0,\alpha_1)$
such that
$p^{-1}(\alpha'_j)\cap K$ has zero measure in $K$.
Here $\Vol_\sigma$ denotes the symplectic volume on $K$, defined by
$d\Vol_\sigma=\sigma_S^{n-1}/(n-1)!$. 
\end{theo}
A band structure similar to the one exhibited in Theorems~\ref{t:gaps} and~\ref{t:weyl-law},
with Weyl laws in each band, has been obtained in~\cite{f-t} for a related setting
of Anosov diffeomorphisms, see the discussion below.

\smallsection{The resonance projector}
The key tool in proving Theorems~\ref{t:gaps} and~\ref{t:weyl-law} is a microlocal projector $\Pi$
corresponding to resonances in the band~\eqref{e:weyl-law}.
We construct it as a
\emph{Fourier integral operator} (see~\S\ref{s:prelim-fio}),
associated to the canonical relation $\Lambda^\circ\subset T^*X\times T^*X$
defined as follows.
Let $\mathcal V_\pm\subset T\Gamma_\pm$ be the symplectic complements of $T\Gamma_\pm$
in $T_{\Gamma_\pm}(T^*X)$. For some neighborhoods $\Gamma_\pm^\circ, K^\circ$
of $K\cap p^{-1}([\alpha_0,\alpha_1])$ in $\Gamma_\pm,K$, respectively,
we can define the projections
$\pi_\pm:\Gamma_\pm^\circ\to K^\circ$
along the flow lines of $\mathcal V_\pm$~-- see~\S\ref{s:projections}.
We define (see also~\cite{bfrz})
\begin{equation}
  \label{e:intro-Lambda}
\Lambda^\circ:=\{(\rho_-,\rho_+)\in \Gamma_-^\circ\times\Gamma_+^\circ\mid
\pi_-(\rho_-)=\pi_+(\rho_+)\}.
\end{equation}
Then $\Lambda^\circ$ is a canonical relation, see~\S\ref{s:projections}; it is
pictured on Figure~\ref{f:Lambda}(b).

We now construct an operator $\Pi$ with the following properties
(see Theorem~\ref{t:our-Pi} in~\S\ref{s:construction-1} for details,
including a uniqueness statement):
\begin{enumerate}
\item\label{c:1} $\Pi$ is a compactly supported Fourier integral operator associated to $\Lambda^\circ$;
\item\label{c:2} $\Pi^2=\Pi+\mathcal O(h^\infty)$ microlocally near $K\cap p^{-1}([\alpha_0,\alpha_1])$;
\item\label{c:3} $[P,\Pi]=\mathcal O(h^\infty)$ microlocally near $K\cap p^{-1}([\alpha_0,\alpha_1])$.
\end{enumerate}
Here $P$ is a pseudodifferential operator equal to $\sqrt{P_V}$ microlocally
in $p_V^{-1}([\beta_0^2,\beta_1^2])$ (see~Lemma~\ref{l:resolution} for the general case).
Conditions~\eqref{c:2} and~\eqref{c:3} mimic idempotency and commutation properties of
spectral projectors of self-adjoint operators.

The operator $\Pi$ is constructed iteratively, solving a degenerate transport equation on
each step, with regularity of resulting functions
guaranteed by $r$-normal hyperbolicity. The obtained operator
provides a rich microlocal structure, which makes it possible
to locally relate our situation to the Taylor expansion, ultimately
proving Theorems~\ref{t:gaps} and~\ref{t:weyl-law}.
See~\S\ref{s:ideas} for a more detailed
explanation of the ideas behind the proofs.

\smallsection{Related work}
A particular consequence of Theorem~\ref{t:gaps} is a resonance free strip $\{\Im\omega>-{\nu_{\min}-\varepsilon\over 2}h\}$.
For \emph{normally hyperbolic} trapped sets, such strips (also called spectral gaps) have been obtained by
G\'erard--Sj\"ostrand~\cite{g-sj-gap}
for operators with analytic coefficients and possibly non-smooth $\Gamma_\pm$;
Wunsch--Zworski~\cite{w-z} for sufficiently smooth $\Gamma_\pm$, without specifying the size of the gap;
and Dolgopyat~\cite{dolgopyat}, Liverani~\cite{liverani}, and Tsujii~\cite{tsujii} for contact Anosov flows.
The recent preprint of Nonnenmacher and Zworski~\cite{n-z-new} gives a gap of optimal size for
a variety of normally hyperbolic trapped sets with very weak assumptions on the regularity of $\Gamma_\pm$;
in our special case, the gap of~\cite{n-z-new} coincides with the one given by Theorem~\ref{t:gaps}.
For a related, yet quite different, case of \emph{hyperbolic} trapped sets
(where the flow is hyperbolic in all directions, but no assumptions are made on the regularity of $\Gamma_\pm$ and $K$),
such gaps are known under a pressure condition, see~\cite{n-z-acta} and the references given there.

Upper bounds for the number of resonances in strips near the real axis have been proved in different situations,
both for normally hyperbolic and for hyperbolic trapping,
by Sj\"ostrand~\cite{sj90}, Guillop\'e--Lin--Zworski~\cite{GLZ}, Sj\"ostrand--Zworski~\cite{sj-z},
Non\-nenmacher--Sj\"ostrand--Zworski~\cite{n-sj-z,n-sj-z2}, Faure--Sj\"ostrand~\cite{f-sj},
Datchev--Dyatlov~\cite{fwl}, and Datchev--Dyatlov--Zworski~\cite{ddz}; see~\cite{n-sj-z2} or~\cite{fwl}
for a more detailed overview. The optimal known bounds follow the~\emph{fractal Weyl law},
\begin{equation}
  \label{e:fwl}
\#(\Res\cap \{\Re\omega\in[\alpha_0,\alpha_1],\ |\Im\omega|\leq C_0h\})\leq Ch^{-1-\delta}.
\end{equation}
Here $C_0$ is any fixed number and
$2\delta+2$ is bigger than the upper Minkowski dimension of the trapped set $K$
(inside $T^*X$), or equal to it if $K$ is of pure dimension. In our case,
$\dim K=2n-2$, therefore the Weyl law~\eqref{e:weyl-law} saturates the bound~\eqref{e:fwl}.

Much less is known about lower bounds for hyperbolic or normally hyperbolic trapped sets~--
some special completely integrable cases were studied in particular by G\'erard--Sj\"ostrand~\cite{ge-sj},
S\'a Barreto--Zworski~\cite{sb-z}, and the author~\cite{zeeman},
a lower bound with a smaller power of $h^{-1}$ than~\eqref{e:fwl} for certain
hyperbolic surfaces was proved by Jakobson--Naud~\cite{j-n},
and Weyl laws have been established in some situations in~\cite{sj-z-bands,sj-v,f-t,f-t2,f-t3}~-- see below.
It has been conjectured~\cite[Definition~6.1]{non} that for $C_0$ large enough, a lower
bound matching~\eqref{e:fwl} holds, but no such bound for non-integer $\delta$ has been proved so far.

There also exists a Weyl asymptotic for surfaces with cusps, see M\"uller~\cite{muller}; in this case, the infinite ends
of the manifold are
so narrow that almost all trajectories are trapped, and the Weyl law in strips coincides
with the Weyl law in disks, with a power $h^{-n}$. Other Weyl asymptotics in large regions in the complex
plane have been obtained
by Zworski~\cite{zw-u} for one-dimensional potential scattering and by Sj\"ostrand~\cite{sj-new}
for Schr\"odinger operators with randomly perturbed potentials.

Finally, some situations where resonances form several bands of different depth
were studied in~\cite{sj-z-bands,st-v,sj-v,f-t,f-t2,f-t3}.
Sj\"ostrand--Zworski~\cite{sj-z-bands} showed existence of cubic bands of resonances for strictly convex obstacles,
under a pinching condition on the curvature, with a Weyl law in each band.
Stefanov--Vodev~\cite{st-v} studied the elasticity problem outside of a convex obstacle with Neumann
boundary condition and showed existence of resonances $\mathcal O((\Re\omega)^{-\infty})$ close
to the real line and a gap below this set of resonances; a Weyl law for resonances close to the
real line was proved by Sj\"ostrand--Vodev~\cite{sj-v}.
A case bearing some similarities to the one considered here,
namely contact Anosov diffeomorphisms, has been studied by Faure--Tsujii~\cite{f-t};
their work~\cite{f-t2,f-t3} handles contact Anosov flows~-- the latter can be put
in the framework of~\S\ref{s:framework-assumptions} using the work of Faure--Sj\"ostrand~\cite{f-sj}.

The results of~\cite{f-t,f-t2,f-t3} for the dynamical setting include, under a pinching condition, the band structure of resonances
(with the first band analogous to the one in Theorem~\ref{t:weyl-law}) and Weyl asymptotics
in each band; the trapped set has to be normally hyperbolic, symplectic, and smooth, however the
manifolds $\Gamma_\pm$ need only have H\"older regularity, and no assumption of $r$-normal hyperbolicity is made.
These considerably weaker assumptions on regularity are crucial for Anosov flows and maps, as
one cannot even expect $\Gamma_\pm$ to be $C^2$ in most cases. The lower
regularity is in part handled by conjugating $\mathcal P(\omega)$ by the exponential of
an escape function, similar to the one in~\cite[Lemma~4.2]{ddz}~-- this reduces the analysis
to an $\mathcal O(h^{1/2})$ sized neighborhood of the trapped set. It then suffices to construct
only the principal part of the projector $\Pi$ to first order on the trapped set; such projector
is uniquely defined locally on $K$ (by putting the principal symbol to be equal to 1 on $K$),
without the need for the global construction of~\S\ref{s:construction-1} or the transport
equation~\eqref{e:intro-transport}. The present paper however was motivated by resonance expansions
on perturbations of slowly rotating black holes,
where the more restrictive $r$-normal hyperbolicity assumption is satisfied and it is important
to have an operator $\Pi$ defined to all orders in $h$ and away, as well as on, the trapped set.
Another advantage of such a global operator is the study of resonant states, see~\S\ref{s:resonant-states}.

\section{Outline of the paper}
  \label{s:outline}

In this section, we explain informally the ideas behind the construction
of the projector $\Pi$ and the proofs of Theorems~\ref{t:gaps}
and~\ref{t:weyl-law}, list some directions in which the results
could possibly be improved, and describe the structure of the paper.

\subsection{Ideas of the proofs and concentration of resonant states}
  \label{s:ideas}\quad

\smallsection{Construction of $\Pi$}
An important tool is the model case (see~\S\ref{s:model})
\begin{equation}
  \label{e:intro-model}
X=\mathbb R^n,\quad
\Gamma_-^0=\{x_n=0\},\quad
\Gamma_+^0=\{\xi_n=0\},\quad
\Pi^0 f(x',x_n)=f(x',0).
\end{equation}
Any operator satisfying properties~\eqref{c:1} and~\eqref{c:2} of $\Pi$ listed in
the introduction can be microlocally
conjugated to $\Pi^0$
(see Proposition~\ref{l:reduce-model} and part~2 of Proposition~\ref{l:idempotents}).
However, there is no canonical way of doing this, and to construct $\Pi$ globally,
we need to use property~\eqref{c:3}, which eventually reduces to solving the transport equation
on $\Gamma_\pm$
\begin{equation}
  \label{e:intro-transport}
H_p a=f,\quad
a|_K=0,
\end{equation}
where $f$ is a given smooth function on $\Gamma_\pm$ with $f|_K=0$. The solution
to~\eqref{e:intro-transport} exists and is unique for any normally hyperbolic
trapped set, by representing $a(\rho)$ as an exponentially converging integral of
$f$ over the forward ($\Gamma_-$) or backward ($\Gamma_+$) flow line of $H_p$ starting at $\rho$.
However, to know that $a$ lies in $C^r$ we need $r$-normal hyperbolicity (see Lemma~\ref{l:ode}).
This explains why $r$-normal hyperbolicity, and not just normal hyperbolicity, is needed
to construct the operator $\Pi$.

\smallsection{Proof of Theorem~\ref{t:gaps}}
The proof in~\S\ref{s:resolvent-bounds}
is based on positive commutator arguments, with additional microlocal structure
coming from the projector $\Pi$ and the annihilating operators $\Theta_\pm$ discussed below.
However, here we present a more intuitive (but harder to make rigorous)
argument based on propagation by
$$
U(t)=e^{-itP/h},
$$
which is a Fourier integral operator quantizing the Hamiltonian flow $e^{tH_p}$
(see Proposition~\ref{l:schrodinger}).
Note that we use not the original operator $\mathcal P(\omega)$, but the operator
$P$ constructed in Lemma~\ref{l:resolution}, equal to $\sqrt{P_V}$ for the case~\eqref{e:P-V};
this means that $U(t)$ is the wave, rather than the Schr\"odinger, propagator.
We will only care about the behavior of $U(t)$ near the trapped set; for this purpose, we
introduce a pseudodifferential cutoff $\mathcal X$ microlocalized in a neighborhood of $K$.
For a family of functions $f=f(h)$
whose semiclassical wavefront set
(as discussed in~\S\ref{s:prelim-basics}) is contained in a small neighborhood of $K\cap p^{-1}([\alpha_0,\alpha_1])$,
Theorem~\ref{t:gaps} follows from the following two estimates
(a rigorous analog of~\eqref{e:ee-kernel} is Proposition~\ref{l:estimate-kernel},
and of~\eqref{e:ee-image}, Proposition~\ref{l:estimate-image}): for $t>0$,
\begin{gather}
  \label{e:ee-kernel}
\|\mathcal X U(t)(1-\Pi)f\|_{L^2}\leq (Ch^{-1}e^{-(\nu_{\min}-\varepsilon/2)t}+\mathcal O(h^\infty))\|f\|_{L^2},\\
  \label{e:ee-image}
\begin{aligned}
C^{-1}e^{-{(\nu_{\max}+\varepsilon/2)t\over 2}}\|\mathcal X\Pi f\|_{L^2}-\mathcal O(h^\infty)\|f\|_{L^2}
\leq \|\mathcal X U(t) \Pi f\|_{L^2}
\\\leq Ce^{-{(\nu_{\min}-\varepsilon/2)t\over 2}}\|\mathcal X\Pi f\|_{L^2}+\mathcal O(h^\infty)\|f\|_{L^2}.
\end{aligned}
\end{gather}
The estimates~\eqref{e:ee-kernel} and~\eqref{e:ee-image} are of independent value, as they
give information about the long time behavior of solutions to the wave equation, resembling
resonance expansions of linear waves; an application to black holes
is given in~\cite{thesis}. Note however that these estimates are nontrivial
only when $t=\mathcal O(\log(1/h))$, because of the $\mathcal O(h^\infty)$ error term.

The resonance free region~\eqref{e:omega-gaps} of Theorem~\ref{t:gaps}
is derived from here as follows. Assume that $\omega$ is a resonance in~\eqref{e:omega-gaps}.
Then there exists a \emph{resonant state}, namely a function $u\in\mathcal H_1$ such that $\mathcal P(\omega)u=0$
and $\|u\|_{\mathcal H_1}\sim 1$.
We formally have $U(t)u=e^{-it\omega/h}u$. Also, $u$ is microlocalized on the outgoing tail
$\Gamma_+$, which is propagated by the flow $e^{tH_p}$ towards infinity; this means that
if $f:=\mathcal X_1 u$ for a suitably chosen pseudodifferential cutoff $\mathcal X_1$, then 
$\Pi u=\Pi f+\mathcal O(h^\infty)$ and for $t>0$,
$$
U(t)f=e^{-it\omega/h}f+\mathcal O(h^\infty)\quad\text{microlocally near }\WFh(\mathcal X).
$$
Since $\Pi$ commutes with $P$ modulo $\mathcal O(h^\infty)$, it also commutes with $U(t)$,
which gives
$$
\begin{gathered}
\mathcal X U(t)(1-\Pi)f=e^{-it\omega/h}\mathcal X(1-\Pi)f+\mathcal O(h^\infty),\\
\mathcal X U(t)\Pi f=e^{-it\omega/h}\mathcal X\Pi f+\mathcal O(h^\infty).
\end{gathered}
$$
Since $\Im\omega\geq-(\nu_{\min}-\varepsilon)h$, we
take $t=N\log(1/h)$ for arbitrarily large constant $N$ in~\eqref{e:ee-kernel} to get
$\|\mathcal X (1-\Pi)f\|_{L^2}=\mathcal O(h^\infty)$. Since
$\Im\omega\not\in (-(\nu_{\max}+\varepsilon)h/2,-(\nu_{\min}-\varepsilon)h/2)$, by~\eqref{e:ee-image}
we get $\|\mathcal X \Pi f\|_{L^2}=\mathcal O(h^\infty)$. Together, they give $\|\mathcal X f\|_{L^2}=\mathcal O(h^\infty)$,
implying by standard outgoing estimates (see Lemma~\ref{l:smart-bound}) that $\|u\|_{\mathcal H_1}=\mathcal O(h^\infty)$,
a contradiction.

We now give an intuitive explanation for~\eqref{e:ee-kernel} and~\eqref{e:ee-image}. We start by considering the
model case~\eqref{e:intro-model}, with the pseudodifferential cutoff $\mathcal X$ replaced by
the multiplication operator by some $\chi\in C_0^\infty(\mathbb R^n)$.
For the operator
$P$, we consider the model (somewhat inappropriate since the actual Hamiltonian vector field $H_p$
is typically nonvanishing on $K$, contrary to the model case, but reflecting the nature of the
flow in the transverse directions) $P=x_n\cdot hD_{x_n}-ih/2$; here the term $-ih/2$ makes $P$ symmetric.
We then have in the model case, $p=x_n\xi_n$, $e^{tH_p}(x,\xi)=(x',e^tx_n,\xi',e^{-t}\xi_n)$,
$\nu_{\min}=\nu_{\max}=1$, and
$$
U(t)f(x',x_n)=e^{-t/2}f(x',e^{-t}x_n).
$$
Then~\eqref{e:ee-kernel} (in fact, a better estimate with $e^{-3t/2}$ in place of $e^{-t}$~--
see the possible improvements subsection below) follows by Taylor expansion at $x_n=0$.
More precisely, we use the following form of this expansion: for $f\in C_0^\infty(\mathbb R^n)$,
\begin{equation}
  \label{e:intro-Xi}
(1-\Pi^0) f=x_n\cdot g,\quad
g(x',x_n):={f(x',x_n)-f(x',0)\over x_n},
\end{equation}
and one can show that $\|g\|_{L^2}\leq Ch^{-1}\|f\|_{H^1_h}$, the factor $h^{-1}$
coming from taking one nonsemiclassical derivative to obtain $g$ from $f$
(see Lemma~\ref{l:Xi-0}). Then
$\chi U(t)(1-\Pi^0)f=\chi U(t)x_n U(-t)U(t)g$, where (by a special case of Egorov's theorem
following by direct computation)
$\chi U(t)x_n U(-t)$ is a multiplication operator by
\begin{equation}
  \label{e:zz-kernel}
\chi U(t)x_n U(-t)=\chi(x)e^{-t}x_n=\mathcal O(e^{-t});
\end{equation}
this shows that $\|\chi U(t)(1-\Pi^0)f\|_{L^2}\leq Ce^{-t}\|g\|_{L^2}
\leq Ch^{-1}e^{-t}\|f\|_{H^1_h}$ and~\eqref{e:ee-kernel} follows.

To show~\eqref{e:ee-image} in the model case, we start with the identity
$$
\|\chi U(t)\Pi^0 f\|_{L^2}=\|\chi_t\Pi^0 f\|_{L^2},\quad
\chi_t:=U(-t)\chi U(t).
$$
If $\chi\in C_0^\infty(\mathbb R^n)$, then $\chi_t(x)=\chi(x',e^tx_n)$ has shrinking
support as $t\to\infty$. To compare $\|\chi_t\Pi^0 f\|_{L^2}$
to $\|\chi\Pi^0 f\|_{L^2}$, we use the following fact:
\begin{equation}
  \label{e:zz-image}
hD_{x_n}\Pi^0 f=0.
\end{equation}
This implies that for each $a(x)\in C_0^\infty(\mathbb R^n)$, the
inner product $\langle a\Pi^0 f,\Pi^0 f\rangle$ depends
only on the function $b(x')=\int_{\mathbb R}a(x',x_n)\,dx_n$;
writing $\|\chi\Pi^0 f\|_{L^2}^2$ and $\|\chi_t\Pi^0 f\|_{L^2}^2$
as inner products, we get
$\|\chi_t\Pi^0 f\|_{L^2}^2=e^{-t}\|\chi \Pi^0 f\|_{L^2}^2$
and~\eqref{e:ee-image} follows.

The proofs of~\eqref{e:ee-kernel} and~\eqref{e:ee-image} in the general case work as in the model case,
once we find appropriate replacements for differential operators
$x_n$ and $hD_{x_n}$ in~\eqref{e:intro-Xi} and~\eqref{e:zz-image}. It turns out
that one needs to take pseudodifferential operators $\Theta_\pm$ solving,
microlocally near $K\cap p^{-1}([\alpha_0,\alpha_1])$,
\begin{equation}
  \label{e:intro-ideals}
\Pi\Theta_-=\mathcal O(h^\infty),\quad
\Theta_+\Pi=\mathcal O(h^\infty),
\end{equation}
then $\Theta_-$ is a replacement for $x_n$ and $\Theta_+$, for $hD_{x_n}$.
Note that $\Theta_\pm$ are not unique, in fact solutions to~\eqref{e:intro-ideals}
form one-sided ideals in the algebra of pseudodifferential operators~-- see
\S\S\ref{s:ideals} and~\ref{s:construction-2}.
The principal symbols of $\Theta_\pm$ are defining functions of $\Gamma_\pm$.

\smallsection{Concentration of resonant states}
As a byproduct of the discussion above, we obtain new information about microlocal
concentration of resonant states, that is, functions $u\in\mathcal H_1$ such that
$\mathcal P(\omega)u=0$ and $\|u\|_{\mathcal H_1}\sim 1$. It is well-known
(see for example~\cite[Theorem~4]{n-z-acta}) that the wavefront set of $u$
is contained in $\Gamma_+\cap p^{-1}(\Re\omega)$. The new information we obtain is that
if $\omega$ is a resonance in the band given by Theorem~\ref{t:weyl-law}
(that is, $\Im\omega>-(\nu_{\min}-\varepsilon)h$), then
by~\eqref{e:ee-kernel}, $u=\Pi u+\mathcal O(h^\infty)$ microlocally near $K$.
Then by~\eqref{e:intro-ideals}, $\Theta_+u=\mathcal O(h^\infty)$
near $K$, that is, $u$ solves a pseudodifferential equation; note that
the Hamiltonian flow lines of the principal symbol of $\Theta_+$ are transverse
to the trapped set. This implies in particular that any corresponding semiclassical defect
measure is determined uniquely by a measure on the trapped
set which is conditionally invariant under $H_p$, similarly
to the damped wave equation. See Theorem~\ref{t:resonant-states}
in~\S\ref{s:resonant-states} for details.

\smallsection{Proof of Theorem~\ref{t:weyl-law}}
We start with constructing a well-posed Grushin problem,
representing resonances as zeroes of a certain Fredholm determinant $F(\omega)$.
Using complex analysis (essentially the argument principle), we reduce
counting resonances to computing a contour integral of the logarithmic
derivative $F'(\omega)/F(\omega)$, which, taking
$\nu_-=-(\nu_{\max}+\varepsilon)/2$, $\nu_+=-(\nu_{\min}-\varepsilon)/2$,
is similar to (see~\S\ref{s:trace}
for the actual expression) 
$$
{1\over 2\pi i}(\mathcal I_--\mathcal I_+),\quad
\mathcal I_\pm:=\int_{\Im\omega=h\nu_\pm} \tilde\chi(\omega)\Tr(\Pi \mathcal R(\omega))\,d\omega
$$
for some cutoff function $\tilde\chi(\omega)$. The integration is over the region where Theorem~\ref{t:gaps}
gives polynomial bounds on the resolvent $\mathcal R(\omega)$, and we can use the methods developed
for the proof of this theorem to evaluate both integrals, yielding Theorem~\ref{t:weyl-law}.
An important additional tool, explaining in particular why the two integrals do not cancel each other,
is microlocal analysis in the spectral parameter $\omega$, or equivalently a study of the
essential support of the Fourier transform of $\Pi\mathcal R(\omega)$ in $\omega$~-- see~\S\S\ref{s:estimate-spectral}
and~\ref{s:trace}.

\subsection{Possible improvements}
  \label{s:improvements}

First of all, it would be interesting to see if one could construct further bands
of resonances, lying below the one in Theorem~\ref{t:weyl-law}. One expects these
bands to have the form
$$
\{\Im\omega\in [-(k+1/2)(\nu_{\max}+\varepsilon)h,-(k+1/2)(\nu_{\min}-\varepsilon)h]\},\quad
k\in \mathbb Z,\ k\geq 0,
$$
and to have a Weyl law in the $k$-th band under the pinching condition $(k+1/2)\nu_{\max}<(k+3/2)\nu_{\min}$.
Note that the presence of the second band of resonances improves the size of the second
resonance free strip in Theorem~\ref{t:gaps} and gives a weaker  pinching condition $\nu_{\max}<3\nu_{\min}$
for the Weyl law in the first band. The proofs are expected to work similarly to the present paper, if one
constructs a family of operators $\Pi_0=\Pi,\Pi_1,\dots,\Pi_k$ such that
$\Pi_j$ is $h^{-j}$ times a Fourier integral operator associated to $\Lambda^\circ$,
$\Pi_j\Pi_k=\mathcal O(h^\infty)$, and $[P,\Pi_j]=\mathcal O(h^\infty)$
(microlocally near $K\cap p^{-1}([\alpha_0,\alpha_1])$). However, the method
of~\S\ref{s:construction-1} does not apply directly to construct $\Pi_k$ for
$k>0$, since one cannot conjugate all $\Pi_j$ to the model case,
which is the base of the crucial Proposition~\ref{l:idempotents}.

Another direction would be to consider the case when the
operator $P$ is quantum completely integrable on the trapped set
(a notion that needs to be made precise), and derive a quantization condition for resonances
like the one for the special case of black holes~\cite{sb-z,zeeman}.
The author also believes that the results of the present paper should
be adaptable to the situation when $\Gamma_\pm$ have codimension higher
than~1, which makes it possible to revisit the distribution of
resonances generated by one closed hyperbolic trajectory, studied in~\cite{ge-sj}.

An interesting special case lying on the intersection of the current work and~\cite{f-t,f-t2,f-t3}
is given by geodesic flows on compact manifolds of constant negative curvature; the corresponding
manifolds $\Gamma_\pm$ and $K$ are smooth in this situation. While $r$-normal hyperbolicity
does not hold (in fact, $\mu_{\max}=\nu_{\min}=\nu_{\max}$), the rigid algebraic structure
of hyperbolic quotients suggests that one could still look for the
projector $\Pi$ as a (smooth) Fourier integral operator~-- in terms of the construction of~\S\ref{s:construction-1},
the transport equation~\eqref{e:intro-transport}, while not yielding a smooth solution
for an arbitrary choice of the right-hand side $f$, will have a smooth solution for the
specific functions $f$ arising in the construction.

Finally, a natural question is improving the $o(1)$
remainder in the Weyl law~\eqref{e:weyl-law}. Obtaining an $\mathcal O(h^\delta)$
remainder for $\delta<1$ does not seem to require conceptual changes to the
microlocal structure of the argument; however, for the $\mathcal O(h)$ remainder
of H\"ormander~\cite{hohoho} or the $o(h)$ remainder of Duistermaat--Guillemin~\cite{du-gu}, one would
need a finer analysis of the interaction of the operator $\Pi$ with the Schr\"odinger
propagator, and more assumptions on the flow on the trapped set might be needed.
Moreover, the complex analysis argument of~\S\ref{s:weyl-law} does
not work in the case of an $\mathcal O(h)$ remainder; a reasonable replacement would be
to adapt to the considered case the work of Sj\"ostrand~\cite{sj-dwe} on the damped wave equation.

\subsection{Structure of the paper}
  \label{s:structure}
\begin{itemize}
\item In~\S\ref{s:prelim}, we review the tools we need from semiclassical analysis.
\item In~\S\ref{s:framework}, we present a framework which makes it possible
to handle resonances and the spatial infinity in an abstract fashion.
The assumptions we make are listed in~\S\ref{s:framework-assumptions}, followed
by some useful lemmas (\S\ref{s:framework-corollaries}) and applications to
Schr\"odinger operators (\S\ref{s:framework-schrodinger}) and even asymptotically hyperbolic manifolds (\S\ref{s:framework-ah}).
\item In~\S\ref{s:nh}, we study $r$-normally hyperbolic trapped sets, stating
the dynamical assumptions (\S\ref{s:dynamics}), discussing their stability under perturbations
(\S\ref{s:stability}), and deriving some corollaries (\S\S\ref{s:phi-pm}--\ref{s:transport}).
\item In~\S\ref{s:calculus}, we study in detail Fourier integral operators associated to $\Lambda^\circ$,
and in particular properties of operators solving $\Pi^2=\Pi+\mathcal O(h^\infty)$.
\item In~\S\ref{s:global-construction}, we construct the projector $\Pi$ and the annihilating operators $\Theta_\pm$.
\item In~\S\ref{s:resolvent-bounds}, we prove Theorem~\ref{t:gaps}, establish microlocal estimates on
the resolvent, and study the microlocal concentration of resonant states (\S\ref{s:resonant-states}).
\item In~\S\ref{s:grushin}, we formulate a well-posed Grushin problem for $\mathcal P(\omega)$,
representing resonances as zeroes of a certain Fredholm determinant.
\item In~\S\ref{s:trace}, we prove a trace formula for $\mathcal R(\omega)$
microlocally on the image of $\Pi$.
\item In~\S\ref{s:weyl-law}, we prove the Weyl asymptotic for resonances (Theorem~\ref{t:weyl-law}).
\item In Appendix~\ref{s:example}, we provide an example of an asymptotically hyperbolic manifold
satisfying the dynamical assumptions of~\S\ref{s:dynamics}.
\end{itemize}

\section{Semiclassical preliminaries}
\label{s:prelim}

In this section, we review semiclassical pseudodifferential operators,
wavefront sets, and Fourier integral operators;
the reader is directed to~\cite{e-z,d-sj} for a detailed treatment
and~\cite{ho3,ho4,gr-sj} for the closely related microlocal case.

\subsection{Pseudodifferential operators and microlocalization}
\label{s:prelim-basics}

Let $X$ be a manifold without boundary.
Following~\cite[\S 9.3 and~14.2]{e-z}, we consider the symbol classes
$S^k(T^*X)$, $k\in\mathbb R$, consisting of smooth functions $a$ on the cotangent bundle
$T^*X$ satisfying in local coordinates
$$
\sup_h\sup_{x\in K}|\partial^\alpha_x\partial^\beta_\xi a(x,\xi;h)|\leq C_{\alpha\beta K}\langle\xi\rangle^{k-|\beta|},
$$
for each multiindices $\alpha,\beta$ and each compact set $K\subset X$.
The corresponding class of semiclassical pseudodifferential operators is denoted $\Psi^k(X)$.
The residual symbol class $h^\infty S^{-\infty}$ consists of symbols decaying
rapidly in $h$ and $\xi$ over compact subsets of $X$; the operators
in the corresponding class $h^\infty\Psi^{-\infty}$ have Schwartz kernels in $h^\infty C^\infty(X\times X)$.
Operators in $\Psi^k$ are bounded, uniformly in $h$, between the semiclassical
Sobolev spaces $H^s_{h,\comp}(X)\to H^{s-k}_{h,\loc}(X)$, see~\cite[(14.2.3)]{e-z} for the
definition of the latter.

Note that for noncompact $X$, we impose no restrictions on the behavior of symbols as $x\to\infty$.
Accordingly, we cannot control the behavior of operators in $\Psi^k(X)$ near spatial infinity;
in fact, a priori we only require them to act $C_0^\infty(X)\to C^\infty(X)$
and on the spaces of distributions $\mathcal E'(X)\to \mathcal D'(X)$.
However, each $A\in\Psi^k(X)$ can be written as the sum of an $h^\infty\Psi^{-\infty}$
remainder and an operator properly supported uniformly in $h$~-- see for example~\cite[Proposition~18.1.22]{ho3}.
Properly supported pseudodifferential
operators act $C_0^\infty\to C_0^\infty$ and $C^\infty\to C^\infty$ and therefore can be multiplied
with each other, giving an algebra structure on the whole $\Psi^k$, modulo $h^\infty\Psi^{-\infty}$.

To study the behavior of symbols near fiber infinity, we use the fiber-radial
compactified cotangent bundle $\overline T^*X$, a manifold with boundary
whose interior is diffeomorphic to $T^*X$ and whose boundary $\partial\overline T^*X$
is diffeomorphic to the cosphere bundle over $X$~-- see for example~\cite[\S 2.1]{v1}.
We will restrict ourselves to the space of \emph{classical symbols}, i.e. those having an asymptotic expansion
$$
a(x,\xi;h)\sim\sum_{j\geq 0}h^ja_j(x,\xi),
$$
with $a_j\in S^{k-j}$ classical in the sense that $\langle\xi\rangle^{j-k}a_j$
extends to a smooth function on $\overline T^*X$.
The principal symbol
$\sigma(A):=a_0\in S^k$ of an operator is defined independently of quantization.
We say that $A\in\Psi^k$ is elliptic at some $(x,\xi)\in\overline T^*X$ if
$\langle \xi\rangle^{-k}\sigma(A)$ does not vanish at $(x,\xi)$.

Another invariant object associated to $A\in\Psi^k(X)$ is its wavefront set
$\WFh(A)$, which is a closed subset of $\overline T^*X$; a point
$(x,\xi)\in \overline T^*X$ does not lie in $\WFh(A)$ if and only if
there exists a neighborhood $U$ of $(x,\xi)$ in $\overline T^*X$ such that
the full symbol of $A$ (in any quantization) is in $h^\infty S^{-\infty}$ in
this neighborhood. Note that $\WFh(A)=\emptyset$ if and only if $A=\mathcal O(h^\infty)_{\Psi^{-\infty}}$.
We say that $A_1=A_2+\mathcal O(h^\infty)$ \emph{microlocally} in some $U\subset\overline T^*X$, if $\WFh(A-B)\cap U=\emptyset$.

We denote by $\Psi^{\comp}(X)$ the space of all operators $A\in\Psi^0(X)$ such that
$\WFh(A)$ is a compact subset of $T^*X$, in particular not intersecting the
fiber infinity $\partial \overline T^*X$. Note that $\Psi^{\comp}(X)\subset \Psi^k(X)$
for all $k\in \mathbb R$.

\smallsection{Tempered distributions and operators}
Let $u=u(h)$ be an $h$-dependent family of distributions in $\mathcal D'(X)$.
We say that $u$ is $h$-tempered
(or polynomially bounded), if for each $\chi\in C_0^\infty(X)$, there exists
$N$ such that $\|\chi u\|_{H^{-N}_h}=\mathcal O(h^{-N})$. The class of $h$-tempered distributions
is closed under properly supported pseudodifferential operators.
For an $h$-tempered $u$, define the wavefront set
$\WFh(u)$, a closed subset of $\overline T^*X$, as follows:
$(x,\xi)\in \overline T^*X$ does not lie in $\WFh(u)$ if and only if there exists
a neighborhood $U$ of $(x,\xi)$ in $\overline T^*X$ such that for each
properly supported $A\in\Psi^0(X)$ with $\WFh(A)\subset U$, we have
$Au=\mathcal O(h^\infty)_{C^\infty}$. We have $\WFh(u)=\emptyset$ if and only
if $u=\mathcal O(h^\infty)_{C^\infty}$. We say that $u=v+\mathcal O(h^\infty)$ microlocally on some
$U\subset \overline T^*X$ if $\WFh(u-v)\cap U=\emptyset$.

Let $X_1$ and $X_2$ be two manifolds.
An operator $B:C_0^\infty(X_1)\to\mathcal D'(X_2)$ is identified with its Schwartz kernel
$\mathcal K_B(y,x)\in\mathcal D'(X_2\times X_1)$:
\begin{equation}
  \label{e:op-kernel}
Bf(y)=\int_{X_1} \mathcal K_B(y,x)u(x)\,dx,\quad
u\in \mathcal C_0^\infty(X_1).
\end{equation}
Here we assume that $X_1$ is equipped with some smooth density $dx$; later, we will also
assume that densities on our manifolds are specified when talking about adjoints.

We say that $B$ is $h$-tempered if $\mathcal K_B$ is, and define the wavefront set of $B$ as
\begin{equation}
  \label{e:op-wf}
\WFh(B):=\{(x,\xi,y,\eta)\in \overline T^*(X_1\times X_2)\mid (y,\eta,x,-\xi)\in\WFh(K_B)\}.
\end{equation}
If $B\in\Psi^k(X)$, then the wavefront set of $B$ as an $h$-tempered operator
is equal to its wavefront set as a pseudodifferential operator, under
the diagonal embedding $\overline T^*X\to \overline T^*(X\times X)$.
  
\subsection{Lagrangian distributions and Fourier integral operators}
  \label{s:prelim-fio}

We now review the theory of Lagrangian distributions; for details, the reader is directed
to~\cite[Chapters~10--11]{e-z}, \cite[Chapter~6]{gu-st1}, or~\cite[\S 2.3]{svn},
and to~\cite[Chapter~25]{ho4} or~\cite[Chapters~10--11]{gr-sj} for the closely
related microlocal setting. Here, we only present the relatively simple local
part of the theory; geometric constructions of invariant symbols will be done
by hand when needed, without studying the structure of the bundles obtained
(see~\S\ref{s:general}). For a more complete discussion, see for example~\cite[\S 3]{qeefun}.

A semiclassical Lagrangian distribution locally takes the form
\begin{equation}
  \label{e:lagrangian}
u(x;h)=(2\pi h)^{-m/2} \int_{\mathbb R^m} e^{{i\over h}\Phi(x,\theta)}a(x,\theta;h)\,d\theta.
\end{equation}
Here $\Phi$ is a nondegenerate phase function, i.e. a real-valued function defined
on an open subset of $X\times \mathbb R^m$, for some $m$, such that
the differentials $d(\partial_{\theta_1}\Phi),\dots,d(\partial_{\theta_m}\Phi)$ are linearly
independent on the critical set
$$
\mathcal C_\Phi := \{(x,\theta)\mid \partial_\theta \Phi(x,\theta)=0\}.
$$
The amplitude $a(x,\theta;h)$ is a classical symbol (that is,
having an asymptotic expansion in nonnegative integer powers of $h$
as $h\to 0$) compactly supported
inside the domain of $\Phi$. The resulting function $u(x;h)$
is smooth, compactly supported, $h$-tempered, and\begin{equation}
  \label{e:lagrangian-wf}
\WFh(u)\subset \{(x,\partial_x\Phi(x,\theta))\mid (x,\theta)\in \mathcal C_\Phi\cap\supp a\}.
\end{equation}
We say that $\Phi$ generates the (immersed, and we shrink the domain
of $\Phi$ to make it embedded) Lagrangian submanifold
$$
\Lambda_\Phi:=\{(x,\partial_x\Phi(x,\theta))\mid (x,\theta)\in \mathcal C_\Phi\};
$$
note that $\WFh(u)\subset\Lambda_\Phi$. Moreover, if we restrict
$\Phi$ to $\mathcal C_\Phi$ and pull it back to $\Lambda_\Phi$,
then $d\Phi$ equals the canonical 1-form $\xi\,dx$ on $\Lambda_\Phi$.

In general, assume that $\Lambda$ is an embedded Lagrangian
submanifold of $T^*X$ which is moreover \emph{exact} in the sense that
the canonical form $\xi\,dx$ is exact on $\Lambda$; we fix an
\emph{antiderivative} on $\Lambda$, namely a function $F$ such that
$\xi\,dx=dF$ on $\Lambda$. (This is somewhat similar to the notion of
Legendre distributions, see~\cite[\S 11]{m-z}.) Then we say that a
compactly supported $h$-tempered family of distributions $u$ is a
(compactly microlocalized) Lagrangian distribution associated to
$\Lambda$, if $u$ can be written as a finite sum of
expressions~\eqref{e:lagrangian}, with phase functions $\Phi_j$
generating open subsets of $\Lambda$, plus an $\mathcal
O(h^\infty)_{C_0^\infty}$ remainder, where $\Phi_j$ are normalized (by
adding a constant) so that the pull-back to $\Lambda$ of the
restriction of $\Phi_j$ to $\mathcal C_{\Phi_j}$ equals $F$. (Without
such normalization, passing from one phase function to the other
produces a factor $e^{is\over h}$ for some constant $s$, which does
not preserve the class of classical symbols~-- this is an additional
complication of the theory compared to the nonsemiclassical case.)
Denote by $I_{\comp}(\Lambda)$ the class of all Lagrangian
distributions associated to $\Lambda$.  For $u\in I_{\comp}(\Lambda)$,
we have $\WFh(u)\subset\Lambda$; in particular, $\WFh(u)$ does not
intersect the fiber infinity $\partial\overline T^*X$.

If now $X_1,X_2$ are two manifolds of dimensions
$n_1,n_2$ respectively, and $\Lambda\subset T^*X_1\times T^*X_2$ is an exact canonical
relation (with some fixed antiderivative), then an operator $B:C^\infty(X_1)\to C_0^\infty(X_2)$ is called a (compactly microlocalized)
Fourier integral operator associated to $\Lambda$, if its Schwartz kernel $\mathcal K_B(y,x)$ is
$h^{-(n_1+n_2)/4}$ times a Lagrangian distribution associated to
$$
\{(y,\eta,x,-\xi)\in T^*(X_1\times X_2)\mid (x,\xi,y,\eta)\in\Lambda\}.
$$
We write $B\in I_{\comp}(\Lambda)$; note that $\WFh(B)\subset\Lambda$.
A particular case is when $\Lambda$ is the graph of a
canonical transformation $\varkappa:U_1\to U_2$, with $U_j$ open subsets in $T^*X_j$.
Operators associated to canonical transformations (but not general relations!) are
bounded $H^s_h\to H^{s'}_h$ uniformly in $h$, for each $s,s'$.

Compactly microlocalized Fourier integral operators associated to the identity
transformation are exactly compactly supported pseudodifferential operators in $\Psi^{\comp}(X)$.
Another example of Fourier integral operators
is given by Schr\"odinger propagators, see for instance~\cite[Theorem~10.4]{e-z}%
\footnote{\cite[Theorem~10.4]{e-z} is stated for self-adjoint $P$, rather
than operators with real-valued principal symbols; however, the proof works
similarly in the latter case, with the transport equation acquiring an additional
zeroth order term due to the subprincipal part of $P$.}
or~\cite[Proposition~3.8]{qeefun}:
\begin{prop}
  \label{l:schrodinger}
Assume that $P\in\Psi^{\comp}(X)$ is compactly supported,
$\WFh(P)$ is contained in some compact subset $V\subset T^*X$,
and $p=\sigma(P)$ is real-valued. Then for $t\in\mathbb R$ bounded by any fixed constant,
the operator $e^{-itP/h}:L^2(X)\to L^2(X)$ is the sum of the identity and
a compactly supported operator microlocalized in $V\times V$. Moreover, for each compactly
supported $A\in\Psi^{\comp}(X)$, $Ae^{-itP/h}$ and $e^{-itP/h}A$ are
smooth families of Fourier integral operators associated to the Hamiltonian
flow $e^{tH_p}:T^*X\to T^*X$.
\end{prop}
Here we put the antiderivative $F$ for the identity transformation to
equal zero, and extend it to the antiderivative $F_t$ on the graph of $e^{tH_p}$
by putting
$$
F_t(\gamma(0),\gamma(t)):=tp(\gamma(0))-\int_{\gamma([0,t])}\xi\,dx
$$
for each flow line $\gamma$ of $H_p$.
The corresponding phase function is produced by a solution to the Hamilton--Jacobi equation~\cite[Lemma~10.5]{e-z}.

We finally discuss products of Fourier integral operators. Assume that
$B_j\in I_{\comp}(\Lambda_j)$, $j=1,2$, where $\Lambda_1\subset T^*X_1\times T^*X_2$ and
$\Lambda_2\subset T^*X_2\times T^*X_3$ are exact canonical relations. Assume moreover that
$\Lambda_1,\Lambda_2$ satisfy the following transversality assumption:
the manifolds $\Lambda_1\times\Lambda_2$ and $T^*X_1\times \Delta(T^*X_2)\times T^*X_3$,
where $\Delta(T^*X_2)\subset T^*X_2\times T^*X_2$ is the diagonal, intersect
transversely inside $T^*X_1\times T^*X_2\times T^*X_2\times T^*X_3$, and their
intersection projects diffeomorphically onto $T^*X_1\times T^*X_3$.
Then $B_2B_1\in I_{\comp}(\Lambda_2\circ\Lambda_1)$, where 
\begin{equation}
  \label{e:comp-fio}
\Lambda_2\circ\Lambda_1:=\{(\rho_1,\rho_3)\mid \exists \rho_2\in T^*X_2:
(\rho_1,\rho_2)\in\Lambda_1,\ (\rho_2,\rho_3)\in\Lambda_2\},
\end{equation}
and, if $F_j$ is the antiderivative on $\Lambda_j$, then
$F_1(\rho_1,\rho_2)+F_2(\rho_2,\rho_3)$ is the antiderivative on $\Lambda_2\circ\Lambda_1$.
See for example~\cite[Theorem~25.2.3]{ho4} or~\cite[Theorem~11.12]{gr-sj} for the closely related microlocal case,
which is adapted directly to the semiclassical situation.

The transversality condition is
always satisfied when at least one of the $\Lambda_j$ is the graph of a canonical transformation.
In particular, one can always multiply a pseudodifferential
operator by a Fourier integral operator, and obtain a Fourier integral operator associated
to the same canonical relation.

\subsection{Basic estimates}

In this section, we review some standard semiclassical estimates, parametrices,
and microlocalization statements.

Throughout the section, we
assume that $k,s\in\mathbb R$, $P,Q\in\Psi^k(X)$ are properly supported and $u,f$ are $h$-tempered distributions
on $X$, in the sense of~\S\ref{s:prelim-basics}.

We start with the elliptic estimate, see for instance~\cite[Proposition~2.2]{zeeman}:
\begin{prop} (Elliptic estimate)
  \label{l:elliptic}
Assume that $Pu=f$. Then:

1. If $A,B\in\Psi^0(X)$ are compactly supported and $P,B$ are elliptic on $\WFh(A)$, then
\begin{equation}
  \label{e:ell-est}
\|Au\|_{H^s_h}\leq C\|Bf\|_{H^{s-k}_h}+\mathcal O(h^\infty).
\end{equation}

2. We have
\begin{equation}
  \label{e:ell-wf}
\WFh(u)\subset \WFh(f)\cup \{\langle\xi\rangle^{-k}\sigma(P)=0\}.
\end{equation}
\end{prop}
Proposition~\ref{l:elliptic} is typically proved using the following fact, which
is of independent interest:
\begin{prop} (Elliptic parametrix)
  \label{l:eparametrix}
If $V\subset \overline T^*X$ is compact and $P$ is elliptic on $V$, then
there exists a compactly supported operator $P'\in\Psi^{-k}(X)$ such that
$PP'=1+\mathcal O(h^\infty),P'P=1+\mathcal O(h^\infty)$ microlocally
near $V$. Moreover, $\sigma(P')=\sigma(P)^{-1}$ near $V$.
\end{prop}
We next give a version of propagation of singularities which allows
for a complex absorbing operator~$Q$, see for instance~\cite[\S2.3]{v1}:
\begin{prop} (Propagation of singularities)
  \label{l:microhyperbolic}
Assume that $\sigma(P)$ is real-valued, $\sigma(Q)\geq 0$, and $(P\pm iQ)u=f$. Then:

1. If $A_1,A_2,B\in\Psi^0(X)$ are compactly supported
and for each flow line $\gamma(t)$ of the Hamiltonian field $\pm\langle\xi\rangle^{1-k}H_{\sigma(P)}$ such that $\gamma(0)\in\WFh(A_1)$,
there exists $t\geq 0$ such that $A_2$ is elliptic at $\gamma(t)$ and 
$B$ is elliptic on the segment $\gamma([0,t])$, then
\begin{equation}
  \label{e:mhp-est}
\|A_1u\|_{H^s_h}\leq C\|A_2u\|_{H^s_h}+Ch^{-1}\|Bf\|_{H^{s-k+1}_h}+\mathcal O(h^\infty).
\end{equation}

2. If $\gamma(t)$, $0\leq t\leq T$, is a flow line of $\pm\langle\xi\rangle^{1-k}H_{\sigma(P)}$, then
$$
\gamma([0,T])\cap \WFh(f)=\emptyset,\
\gamma(T)\not\in\WFh(u)\ \Longrightarrow\ \gamma(0)\not\in\WFh(u).
$$
\end{prop}
For $Q=0$, Proposition~\ref{l:microhyperbolic} can be viewed as a microlocal version of uniqueness of solutions
to the Cauchy problem for hyperbolic equations; a corresponding microlocal existence fact
is given by
\begin{prop} (Hyperbolic parametrix)
  \label{l:hparametrix}
Assume that $\sigma(P)$ is real-valued, $\WFh(f)\subset T^*X$ is compact,
$U,V\subset T^*X$ are compactly contained
open sets, and for each flow line $\gamma(t)$ of the Hamiltonian
field $H_{\sigma(P)}$ such that $\gamma(0)\in\WFh(f)$, there exists $t\in\mathbb R$
such that $\gamma(t)\in U$ and $\gamma(s)\in V$ for all $s$ between $0$ and $t$.

Then there exists an $h$-tempered family $v(h)\in C_0^\infty(X)$ such that
$\WFh(v)\subset V$ and
$$
\|v\|_{L^2}\leq Ch^{-1}\|f\|_{L^2},\quad
\|Pv\|_{L^2}\leq C\|f\|_{L^2},\quad
\WFh(Pv-f)\subset U.
$$
\end{prop}
\begin{proof}
By applying a microlocal partition of unity to $f$, we may assume that
there exists $T>0$ (the case $T<0$ is considered similarly and the case $T=0$ is trivial
by putting $v=0$)
such that for each flow line $\gamma(t)$ of $H_{\sigma(P)}$
such that $\gamma(0)\in\WFh(f)$, we have $\gamma(T)\in U$ and $\gamma([0,T])\in V$.
Take $\varepsilon\in (0,T)$ such that $\gamma([T-\varepsilon,T])\subset U$
for each such $\gamma$.
Since $V$ is compactly contained in $T^*X$, we may assume that $P$ is compactly
supported and $P\in\Psi^{\comp}(X)$.
We then take $\chi\in C_0^\infty(-\infty,T)$ such that
$\chi=1$ near $[0,T-\varepsilon]$ and put
$$
v:={i\over h}\int_0^T \chi(t)e^{-itP/h}f\,dt.
$$
Then $\|v\|_{L^2}\leq Ch^{-1}\|f\|_{L^2}$ and
$\WFh(v)\subset V$ by Proposition~\ref{l:schrodinger}. Integrating by parts, we compute
$$
Pv=-\int_0^T \chi(t)\partial_t e^{-itP/h}f\,dt=f+\int_0^T(\partial_t\chi(t))e^{-itP/h}f\,dt;
$$
therefore, $\|Pv\|_{L^2}\leq C\|f\|_{L^2}$ and by Proposition~\ref{l:schrodinger}, $\WFh(Pv-f)\subset U$.
\end{proof}
We also need the following version of the sharp G\r arding inequality, see~\cite[Theorem~4.32]{e-z}
or~\cite[Proposition~5.2]{skds}:
\begin{prop} (Sharp G\r arding inequality)
  \label{l:garding}
Assume that $A\in\Psi^{\comp}(X)$ is compactly
supported and
$\Re\sigma(A)\geq 0$ near $\WFh(u)$. Assume also that $B\in\Psi^{\comp}(X)$ is compactly supported
and elliptic on $\WFh(A)\cap\WFh(u)$. Then
$$
\Re\langle Au,u\rangle\geq -Ch\|Bu\|_{L^2}^2-\mathcal O(h^\infty).
$$
\end{prop}

\section{Abstract framework near infinity}
  \label{s:framework}

In this section, we provide an abstract microlocal framework for studying resonances;
the general assumptions are listed in~\S\ref{s:framework-assumptions}.
Rather than considering resonances as poles of the meromoprhic continuation of the cutoff
resolvent, we define them as solutions of a nonselfadjoint eigenvalue problem
featuring a holomorphic family of Fredholm operators, $\mathcal P(\omega)$.
We assume that the dependence of the principal symbol of $\mathcal P(\omega)$
on $\omega$ can be resolved in a convex neighborhood $\mathcal U$ of the trapped set,
yielding the $\omega$-independent symbol $p$ (and the operator $P$ later in Lemma~\ref{l:resolution}).
Finally, we require the existence of a semiclassically outgoing parametrix for $\mathcal P(\omega)$,
resolving it modulo an operator microlocalized near the trapped set.

In~\S\ref{s:framework-corollaries}, we derive several
useful corollaries of our assumptions, making it possible to treat spatial infinity as a black box
in the following sections. Finally, in~\S\S\ref{s:framework-schrodinger} and
\ref{s:framework-ah},
we provide two examples of situations when the assumptions
of~\S\ref{s:framework-assumptions} (but not necessarily the dynamical assumptions
of~\S\ref{s:dynamics}) are satisfied:
Schr\"odinger operators on $\mathbb R^n$, studied using complex scaling, and
Laplacians on even asymptotically hyperbolic manifolds, handled using~\cite{v1,v2}.

\subsection{General assumptions}
  \label{s:framework-assumptions}

\newcounter{assumptions}

Assume that:
\begin{enumerate}
\item \label{a:basic}
$X$ is a smooth $n$-dimensional manifold without boundary, possibly noncompact,
with a prescribed volume form;
\item $\mathcal P(\omega)\in\Psi^k(X)$ is a family of
properly supported semiclassical pseudodifferential operators
depending holomorphically on $\omega$ lying in an open simply connected set $\Omega\subset\mathbb C$
such that $\mathbb R\cap\Omega$ is connected, with principal symbol $\mathbf p(x,\xi,\omega)$;
\item \label{a:spaces}
$\mathcal H_1,\mathcal H_2$ are $h$-dependent Hilbert spaces such that
$H^{N}_{h,\comp}(X)\subset \mathcal H_j\subset H^{-N}_{h,\loc}(X)$ for some $N$,
with norms of embeddings $\mathcal O(h^{-N})$, and $\mathcal P(\omega)$
is bounded $\mathcal H_1\to\mathcal H_2$ with norm $\mathcal O(1)$;
\item \label{a:fredholm}
for some fixed $[\alpha_0,\alpha_1]\subset \mathbb R\cap\Omega$ and
$C_0>0$, the operator $\mathcal P(\omega):\mathcal H_1\to\mathcal H_2$
is a Fredholm operator of index zero in the region
\begin{equation}
  \label{e:omega-region}
\Re\omega\in [\alpha_0,\alpha_1],\quad
|\Im\omega|\leq C_0h.
\end{equation}
\setcounter{assumptions}{\value{enumi}}
\end{enumerate}
Together with invertibility of $\mathcal P(\omega)$
in a subregion of~\eqref{e:omega-region} proved in Theorem~\ref{t:gaps},
by Analytic Fredholm Theory~\cite[Theorem~D.4]{e-z} our
assumptions imply that
\begin{equation}
  \label{e:resolvent}
\mathcal R(\omega):=\mathcal P(\omega)^{-1}:\mathcal H_2\to\mathcal H_1
\end{equation}
is a meromorphic family of operators with poles of finite rank for $\omega$ satisfying~\eqref{e:omega-region}.
\emph{Resonances} are defined as poles of $\mathcal R(\omega)$. 
Following~\cite[Theorem~2.1]{go-si}, we define the multiplicity of a resonance $\omega_0$ as
\begin{equation}
  \label{e:multiplicity}
{1\over 2\pi i}\Tr\oint_{\omega_0} \mathcal P(\omega)^{-1}\partial_\omega\mathcal P(\omega)\,d\omega.
\end{equation}
Here $\oint_{\omega_0}$ stands for the integral over a contour enclosing $\omega_0$, but no other
poles of $\mathcal R(\omega)$. Since $\mathcal R(\omega)$ has poles of finite rank, we see
that the integral in~\eqref{e:multiplicity} yields a finite dimensional operator on $\mathcal H_1$ and thus
one can take the trace. The fact that the resulting multiplicity is a positive integer
will follow for example from the representation of resonances as zeroes of a Fredholm determinant,
in part~1 of Proposition~\ref{l:grushin-ultimate}. See also~\cite[Appendix~A]{sj-dwe}.

We next fix a `physical region' $\mathcal U$ in phase space, where most of our analysis will take
place, in particular the intersection of the trapped set with the relevant energy
shell will be contained in $\mathcal U$. The region $\mathcal U$ will be contained
in a larger region $\mathcal U'$, which is used to determine when trajectories have escaped from $\mathcal U$.
(See~\eqref{e:sets-schrodinger} and~\eqref{e:sets-ah} for the definitions of $\mathcal U,\mathcal U'$
for the examples we consider.)
We assume that:
\begin{enumerate}
\setcounter{enumi}{\value{assumptions}}
\item \label{a:u'} $\mathcal U'\subset T^*X$ is open and bounded, and each compactly supported
$A\in\Psi^{\comp}(X)$ with $\WFh(A)\subset\mathcal U'$ is bounded $L^2\to\mathcal H_j,
\mathcal H_j\to L^2$, $j=1,2$, with norm $\mathcal O(1)$;
\item\label{a:s-a}
$\mathcal P(\omega)^*= \mathcal P(\omega)+\mathcal O(h^\infty)$
microlocally in $\mathcal U'$ when $\omega\in \mathbb R\cap \Omega$;
\item\label{a:p-res}
for each $(x,\xi)\in \mathcal U'$, the equation
$\mathbf p(x,\xi,\omega)=0$, $\omega\in\Omega$
has unique solution
\begin{equation}
  \label{e:p}
\omega=p(x,\xi).
\end{equation}
Moreover, $p(x,\xi)\in\mathbb R$ and $\partial_\omega \mathbf p(x,\xi,p(x,\xi))<0$
for $(x,\xi)\in \mathcal U'$;
\item \label{a:U}
$\mathcal U\subset \mathcal U'$ is a compactly contained open subset,
whose closure $\overline{\mathcal U}$ is
relatively convex with respect to the Hamiltonian flow of $p$, i.e.
if $\gamma(t),0\leq t\leq T$, is a flow line of $H_p$ in $\mathcal U'$
and $\gamma(0),\gamma(T)\in \overline{\mathcal U}$, then $\gamma([0,T])\subset \overline{\mathcal U}$;
\setcounter{assumptions}{\value{enumi}} 
\end{enumerate}
Note that for $\omega\in \mathbb R\cap\Omega$, Hamiltonian flow lines
of $p$ in $\mathcal U'\cap p^{-1}(\omega)$ are rescaled Hamiltonian flow lines
of $\mathbf p(\,\cdot\,,\omega)$ in $\{\rho\in\mathcal U'\mid \mathbf p(\rho,\omega)=0\}$.
The symbol $p$ is typically the square root of the principal symbol
of the original Laplacian or Schr\"odinger
operator, see~\eqref{e:p-schrodinger} and~\eqref{e:p-ah}.

We can now define the \emph{incoming/outgoing tails} $\Gamma_\pm\subset\overline{\mathcal U}$ as
follows: $\rho\in\overline{\mathcal U}$ lies in $\Gamma_\pm$ if and only if $e^{\mp tH_p}(\rho)$
stays in $\overline{\mathcal U}$ for all $t\geq 0$. Define the \emph{trapped set} as	
\begin{equation}
  \label{e:trapped-set}
K:=\Gamma_+\cap\Gamma_-.
\end{equation}
Note that $\Gamma_\pm$ and $K$ are closed subsets of $\overline{\mathcal U}$
(and thus the sets $\Gamma_\pm$ defined here are smaller than the original $\Gamma_\pm$ defined in the introduction),
and $e^{tH_p}(\Gamma_\pm)\subset\Gamma_\pm$ for $\mp t\geq 0$,
thus $e^{tH_p}(K)=K$ for all $t$. We assume that, with $\alpha_0,\alpha_1$
defined in~\eqref{e:omega-region},
\begin{enumerate}
  \setcounter{enumi}{\value{assumptions}}
\item\label{a:k-compact}
$K\cap p^{-1}([\alpha_0,\alpha_1])$ is a nonempty compact subset of $\mathcal U$.
\setcounter{assumptions}{\value{enumi}}
\end{enumerate}
Finally, we assume the existence of a semiclassically outgoing parametrix,
which will make it possible to reduce our analysis to a neighborhood of the
trapped set in~\S\ref{s:framework-corollaries}:
\begin{enumerate}
  \setcounter{enumi}{\value{assumptions}}
\item \label{a:parametrix}
$\mathcal Q\in\Psi^{\comp}(X)$ is compactly supported, $\WFh(\mathcal Q)\subset \mathcal U$,
and the operator
\begin{equation}
  \label{e:r'}
\mathcal R'(\omega):=(\mathcal P(\omega)-i \mathcal Q)^{-1}:\mathcal H_2\to \mathcal H_1
\end{equation}
satisfies, for $\omega$ in~\eqref{e:omega-region},
\begin{equation}
  \label{e:r'-bound}
\|\mathcal R'(\omega)\|_{\mathcal H_2\to \mathcal H_1}\leq Ch^{-1};
\end{equation}
\item \label{a:outgoing}
for $\omega$ in~\eqref{e:omega-region},
$\mathcal R'(\omega)$ is \emph{semiclassically outgoing} in the following sense:
if $(\rho,\rho')\in\WFh(\mathcal R'(\omega))$ and $\rho,\rho'\in \mathcal U'$,
there exists $t\geq 0$ such that $e^{tH_p}(\rho)=\rho'$ and $e^{sH_p}(\rho)\in \mathcal U'$
for $0\leq s\leq t$. (See Figure~\ref{f:outgoing}(a) below.)
\end{enumerate}

\subsection{Some consequences of general assumptions}
  \label{s:framework-corollaries}

In this section, we derive several corollaries of the assumptions of~\S\ref{s:framework-assumptions},
used throughout the rest of the paper. 

\smallsection{Global properties of the flow}
We start with two technical lemmas:
\begin{lemm}
  \label{l:the-flow}
Assume that $\rho\in\Gamma_\pm$. Then as
$t\to\mp\infty$, the distance $d(e^{tH_p}(\rho),K)$ converges to zero. 
\end{lemm}
\begin{proof}
We consider the case $\rho\in\Gamma_-$. Put
$\gamma(t):=e^{tH_p}(\rho)$, then $\gamma(t)\in\Gamma_-$
for all $t\geq 0$. Assume
that $d(\gamma(t),K)$ does not converge to zero
as $t\to+\infty$, then there exists a sequence of times $t_j\to +\infty$
such that $\gamma(t_j)$ does not lie in a fixed neighborhood
of $K$. By passing to a subsequence, we may assume that
$\gamma(t_j)$ converge to some $\rho_\infty\in\Gamma_-\setminus K$.
Then $\rho_\infty\not\in\Gamma_+$; therefore, there exists
$T\geq 0$ such that $e^{-TH_p}(\rho_\infty)\not\in\overline{\mathcal U}$.
For $j$ large enough, we have $\gamma(t_j-T)=e^{-TH_p}(\gamma(t_j))\not\in\overline{\mathcal U}$
and $t_j\geq T$; this contradicts convexity of $\overline{\mathcal U}$
(assumption~\eqref{a:U}).
\end{proof}
%
\begin{lemm}
  \label{l:the-flow-2}
Assume that $U_1$ is a neighborhood
of $K$ in $\overline{\mathcal U}$. Then there exists a neighborhood
$U_2$ of $K$ in $\overline{\mathcal U}$ such that for each flow line
$\gamma(t)$, $0\leq t\leq T$ of $H_p$ in $\overline{\mathcal U}$,
if $\gamma(0),\gamma(T)\in U_2$, then $\gamma([0,t])\subset U_1$.
\end{lemm}
\begin{proof}
Assume the contrary, then there exist flow lines $\gamma_j(t)$, $0\leq t\leq T_j$,
in $\overline{\mathcal U}$, such that
$d(\gamma_j(0),K)\to 0$, $d(\gamma_j(T_j),K)\to 0$, yet
$\gamma_j(t_j)\not\in U_1$ for some $t_j\in [0,T_j]$.
Passing to a subsequence, we may assume that $\gamma_j(t_j)\to\rho_\infty\in\overline {\mathcal U}\setminus K$.
Without loss of generality, we assume that $\rho_\infty\not\in\Gamma_+$.
Then there exists $T>0$ such that $e^{-TH_p}(\rho_\infty)\in\mathcal U'\setminus\overline{\mathcal U}$,
and thus $e^{-TH_p}(\gamma_j(t_j))\not\in\overline{\mathcal U}$ for $j$ large enough. Since
$\gamma_j([0,T_j])\subset \overline{\mathcal U}$, we have $t_j\leq T$. By passing to a subsequence,
we may assume that $t_j\to t_\infty\in [0,T]$. However, then $\gamma_j(0)\to e^{-t_\infty H_p}(\rho_\infty)$,
which implies that $e^{-t_\infty H_p}(\rho_\infty)\in \Gamma_+$,
contradicting the fact that $\rho_\infty\not\in \Gamma_+$.
\end{proof}

\smallsection{Resolution of dependence on $\omega$}
We reduce the operator $\mathcal P(\omega)$
microlocally near $\mathcal U$ to an operator of the form $P-\omega$,
see also~\cite[\S4]{isz}:
\begin{lemm}
  \label{l:resolution}
There exist:
\begin{itemize}
\item a compactly supported $P\in\Psi^{\comp}(X)$ such that $P^*=P$
and $\sigma(P)=p$ near $\mathcal U$, where~$p$ is defined in~\eqref{e:p}, and
\item a family
of compactly supported operators $\mathcal S(\omega)\in\Psi^{\comp}(X)$,
holomorphic in $\omega\in\Omega$, with $\mathcal S(\omega)^*=\mathcal S(\omega)$
for $\omega\in\mathbb R\cap\Omega$ and $\mathcal S(\omega)$ elliptic near $\mathcal U$,
such that
\end{itemize}
\begin{equation}
  \label{e:resolution}
\mathcal P(\omega)=\mathcal S(\omega)(P-\omega)\mathcal S(\omega)+\mathcal O(h^\infty)
\quad\text{microlocally near }\mathcal U.
\end{equation}
\end{lemm}
\begin{proof}
We argue by induction, constructing compactly supported operators $P_j,\mathcal S_j(\omega)\in\Psi^{\comp}(X)$, such that
$P_j^*=P_j$, $\mathcal S_j^*(\omega)=\mathcal S_j(\omega)$ for $\omega\in\mathbb R\cap\Omega$, and
$\mathcal P(\omega)=\mathcal S_j(\omega)(P_j-\omega)\mathcal S_j(\omega)+\mathcal O(h^{j+1})$
microlocally near $\mathcal U$. It will remain
to take the asymptotic limit.

For $j=0$, it suffices to take any $P_0,\mathcal S_0(\omega)$ such that $\sigma(P_0)=p$
and $\sigma(\mathcal S_0(\omega))(\rho)=s_0(\rho,\omega)$ near $\mathcal U$, where
(with $\mathbf p(\cdot,\omega)$ denoting the principal symbol of $\mathcal P(\omega)$)
$$
\mathbf p(\rho,\omega)=s_0(\rho,\omega)^2(p(\rho)-\omega),\quad
\rho\in \mathcal U';
$$
the existence of such $s_0$ and the fact
that it is real-valued for real $\omega$ follows from assumption~\eqref{a:p-res}.

Now, given $P_j,\mathcal S_j(\omega)$ for some $j\geq 0$,
we construct $P_{j+1},\mathcal S_{j+1}(\omega)$. We have
$\mathcal P(\omega)=\mathcal S_j(\omega)(P_j-\omega)\mathcal S_j(\omega)+h^{j+1} R_j(\omega)$
microlocally near $\mathcal U$, where $R_j(\omega)\in\Psi^{\comp}$ is a holomorphic
family of operators and, by assumption~\eqref{a:s-a}, $R_j(\omega)^*=R_j(\omega)+\mathcal O(h^\infty)$
microlocally near $\mathcal U$ when $\omega\in\mathbb R\cap\Omega$. We then put
$P_{j+1}=P_j+h^{j+1}A_j$, $\mathcal S_{j+1}(\omega)=\mathcal S_j(\omega)+h^{j+1}B_j(\omega)$,
where $\sigma(A_j)=p_j,\sigma(B_j(\omega))(\rho)=s_j(\rho,\omega)$ near $\mathcal U$ and
$$
\sigma(R_j)(\rho,\omega)=2s_0(\rho,\omega)s_j(\rho,\omega)(p(\rho)-\omega)
+s_0(\rho,\omega)^2p_j(\rho),\quad
\rho\in\mathcal U'.
$$
The existence of $s_j(\rho,\omega),p_j(\rho)$ and the fact that
$p_j(\rho)\in\mathbb R$ and $s_j(\rho,\omega)\in\mathbb R$ for $\rho$ near $\mathcal U$ and
$\omega\in\mathbb R\cap\Omega$ follow
from assumption~\eqref{a:p-res}. In particular, we put
$p_j(\rho)=\sigma(R_j)(\rho,p(\rho))/s_0(\rho,p(\rho))^2$.
\end{proof}
Note that, if $u(h)\in \mathcal H_1,f(h)\in \mathcal H_2$ have norms polynomially
bounded in $h$ (and in light of assumption~\eqref{a:spaces} are $h$-tempered in the
sense of~\S\ref{s:prelim-basics}), and $\mathcal P(\omega)u=f$, then
\begin{equation}
  \label{e:conjugated}
(P-\omega)\mathcal S(\omega)u=\mathcal S'(\omega)f+\mathcal O(h^\infty)\quad
\text{microlocally near }\mathcal U,
\end{equation}
where $\mathcal S'(\omega)\in\Psi^{\comp}(X)$ is an elliptic parametrix of $\mathcal S(\omega)$ microlocally 
near $\mathcal U$, constructed in Proposition~\ref{l:eparametrix}.
\begin{figure}
\includegraphics{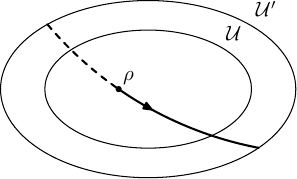}
\qquad\qquad\qquad
\includegraphics{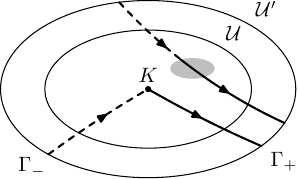}
\hbox to\hsize{\hss (a)\hss\hss (b)\hss}
\caption{(a) Assumption~\eqref{a:outgoing}, with
the undashed part of the flow line of $\rho$
corresponding to $\rho'\in\mathcal U'$
such that $(\rho,\rho')\in\WFh(\mathcal R'(\omega))$.\break
(b) An illustration
of Lemma~\ref{l:propagate-outgoing}, with $\WFh(f)$ the shaded set
and $\WFh(u)$ containing undashed parts of the flow lines.}
\label{f:outgoing}
\end{figure}

\smallsection{Microlocalization of $\mathcal R(\omega)$}
Next, we use the semiclassically outgoing parametrix $\mathcal R'(\omega)$ from~\eqref{e:r'}
to derive a key restriction on the wavefront set of functions in the image of $\mathcal R(\omega)$,
see Figure~\ref{f:outgoing}(b):
\begin{lemm}
  \label{l:propagate-outgoing}
Assume that $u(h)\in \mathcal H_1,f(h)\in \mathcal H_2$ have norms polynomially
bounded in $h$, $\mathcal P(\omega)u=f$ 
for some $\omega=\omega(h)$ satisfying~\eqref{e:omega-region},
and $\WFh(f)\subset \mathcal U$. Then for each $\rho\in\WFh(u)\cap \mathcal U$,
if $\gamma(t)=e^{tH_p}(\rho)$ is the corresponding maximally extended flow line in $\mathcal U'$, then
either $\gamma(t)\in \overline{\mathcal U}$ for all $t\leq 0$
or $\gamma(t)\in\WFh(f)$ for some $t\leq 0$.
\end{lemm}
\begin{proof}
By propagation of singularities (Proposition~\ref{l:microhyperbolic})
applied to~\eqref{e:conjugated}, we see that either $\gamma(t)\in\overline{\mathcal U}$ for all $t\leq 0$,
or $\gamma(t)\in\WFh(f)$ for some $t\leq 0$, or there exists $t\leq 0$ such that
$\gamma(t)\in\WFh(u)\cap (\mathcal U'\setminus \overline{\mathcal U})$; we need to exclude the third case.
However, in this case by convexity of $\overline{\mathcal U}$ (assumption~\eqref{a:U}),
$\gamma(t-s)\not\in\overline{\mathcal U}$ for all $s\geq 0$; by assumption~\eqref{a:outgoing},
and since
$u=\mathcal R'(\omega)(f-i \mathcal Qu)$
with $\WFh(f-i\mathcal Q u)\subset \mathcal U$, we see that $\gamma(t)\not\in\WFh(u)$,
a contradiction.
\end{proof}
It follows from Lemma~\ref{l:propagate-outgoing} that any resonant state,
i.e. a function $u$ such that $\|u\|_{\mathcal H_1}\sim 1$ and
$\mathcal P(\omega)u=0$, has to satisfy $\WFh(u)\cap \mathcal U\subset\Gamma_+$.

The next statement improves on the parametrix $\mathcal R'(\omega)$, inverting
the operator $\mathcal P(\omega)$ outside of any given neighborhood of the trapped set.
One can see this as a geometric control statement (see for instance~\cite[Theorem~3]{b-z}).
\begin{lemm}
  \label{l:smart-parametrix}
Let $W\subset \mathcal U$ be a neighborhood of $K\cap p^{-1}([\alpha_0,\alpha_1])$
(which is a compact subset of $\mathcal U$ by assumption~\eqref{a:k-compact}),
and assume that $f(h)\in\mathcal H_2$ has norm bounded polynomially in $h$
and each $\omega=\omega(h)$ is in~\eqref{e:omega-region}. Then there exists
$v(h)\in\mathcal H_1$, with $f-\mathcal P(\omega)v$ compactly supported in $X$
and
$$
\|v\|_{\mathcal H_1}\leq Ch^{-1}\|f\|_{\mathcal H_2},\quad
\|\mathcal P(\omega)v\|_{\mathcal H_2}\leq C\|f\|_{\mathcal H_2},\quad
\WFh(f-\mathcal P(\omega)v)\subset W.
$$
\end{lemm}
\begin{proof}
First of all, take compactly supported $\mathcal Q'\in\Psi^{\comp}(X)$ such that
$\WFh(\mathcal Q')\subset \mathcal U$ and $\mathcal Q'=1$ microlocally
near $\WFh(\mathcal Q)$ (with $\mathcal Q$ defined in assumption~\eqref{a:parametrix}), and put
$$
v_1:=(1-\mathcal Q')\mathcal R'(\omega)f.
$$
Then by~\eqref{e:r'-bound}, $\|v_1\|_{\mathcal H_1}\leq Ch^{-1}\|f\|_{\mathcal H_2}$
and $\mathcal P(\omega)v_1=f_1$, where
$$
f_1=(1-\mathcal Q'-[\mathcal P(\omega),\mathcal Q']\mathcal R'(\omega)+
(1-\mathcal Q')i \mathcal Q \mathcal R'(\omega))f.
$$
Since $(1-\mathcal Q')i \mathcal Q=\mathcal O(h^\infty)_{\Psi^{-\infty}}$, by~\eqref{e:r'-bound}
we find $\|f_1\|_{\mathcal H_2}\leq C\|f\|_{\mathcal H_2}$, $f-f_1$ is compactly supported, and
$\WFh(f-f_1)\subset \WFh(\mathcal Q')$. It is now enough to prove our statement
for $f-f_1$ in place of $f$; therefore, we may assume that
$f$ is compactly supported and
$$
\WFh(f)\subset \WFh(\mathcal Q').
$$
Since $\WFh(\mathcal Q')$ is compact, by a microlocal partition of unity
we may assume that $\WFh(f)$ is contained in a small neighborhood
of some fixed $\rho\in\WFh(\mathcal Q')\subset\mathcal U$. We now consider three cases:

\noindent\textbf{Case 1}: $\rho\not\in p^{-1}([\alpha_0,\alpha_1])$. Then the operator
$\mathcal P(\omega)$ is elliptic at $\rho$, therefore we may assume it is elliptic
on $\WFh(f)$. The function $v$ is then obtained by applying to $f$ an elliptic parametrix
of $\mathcal P(\omega)$ given in Proposition~\ref{l:eparametrix}; we have
$f-\mathcal P(\omega)v=\mathcal O(h^\infty)_{C_0^\infty}$.

\noindent\textbf{Case 2}: $\rho\in\Gamma_-\cap p^{-1}([\alpha_0,\alpha_1])$.
By Lemma~\ref{l:the-flow}, there exists $t\geq 0$ such that $e^{tH_p}(\rho)\in W$.
We may then assume that $e^{tH_p}(\WFh(f))\subset W$, and $v$ is then constructed
by Proposition~\ref{l:hparametrix},
using~\eqref{e:resolution}; we have $\WFh(v)\subset\mathcal U$ and
$\WFh(f-\mathcal P(\omega) v)\subset W$.

\noindent\textbf{Case 3}: $\rho\not\in\Gamma_-$. Then there exists $t\geq 0$
such that
$e^{tH_p}(\rho)\in\mathcal U'\setminus\overline{\mathcal U}$. As in case~2,
subtracting from $v$ the parametrix of Proposition~\ref{l:hparametrix},
we may assume that $f$ is instead microlocalized in a neighborhood of $e^{tH_p}(\rho)$.
Now, put $v=\mathcal R'(\omega)f$, with $\mathcal R'(\omega)$ defined
in~\eqref{e:r'}; then $\|v\|_{\mathcal H_1}\leq Ch^{-1}\|f\|_{\mathcal H_2}$
by~\eqref{e:r'-bound} and
$$
f-\mathcal P(\omega)v=-i \mathcal Qv.
$$
However, by assumption~\eqref{a:outgoing}, and by convexity of $\overline{\mathcal U}$ (assumption~\eqref{a:U}),
we have $\WFh(\mathcal Q)\cap \WFh(v)=\emptyset$ and thus $f-\mathcal P(\omega) v=\mathcal O(h^\infty)_{C_0^\infty}$.
\end{proof}
Finally, we can estimate the norm of $u\in\mathcal H_1$ by the norm
of $\mathcal P(\omega)u$ and the norm of $u$ microlocally near the trapped set.
This can be viewed as an observability statement (see for instance~\cite[Theorem~2]{b-z}).
\begin{lemm}
  \label{l:smart-bound}
Let $A\in\Psi^{\comp}(X)$ be compactly supported and elliptic on $K\cap p^{-1}([\alpha_0,\alpha_1])$.
Then we have for any $u\in \mathcal H_1$ and any $\omega$ in~\eqref{e:omega-region},
\begin{equation}
  \label{e:smart-bound}
\|u\|_{\mathcal H_1}\leq C\|Au\|_{L^2}+Ch^{-1}\|\mathcal P(\omega)u\|_{\mathcal H_2}.
\end{equation}
\end{lemm}
\begin{proof}
By rescaling, we may assume that $u=u(h)$ has $\|u\|_{\mathcal H_1}=1$
and put $f=\mathcal P(\omega)u$. Take a neighborhood
$W$ of $K\cap p^{-1}([\alpha_0,\alpha_1])$ such that
$A$ is elliptic on $W$.
Replacing $u$ by $u-v$, where $v$
is constructed from $f$ in Lemma~\ref{l:smart-parametrix}, we
may assume that $\WFh(f)\subset W$.

Take $\mathcal Q',\mathcal Q''\in\Psi^{\comp}(X)$ compactly supported, with $\WFh(\mathcal Q'')\subset\mathcal U$,
$\mathcal Q''=1+\mathcal O(h^\infty)$ microlocally near $\WFh(\mathcal Q')$,
and $\mathcal Q'=1+\mathcal O(h^\infty)$ microlocally near $\WFh(\mathcal Q)$
(with $\mathcal Q$ defined in assumption~\eqref{a:parametrix}). Then
by the elliptic estimate (Proposition~\ref{l:elliptic}),
\begin{align}
  \label{e:toronto-1}
\|\mathcal Q' u\|_{\mathcal H_1}&\leq C\|\mathcal Q''u\|_{L^2}+\mathcal O(h^\infty),\\
  \label{e:toronto-2}
\|[\mathcal P(\omega),\mathcal Q']u\|_{\mathcal H_2}&\leq Ch\|\mathcal Q'' u\|_{L^2}
+\mathcal O(h^\infty).
\end{align}
Now,
$$
(1-\mathcal Q')u=\mathcal R'(\omega)((1-\mathcal Q')f-[\mathcal P(\omega),\mathcal Q']u-i\mathcal Q(1-\mathcal Q')u);
$$
since $i\mathcal Q(1-\mathcal Q')=\mathcal O(h^\infty)_{\Psi^{-\infty}}$, we get by~\eqref{e:r'-bound}
and~\eqref{e:toronto-2},
$$
\|(1-\mathcal Q')u\|_{\mathcal H_1}\leq C\|\mathcal Q'' u\|_{L^2}+Ch^{-1}\|f\|_{\mathcal H_2}
+\mathcal O(h^\infty);
$$
by~\eqref{e:toronto-1}, it then remains to prove that
$$
\|\mathcal Q'' u\|_{L^2}\leq C\|Au\|_{L^2}+Ch^{-1}\|f\|_{\mathcal H_2}+\mathcal O(h^\infty).
$$
By a microlocal partition of unity, it suffices to estimate $\|Bu\|_{L^2}$ for
$B\in\Psi^{\comp}(X)$ compactly supported with $\WFh(B)$ in a small neighborhood of some $\rho\in\WFh(\mathcal Q'')\subset\mathcal U$.
We now consider three cases:

\noindent\textbf{Case 1}: $\rho\not\in p^{-1}([\alpha_0,\alpha_1])$. Then
$\mathcal P(\omega)$ is elliptic at $\rho$, therefore we may assume it is elliptic
on $\WFh(B)$. By Proposition~\ref{l:elliptic}, we get
$\|Bu\|_{L^2}\leq C\|f\|_{\mathcal H_2}+\mathcal O(h^\infty)$.

\noindent\textbf{Case 2}: there exists $t\leq 0$ such that
$e^{tH_p}(\rho)\in W$, therefore we may assume
that $e^{tH_p}(\WFh(B))\subset W$. Since $A$ is elliptic on $W$, by Proposition~\ref{l:microhyperbolic}
together with~\eqref{e:resolution},
we get $\|Bu\|_{L^2}\leq C\|Au\|_{L^2}+Ch^{-1}\|f\|_{\mathcal H_2}+\mathcal O(h^\infty)$.

\noindent\textbf{Case 3}: if $\gamma(t)=e^{tH_p}(\rho)$ is the maximally extended
trajectory of $H_p$ in $\mathcal U'$, then $\rho\in p^{-1}([\alpha_0,\alpha_1])$
and $\gamma(t)\not\in W$ for all $t\leq 0$. By Lemma~\ref{l:the-flow},
we have $\rho\not\in\Gamma_+$. Since $\WFh(f)\subset W$, Lemma~\ref{l:propagate-outgoing}
implies that $\rho\not\in\WFh(u)$. We may then assume that $\WFh(B)\cap\WFh(u)=\emptyset$
and thus $\|Bu\|_{L^2}=\mathcal O(h^\infty)$.
\end{proof}

\subsection{Example: Schr\"odinger operators on \texorpdfstring{$\mathbb R^n$}{Rn}}
  \label{s:framework-schrodinger}

In this section, we consider the case described the introduction, namely
a Schr\"odinger operator on $X=\mathbb R^n$ with
$$
P_V=h^2\Delta+V(x),
$$
where $\Delta$ is the Euclidean Laplacian
and $V\in C_0^\infty(\mathbb R^n;\mathbb R)$. We will explain how
this case fits into the framework of~\S\ref{s:framework-assumptions}.

To define resonances for~$P_0$, we use the method of \emph{complex scaling}
of Aguilar--Combes~\cite{ag-co}, which also applies to more general operators and potentials~--
see \cite{sj-z-91}, \cite{sj}, and the references
given there.
Take $R>0$ large enough so that
$$
\supp V\subset \{|x|<R/2\}.
$$
Fix the deformation angle $\theta\in (0,\pi/2)$ and consider a deformation
$\Gamma_{\theta,R}\subset\mathbb C^n$ of $\mathbb R^n$ defined by
$$
\Gamma_{\theta,R}:=\{x+iF_{\theta,R}(x)\mid x\in\mathbb R^n\},
$$
where $F_{\theta,R}:\mathbb R^n\to\mathbb R^n$ is defined in polar coordinates
$(r,\varphi)\in [0,\infty)\times \mathbb S^{n-1}$ by
$$
F_{\theta,R}(r,\varphi)=(f_{\theta,R}(r),\varphi),
$$
and the function $f_{\theta,R}\in C^\infty([0,\infty))$ is chosen so that
(see Figure~\ref{f:complex-scaling}(a))
$$
\begin{aligned}
f_{\theta,R}(r)=0,\quad r\leq R;&\quad
f_{\theta,R}(r)=r\tan\theta,\quad r\geq 2R;\\
f'_{\theta,R}(r)\geq 0,\quad r\geq 0;&\quad
\{f'_{\theta,R}=0\}=\{f_{\theta,R}=0\}.
\end{aligned}
$$
Note that
$$
\begin{aligned}
\Gamma_{\theta,R}\cap \{|\Re z|\leq R\}&=\mathbb R^n\cap \{|\Re z|\leq R\};\\
\Gamma_{\theta,R}\cap \{|\Re z|\geq 2R\}&=e^{i\theta}\mathbb R^n\cap \{|\Re z|\geq 2R\}.
\end{aligned}
$$
Define the deformed differential operator $\widetilde P_V$ on $\Gamma_{\theta,R}$ it as follows:
$\widetilde P_V=P_V$ on $\mathbb R^n\cap\Gamma_{\theta,R}$,
and on the complementing region $\{|\Re z|>R\}$, it is defined by the formula
$$
\widetilde P_V(v)=\sum_{j=1}^n (hD_{z_j})^2 \tilde v|_{\Gamma_{\theta,R}},
$$
for each $v\in C_0^\infty(\Gamma_{\theta,R}\cap \{|\Re z|>R\})$
and each almost analytic continuation $\tilde v$ of $v$
(that is, $\tilde v|_{\Gamma_{\theta,R}}=v$ and $\partial_{\bar z}\tilde v$ vanishes
to infinite order on $\Gamma_{\theta,R}$~-- the existence of such continuation
follows from the fact that $\Gamma_{\theta,R}$ is totally real, that is
for each $z\in \Gamma_{\theta,R}$, $T_z\Gamma_{\theta,R}\cap iT_z\Gamma_{\Theta,R}=0$).
We identify $\Gamma_{\theta,R}$ with $\mathbb R^n$ by the map
$$
\iota:\mathbb R^n\to\Gamma_{\theta,R}\subset\mathbb C^n,\quad
\iota(x)=x+iF_{\theta,R}(x),
$$
so that $\widetilde P_V$ can be viewed as a second order differential operator
on $\mathbb R^n$.
Then in polar coordinates $(r,\varphi)$, we can write for $r>R$,
$$
\widetilde P_V=\bigg({1\over 1+if'_{\theta,R}(r)}hD_r\bigg)^2-{(n-1)i\over (r+if_{\theta,R}(r))(1+if'_{\theta,R}(r))}h^2D_r
+{\Delta_\varphi\over (r+if_{\theta,R}(r))^2},
$$
with $\Delta_\varphi$ denoting the Laplacian on the round sphere $\mathbb S^{n-1}$.
We have
\begin{equation}
  \label{e:cs-symbol}
\sigma(\widetilde P_V)={|\xi_r|^2\over (1+i f'_{\theta,R}(r))^2}+{|\xi_\varphi|^2\over (r+if_{\theta,R}(r))^2}+V(r,\varphi).
\end{equation}
Fix a range of energies $[\alpha_0,\alpha_1]\subset (0,\infty)$ and
a bounded open set $\Omega\subset\mathbb C$ such that
(see Figure~\ref{f:complex-scaling}(b))
\begin{figure}
\includegraphics{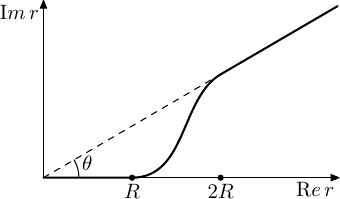}
\qquad\qquad
\includegraphics{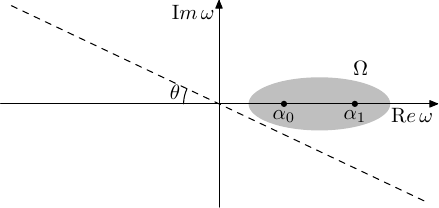}
\hbox to\hsize{\hss (a) \hss\hss (b)\hss}
\caption{(a) The graph of $f_{\theta,r}$. (b) The region where complex scaling provides meromorphic
continuation of the resolvent.}
\label{f:complex-scaling}
\end{figure}
$$
[\alpha_0,\alpha_1]\subset\Omega,\quad
\overline\Omega\subset \{-\theta<\arg \omega<\pi-\theta\}.
$$
For $\omega\in\Omega$, define the operator
$$
\mathcal P(\omega)=\widetilde P_V-\omega^2:\mathcal H_1\to \mathcal H_2,\qquad
\mathcal H_1:=H^2_h(\mathbb R^n),\quad
\mathcal H_2:=L^2(\mathbb R^n).
$$
Then $\mathcal P(\omega)$ is a Fredholm operator $\mathcal H_1\to\mathcal H_2$ for $\omega\in\Omega$. Indeed,
$$
\mathcal P(\omega)=\cos^2\theta e^{-2i\theta}h^2\Delta-\omega^2\quad\text{on }
\{|x|\geq 2R\},
$$
thus $\mathcal P(\omega)$ is elliptic on $\{|x|\geq 2R\}$, as well as for $|\xi|$ large enough,
in the class $S(\langle\xi\rangle^2)$ of~\cite[\S4.4.1]{e-z} (this class
incorporates the behavior of symbols as $x\to\infty$, in contrast with those
used in~\S\ref{s:prelim-basics}).
Using a construction similar to Lemma~\ref{l:eparametrix},
but with symbols in the class $S(\langle\xi\rangle^{-2})$, we can
define a parametrix near (both spatial and fiber) infinity, $\mathcal R_\infty(\omega)$,
with $\|\mathcal R_\infty\|_{L^2(\mathbb R^n)\to H^2_h(\mathbb R^n)}=\mathcal O(1)$ and
\begin{equation}
  \label{e:cs-par}
\begin{aligned}
\mathcal R_\infty(\omega)\mathcal P(\omega)&=1+Z(\omega)+\mathcal O(h^\infty)_{H^2_h(\mathbb R^n)\to H^2_h(\mathbb R^n)},\\
\mathcal P(\omega)\mathcal R_\infty(\omega)&=1+Z'(\omega)+\mathcal O(h^\infty)_{L^2(\mathbb R^n)\to L^2(\mathbb R^n)},
\end{aligned}
\end{equation}
where $Z(\omega),Z'(\omega)\in\Psi^{\comp}(\mathbb R^n)$ are compactly supported inside $\{|x|<2R+1\}$. Since
$1+\mathcal O(h^\infty)$ is invertible and $Z(\omega),Z'(\omega)$ are compact, we see that
$\mathcal P(\omega)$ is indeed a Fredholm operator $\mathcal H_1\to\mathcal H_2$. We have thus
verified assumptions~\eqref{a:basic}--\eqref{a:fredholm} of~\S\ref{s:framework-assumptions}.

The identification of the poles of $\mathcal R(\omega)$ with the 
poles of the meromorphic continuation of the resolvent
$R_V(\omega)=(P_V-\omega^2)^{-1}$ defined in~\eqref{e:r-v}
from $ \{\Im \omega>0\}$ to  $\Omega$, and in fact, the 
existence of such a continuation, follows from the following 
formula (implicit in~\cite{sj}, and discussed in~\cite{TaZw}):
if $ \chi \in C_0^\infty ( \mathbb R^n ) $, $ \supp \chi \Subset B ( 0 , R ) $, 
then 
\begin{equation}
\label{eq:two_r}
  \chi \mathcal R(\omega) \chi = \chi R_V(\omega) \chi.
\end{equation}
This is initially valid in $ \Omega\cap\{\Im \omega > 0\} $ so that the right-hand side
is well-defined, and then by analytic continuation in the region 
where the left hand side is meromorphic.

Now, we take intervals
$$
[\alpha_0,\alpha_1]\Subset [\beta_0,\beta_1]\Subset [\beta'_0,\beta'_1]\subset\Omega\cap (0,\infty)
$$
and put
\begin{equation}
  \label{e:sets-schrodinger}
\begin{aligned}
\mathcal U'&:=\{|x|<R,\
|\xi|^2+V(x)\in ((\beta'_0)^2,(\beta'_1)^2)\},\\
\mathcal U&:=\{|x|<3R/4,\
|\xi|^2+V(x)\in (\beta_0^2,\beta_1^2)\}.
\end{aligned}
\end{equation}
Note that $\mathcal P(\omega)=P_V-\omega^2$ in $\mathcal U'$; this verifies assumptions~\eqref{a:u'} and~\eqref{a:s-a}.
Assumption~\eqref{a:p-res} is also satisfied, with
\begin{equation}
  \label{e:p-schrodinger}
p(x,\xi)=\sqrt{|\xi|^2+V(x)},\quad
(x,\xi)\in\mathcal U'.
\end{equation}
The operators $P$ and $\mathcal S(\omega)$ from Lemma~\ref{l:resolution} take
the form, microlocally near $\mathcal U$,
\begin{equation}
  \label{e:sqrt}
P=\sqrt{P_V},\quad
\mathcal S(\omega)=\sqrt{\sqrt{P_V}+\omega}.
\end{equation}
Here the square root is understood in the microlocal sense: for an
operator $A\in\Psi^k(X)$ with $\sigma(A)>0$ on $\mathcal U'$, we
define the microlocal square root $\sqrt{A}\in\Psi^{\comp}(X)$ of $A$
in $\mathcal U'$ as the (unique modulo $\mathcal O(h^\infty)$
microlocally in $\mathcal U'$) operator such that
$(\sqrt{A})^2=A+\mathcal O(h^\infty)$ microlocally in $\mathcal U'$
and $\sigma(\sqrt{A})=\sqrt{\sigma(A)}$.  See for
example~\cite[Lemma~4.6]{gr-sj} for details of the construction of the
symbol.

Assumption~\eqref{a:U}, namely convexity of $\overline{\mathcal U}$,
is satisfied since for each $(x,\xi)\in\mathcal U'$, if $|x|\geq R/2$
and $H_p|x|^2=0$ at $(x,\xi)$, then $H_p^2|x|^2>0$ at $(x,\xi)$;
therefore, the function $|x|^2$ cannot attain a local maximum on a
trajectory of $e^{tH_p}$ in $\mathcal U'\setminus\overline{\mathcal
U}$.  Same observation shows assumption~\eqref{a:k-compact}; in fact,
$K\subset \{|x|\leq R/2\}$.

Finally, for assumptions~\eqref{a:parametrix} and~\eqref{a:outgoing},
we take any compactly supported $\mathcal Q\in\Psi^{\comp}(X)$ such
that $\WFh(\mathcal Q)\subset\mathcal U$ and
$$
\sigma(\mathcal Q)\geq 0\quad\text{everywhere};\quad
\sigma(\mathcal Q)>0\quad\text{on } p^{-1}([\alpha_0,\alpha_1])\cap \{|x|\leq R/2\}.
$$
To verify assumption~\eqref{a:parametrix}, consider an arbitrary family
$u=u(h)\in H^2_h(\mathbb R^n)$, with norm bounded polynomially in $h$, and put
$$
f=(\mathcal P(\omega)-i\mathcal Q)u,
$$
where $\omega$ satisfies~\eqref{e:omega-region}.
By~\eqref{e:cs-symbol}, and since $\Im\omega=\mathcal O(h)$, we find
$$
\begin{gathered}
\Im\sigma(\mathcal P(\omega))\leq 0\quad\text{everywhere};\\
\{\langle\xi\rangle^{-2}\sigma(\mathcal P(\omega))=0\}\subset \{F_{\theta,R}(x)=0\}.
\end{gathered}
$$
Note also that $\sigma(\mathcal P(\omega))=|\xi|^2+V(x)-\omega^2$
on $\{F_{\theta,R}(x)=0\}$. Together with the convexity property of $|x|^2$ mentioned
above, we see that for each $\rho\in T^*X$, there exists $t\leq 0$ such that
$\mathcal P(\omega)-i\mathcal Q$ is elliptic at
$\exp(tH_{\Re\sigma(\mathcal P(\omega))})(\rho)$. Since $\Im\sigma(\mathcal P(\omega)-i\mathcal Q)\leq 0$ everywhere,
by propagation of singularities with a complex absorbing term (Proposition~\ref{l:microhyperbolic})
and the elliptic estimate (Proposition~\ref{l:elliptic}) we get
$$
\|Z(\omega)u\|_{H^2_h}\leq Ch^{-1}\|f\|_{L^2}+\mathcal O(h^\infty),
$$
where $Z(\omega)$ is defined in~\eqref{e:cs-par}. Then by~\eqref{e:cs-par},
$$
\|u\|_{H^2_h(\mathbb R^n)}\leq C\|f\|_{L^2(\mathbb R^n)}+\|Z(\omega)u\|_{H^2_h}+\mathcal O(h^\infty)
\leq Ch^{-1}\|f\|_{L^2(\mathbb R^n)}+\mathcal O(h^\infty),
$$
proving the estimate~\eqref{e:r'-bound} of assumption~\eqref{a:parametrix}.

Assumption~\eqref{a:outgoing} is proved in a similar fashion: assume that
$\WFh(f)\subset\mathcal U'$ and $\rho'\in\WFh(u)\cap \mathcal U'$.
Denote $\gamma(t)=\exp(tH_{\Re\sigma(\mathcal P(\omega))})(\rho')$.
Then there exists
$t_0\geq 0$ such that $\mathcal P(\omega)-i \mathcal Q$ is elliptic at $\gamma(-t_0)$.
By Proposition~\ref{l:microhyperbolic}, we see that
either $\exp(-tH_{\Re\sigma(\mathcal P(\omega))})(\rho')\in\WFh(f)$ for some $t\in [0,t_0]$
or $\exp(-t_0H_{\Re\sigma(\mathcal P(\omega))})(\rho')\in\WFh(u)$, in which case
this point also lies in $\WFh(f)$ by Proposition~\ref{l:elliptic}; therefore,
$\gamma(-t)\in\WFh(f)$ for some $t\geq 0$.
Let $t_1$ be the minimal nonnegative number such that $\gamma(-t_1)\in\WFh(f)$;
we may assume that $t_1>0$.
Since $\gamma((-t_1,0])$ does not intersect $\WFh(f)$, it also does not intersect
the elliptic set of $\mathcal P(\omega)$; therefore, $\gamma([-t_1,0])\subset \{F_{\theta,R}(x)=0\}$
and thus $\sigma(\mathcal P(\omega))=p^2-\omega^2$ on $\gamma([-t_1,0])$. It follows
that $e^{-tH_p}(\rho')\in\WFh(f)$ for some $t\geq 0$, as required.

\subsection{Example: even asymptotically hyperbolic manifolds}
  \label{s:framework-ah}

In this section, we define resonances, in the framework of~\S\ref{s:framework-assumptions},
for an $n$-dimensional
complete noncompact Riemannian manifold
$(M,g)$ which is \emph{asymptotically hyperbolic} in the following sense: $M$ is diffeomorphic to
the interior of a smooth manifold with boundary $\overline M$, and for some choice
of the boundary defining function $\tilde x\in C^\infty(\overline M)$
and the product decomposition $\{\tilde x<\varepsilon\}\sim [0,\varepsilon)\times\partial\overline M$, the metric $g$ takes the following form in $\{0<\tilde x<\varepsilon\}$:
\begin{equation}
  \label{e:ah-metric}
g={d\tilde x^2+g_1(\tilde x,\tilde y,d\tilde y)\over \tilde x^2}.
\end{equation}
Here $g_1$ is a family of Riemannian metrics on $\partial\overline M$ depending smoothly on $\tilde x\in [0,\varepsilon)$.
We moreover require that the metric is~\emph{even} in the sense that
$g_1$ is a smooth function of $\tilde x^2$.

To put the Laplacian $\Delta_g$ on $M$ into the framework of~\S\ref{s:framework-assumptions},
we use the recent construction of Vasy~\cite{v2}. We follow in part~\cite[\S4.1]{fwl},
see also~\cite[Appendix~B]{fwl} for a detailed description of the phase
space properties of the resulting operator in a model case.
Take the space $\overline M_{\even}$ obtained from $\overline M$ by taking the new boundary defining function
$\mu=\tilde x^2$ and put (see~\cite[\S3.1]{v2})
$$
P_1(\omega)=\mu^{-{1\over 2}-{n+1\over 4}}e^{i\omega\phi\over h}(h^2(\Delta_g-(n-1)^2/4)-\omega^2)
e^{-{i\omega\phi\over h}}\mu^{-{1\over 2}+{n+1\over 4}}.
$$
Here $\phi$ is a smooth real-valued function on $M$ such that
$$
e^\phi=\mu^{1/2}(1+\mu)^{-1/4}\quad\text{on }\{0<\mu<\delta_0\},
$$
where $\delta_0>0$ is a small constant; the values of $\phi$ on $\{\mu\geq\delta_0\}$
are chosen as in the paragraph preceding~\cite[(3.14)]{v2}. We can furthermore choose $e^\phi$
and $\mu$ to be equal to 1 near the set $\{\tilde x>\varepsilon_0/2\}$, for any fixed $\varepsilon_0>0$
(and $\delta_0$ chosen small depending on~$\varepsilon_0$)
so that
\begin{equation}
  \label{e:same}
P_1(\omega)=h^2(\Delta_g-(n-1)^2/4)-\omega^2\quad\text{on }\{\tilde x>\varepsilon_0/2\}.
\end{equation}
The differential operator $P_1(\omega)$ has
coefficients smooth up to the boundary of $\overline M_{\even}$; then it is possible
to find a compact $n$-dimensional manifold $X$ without boundary such that $\overline M_{\even}$
embeds into $X$ as $\{\mu\geq 0\}$
and extend $P_1(\omega)$ to an operator $P_2(\omega)\in\Psi^2(X)$, see~\cite[\S3.5]{v2}
or~\cite[Lemma~4.1]{fwl}. Finally, we fix a complex absorbing operator $Q\in\Psi^2(X)$, with Schwartz kernel supported
in the nonphysical region $\{\mu<0\}$, satisfying the assumptions of~\cite[\S3.5]{v2}.
We now fix an interval $[\alpha_0,\alpha_1]\subset (0,\infty)$, take
$\Omega\subset\mathbb C$ a small neighborhood of $[\alpha_0,\alpha_1]$, and put
$$
\mathcal P(\omega):=P_2(\omega)-iQ,\quad
\omega\in\Omega.
$$
Fix $C_0>0$, take $s>C_0+1/2$, and put $\mathcal H_2=H^{s-1}_h(X)$ and
$$
\mathcal H_1=\{u\in H_h^s(X)\mid P_2(1)u\in H_h^{s-1}(X)\},\quad
\|u\|_{\mathcal H_1}^2=\|u\|_{H_h^s(X)}^2+\|P_2(1)u\|_{H_h^{s-1}(X)}^2.
$$
It is proved in~\cite[Theorem~4.3]{v2} that for $\omega$ satisfying~\eqref{e:omega-region},
the operator $\mathcal P(\omega):\mathcal H_1\to\mathcal H_2$ is a Fredholm operator of index zero;
therefore, we have verified assumptions~\eqref{a:basic}--\eqref{a:fredholm} of~\S\ref{s:framework-assumptions}.
The poles of $\mathcal R(\omega)=\mathcal P(\omega)^{-1}$ coincide with the poles of the meromorphic
continuation of the Schwartz kernel of the resolvent
$$
R_g(\omega):=(h^2(\Delta_g-(n-1)^2/4)-\omega^2)^{-1}:L^2(M)\to L^2(M),\quad
\Im\omega>0,
$$
to the entire $\mathbb C$, first constructed in~\cite{m-m} with improvements by~\cite{gui}~--
see~\cite[Theorem~5.1]{v2}.

We can now proceed similarly to~\S\ref{s:framework-schrodinger}, using that
the regions $\{\tilde x>\varepsilon_0\}$ are geodesically convex for $\varepsilon_0>0$
small enough (see for instance~\cite[Lemma~7.1]{qeefun}). Fix small $\varepsilon_0>0$, take any intervals
$$
[\alpha_0,\alpha_1]\Subset[\beta_0,\beta_1]\Subset [\beta'_0,\beta'_1]\subset\Omega\cap (0,\infty),
$$
and define
\begin{equation}
  \label{e:sets-ah}
\mathcal U':=\{\tilde x>\varepsilon_0/2,\ |\xi|_g\in (\beta'_0,\beta'_1)\},\quad
\mathcal U:=\{\tilde x>\varepsilon_0,\ |\xi|_g\in (\beta_0,\beta_1)\}.
\end{equation}
As in~\S\ref{s:framework-schrodinger}, assumptions~\eqref{a:u'}--\eqref{a:k-compact}
hold, with
\begin{equation}
  \label{e:p-ah}
p(x,\xi)=|\xi|_g.
\end{equation}
The operators $P$ and $\mathcal S(\omega)$ constructed in Lemma~\ref{l:resolution}
are given microlocally near $\mathcal U$ by
$$
P=\sqrt{h^2\Delta_g-(n-1)^2/4},\quad
S(\omega)=\sqrt{\sqrt{h^2\Delta_g-(n-1)^2/4}+\omega},
$$
with the square roots defined as in~\eqref{e:sqrt}.

Finally, for assumptions~\eqref{a:parametrix} and~\eqref{a:outgoing},
take $\mathcal Q\in\Psi^{\comp}(X)$
with $\WFh(\mathcal Q)\subset\mathcal U$ and
$$
\sigma(\mathcal Q)\geq 0\quad\text{everywhere};\quad
\sigma(\mathcal Q)>0\quad\text{on }p^{-1}([\alpha_0,\alpha_1])\cap \{\tilde x\geq 2\varepsilon_0\}.
$$
Then assumption~\eqref{a:parametrix} follows from~\cite[Theorem~4.8]{v2}.
To verify assumption~\eqref{a:outgoing}, we modify the proof of~\cite[Theorem~4.9]{v2}
as follows: assume that $f=f(h)\in \mathcal H_2$ has norm bounded polynomially in $h$
and put $u=\mathcal R'(\omega)f$, for $\omega=\omega(h)$ satisfying~\eqref{e:omega-region}. Assume also that $\WFh(f)\subset\mathcal U'$
and take $\rho'\in\WFh(u)\cap\mathcal U'$. We may assume that $P_2(\omega)$ is not elliptic at $\rho'$,
since otherwise $\rho'\in\WFh(f)$.
If $\gamma(t)$ is the bicharacteristic of
$\sigma(P_2(\omega))$ starting at $\rho'$, then (see~\cite[(3.32) and the end of~\S3.5]{v2})
either $\gamma(t)$ converges to the set $L_+\subset \partial \overline T^*X\cap\{\mu=0\}$ of radial points
as $t\to -\infty$, or $\mathcal Q$ is elliptic at $\gamma(-t_0)$ for some $t_0>0$. In the first case,
$\gamma(-t_0)\not\in\WFh(u)$ for $t_0>0$ large enough by the radial points argument~\cite[Proposition~4.5]{v2};
in the second case, by Proposition~\ref{l:elliptic} we see that if $\gamma(-t_0)\in\WFh(u)$,
then $\gamma(-t_0)\in\WFh(f)$. Combining this
with Proposition~\ref{l:microhyperbolic}, we see that there exists $t_1\geq 0$ such that
$\gamma(-t_1)\in\WFh(f)$. Since $\gamma(0),\gamma(-t_1)\in\mathcal U'$, and $\mathcal U'$
is convex with respect to the bicharacteristic flow of $\sigma(P_2(\omega))$
(the latter being just a rescaling of the geodesic flow pulled back by
a certain diffeomorphism), we see that $\gamma([-t_1,0])\subset\mathcal U'$.
Now, by~\eqref{e:same}, $\gamma([-t_1,0])$ is a flow line of $H_{p^2}$; therefore, 
for some $t\geq 0$, $e^{-tH_p}(\rho')\in\WFh(f)$, as required.

\section{\texorpdfstring{$r$}{r}-normally hyperbolic trapped sets}
\label{s:nh}

In this section, we state the dynamical assumptions on the flow near the
trapped set $K$, namely $r$-normal hyperbolicity,
and define the expansion rates $\nu_{\min},\nu_{\max}$ (\S\ref{s:dynamics}).
We next establish some properties of $r$-normally hyperbolic trapped sets:
existence of special defining functions~$\varphi_\pm$ of the
incoming/outgoing tails $\Gamma_\pm$ near $K$ (\S\ref{s:phi-pm}),
existence of the canonical projections $\pi_\pm$ from open subsets $\Gamma_\pm^\circ\subset\Gamma_\pm$
to $K$ and the canonical relation $\Lambda^\circ$ (\S\ref{s:projections}),
and regularity of solutions to the transport equations (\S\ref{s:transport}).

\subsection{Dynamical assumptions}
  \label{s:dynamics}
  
\newcounter{dynamic}

Let $\mathcal U\subset \mathcal U'$ be the open sets from~\S\ref{s:framework-assumptions},
and $p\in C^\infty(\mathcal U';\mathbb R)$ be the function defined in~\eqref{e:p}.
Consider also the incoming/outgoing tails $\Gamma_\pm\subset\overline {\mathcal U}$
and the trapped set $K=\Gamma_+\cap\Gamma_-$ defined in~\eqref{e:trapped-set}. We assume that,
for a large fixed integer $r$ depending only on the dimension $n$ (see Figure~\ref{f:basic-dynamics}(a)),
\begin{enumerate}
\item \label{aa:basic}
$\Gamma_\pm$ are equal to the intersections of $\overline{\mathcal U}$ with 
codimension 1 orientable $C^r$ submanifolds of $T^*X$;
\item \label{aa:symplectic}
$\Gamma_+$ and $\Gamma_-$ intersect transversely, and the symplectic form $\sigma_S$
is nondegenerate on $TK$; that is, $K$ extends to a symplectic submanifold of $T^*X$ of codimension two.
\setcounter{dynamic}{\value{enumi}}
\end{enumerate}
Consider one-dimensional subbundles $\mathcal V_\pm\subset T\Gamma_\pm$
defined as the symplectic complements of $T\Gamma_\pm$ in $T_{\Gamma_\pm}(T^*X)$
(see Figure~\ref{f:basic-dynamics}(b));
they are invariant under the flow $e^{tH_p}$. By assumption~\eqref{aa:symplectic},
we have $T_K\Gamma_\pm=\mathcal V_\pm|_K\oplus TK$.
Define the minimal expansion rate in the normal direction, $\nu_{\min}$,
as the supremum of all $\nu$ for which there exists a constant $C$ such that
\begin{equation}
  \label{e:nu-min}
\sup_{\rho\in K}\|de^{\mp tH_p}(\rho)|_{\mathcal V_\pm}\|\leq Ce^{-\nu t},\quad t>0.
\end{equation}
Here $\|\cdot\|$ denotes the operator norm with respect to any smooth inner product
on the fibers of $T(T^*X)$. 
Similarly we define the maximal expansion rate in the normal direction,
$\nu_{\max}$, as the infimum of all $\nu$ for which there exists a constant $c>0$ such that
\begin{equation}
  \label{e:nu-max}
\inf_{\rho\in K}\|de^{\mp t H_p}(\rho)|_{\mathcal V_\pm}\|\geq ce^{-\nu t},\quad t>0.
\end{equation}
Since $e^{tH_p}$ preserves the symplectic form $\sigma_S$,
which is nondegenerate on $\mathcal V_+|_K\oplus\mathcal V_-|_K$, it is enough to require~\eqref{e:nu-min}
and~\eqref{e:nu-max} for a specific choice of sign.

\begin{figure}
\includegraphics{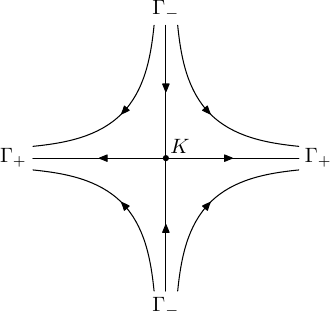}
\qquad
\includegraphics{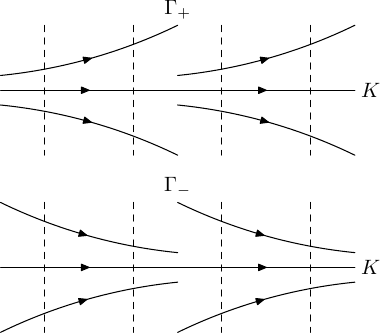}
\hbox to\hsize{\hss (a)\hss\hss (b)\hss}
\caption{(a) Dynamics of $e^{tH_p}$ in the directions transverse to the trapped set.
(b) Dynamics on $\Gamma_\pm$; the flow lines of $\mathcal V_\pm$ are dashed.}
\label{f:basic-dynamics}
\end{figure}

We assume \emph{$r$-normal hyperbolicity}:
\begin{enumerate}
  \setcounter{enumi}{\value{dynamic}}
\item \label{aa:r-nh} 
Let $\mu_{\max}$ be the maximal expansion rate of the flow along $K$,
defined as the infimum of all $\mu$ for which there exists a constant $C$ such that
\begin{equation}
  \label{e:mu-max}
\sup_{\rho\in K}\|de^{tH_p}(\rho)|_{TK}\|\leq Ce^{\mu|t|},\quad t\in \mathbb R.
\end{equation}
Then
\begin{equation}
  \label{e:r-nh}
\nu_{\min}>r\mu_{\max}.
\end{equation}
\end{enumerate}
Assumption~\eqref{aa:r-nh}, rather than a weaker assumption of
\emph{normal hyperbolicity} $\nu_{\min}>0$,
is needed for regularity of solutions
to the transport equations, see Lemma~\ref{l:ode} below.
The number $r$ depends on how many derivatives of the symbols
constructed below are needed for the semiclassical arguments to work.
In the proofs, we will often take $r=\infty$, keeping in mind that a large fixed $r$
is always enough.

\subsection{Stability}
  \label{s:stability}

We now briefly discuss stability of our dynamical assumptions under perturbations;
a more detailed presentation, with applications to general relativity, is
given in~\cite{thesis}. Assume that $p_s$, where $s\in\mathbb R$ varies in a neighborhood
of zero, is a family of real-valued functions on $\mathcal U'$ such that $p_0=p$
and $p_s$ is continuous at $s=0$ with values in $C^\infty(\mathcal U')$.
Assume moreover that conditions~\eqref{a:U} and~\eqref{a:k-compact} of~\S\ref{s:framework-assumptions}
are satisfied with $p$ replaced by any $p_s$. Here $\Gamma_\pm$ and $K$ are replaced
by the sets $\Gamma_\pm(s)$ and $K(s)$ defined using $p_s$ instead of $p$. We claim
that assumptions~\eqref{aa:basic}--\eqref{aa:r-nh} of~\S\ref{s:dynamics} are satisfied
for $p_s,\Gamma_\pm(s),K(s)$ when $s$ is small enough.

We use the work of Hirsch--Pugh--Shub~\cite{HPS} on stability of $r$-normally
hyperbolic invariant manifolds. Assumptions~\eqref{aa:basic}--\eqref{aa:r-nh}
imply that the flow $e^{tH_p}$ is eventually absolutely $r$-normally hyperbolic
on $K$ in the sense of~\cite[Definition~4]{HPS}. Then by~\cite[Theorem~4.1]{HPS},
for $s$ small enough, $\Gamma_\pm(s)$ and $K(s)$ are $C^r$ submanifolds
of $T^*X$, which converge to $\Gamma_\pm$ and $K$ in $C^r$ as $s\to 0$.
It follows immediately that conditions~\eqref{aa:basic} and~\eqref{aa:symplectic}
are satisfied for small $s$.

To see that condition~\eqref{aa:r-nh} is satisfied for small $s$,
as well as stability of the pinching condition~\eqref{e:pinching}
under perturbations, it suffices to show that, with
$\nu_{\min}(s),\nu_{\max}(s),\mu_{\max}(s)$ defined using
$e^{tH_{p_s}},\Gamma_\pm(s),K(s)$,
\begin{gather}
  \label{e:stab-1}
\liminf_{s\to 0}\nu_{\min}(s)\geq \nu_{\min},\\
  \label{e:stab-2}
\limsup_{s\to 0}\nu_{\max}(s)\leq \nu_{\max},\\
  \label{e:stab-3}
\limsup_{s\to 0}\mu_{\max}(s)\leq \mu_{\max}.
\end{gather}
We show~\eqref{e:stab-1}; the other two inequalities are proved similarly.
Fix a smooth metric on the fibers of $T(T^*X)$.
Take arbitrary $\varepsilon>0$, then for $T>0$ large enough,
we have
$$
\sup_{\rho\in K}\|de^{\mp TH_p}(\rho)|_{\mathcal V_\pm}\|\leq e^{-(\nu_{\min}-\varepsilon)T}.
$$
Fix $T$; since $p_s$, $\Gamma_\pm(s)$, $K(s)$, and the corresponding subbundles $\mathcal V_\pm(s)$
depend continuously on $s$ at $s=0$, we have for $s$ small enough,
$$
\sup_{\rho\in K(s)}\|de^{\mp TH_{p_s}}(\rho)|_{\mathcal V_\pm(s)}\|\leq e^{-(\nu_{\min}-\varepsilon/2)T}.
$$
Since $e^{tH_{p_s}}$ is a one-parameter group of diffeomorphisms, we get
$$
\sup_{\rho\in K(s)}\|de^{\mp tH_{p_s}}(\rho)|_{\mathcal V_\pm(s)}\|\leq Ce^{-(\nu_{\min}-\varepsilon/2)t},\quad
t\geq 0;
$$
therefore, $\nu_{\min}(s)\geq \nu_{\min}-\varepsilon/2$ for $s$ small enough and~\eqref{e:stab-1} follows.

\subsection{Adapted defining functions}
  \label{s:phi-pm}

In this section, we construct special defining functions $\varphi_\pm$
of $\Gamma_\pm$ near $K$. We will assume below that $\Gamma_\pm$
are smooth; however, if $\Gamma_\pm$ are $C^r$ with $r\geq 1$,
we can still obtain $\varphi_\pm\in C^r$. A similar construction
can be found in~\cite[Lemma~4.1]{w-z}.
\begin{lemm}
  \label{l:phi-pm}
Fix $\varepsilon>0$.%
\footnote{The parameter $\varepsilon$ is fixed in Theorem~\ref{t:gaps};
it is also taken small enough for the results
of~\S\ref{s:transport} to hold.}
Then
there exist smooth functions $\varphi_\pm$, defined in a neighborhood
of $K$ in $\mathcal U'$, such that for $\delta>0$ small enough, the set
\begin{equation}
  \label{e:u-delta}
U_\delta:=\overline{\mathcal U}\cap \{|\varphi_+|\leq\delta,\ |\varphi_-|\leq\delta\},
\end{equation}
is a compact subset of $\mathcal U$ when intersected with $p^{-1}([\alpha_0,\alpha_1])$, and:
\begin{enumerate}
\item \label{c:defining}
$\Gamma_\pm\cap U_\delta=\{\varphi_\pm=0\}\cap U_\delta$, and $d\varphi_\pm\neq 0$
on $U_\delta$;
\item \label{c:c-pm}
$H_p\varphi_\pm=\mp c_\pm\varphi_\pm$ on $U_\delta$, where $c_\pm$ are smooth functions
on $U_\delta$ and, with $\nu_{\min},\nu_{\max}$
defined in~\eqref{e:nu-min}, \eqref{e:nu-max},
\begin{equation}
  \label{e:nu-bound}
\nu_{\min}-\varepsilon<c_\pm<\nu_{\max}+\varepsilon\quad\text{on }U_\delta;
\end{equation}
\item the Hamiltonian field $H_{\varphi_\pm}$ spans the subbundle $\mathcal V_\pm$
on $\Gamma_\pm\cap U_\delta$ defined before~\eqref{e:nu-min};
\item \label{c:poisson}
$\{\varphi_+,\varphi_-\}>0$ on $U_\delta$;
\item \label{c:convex}
$U_\delta$ is convex, namely if $\gamma(t)$, $0\leq t\leq T$,
is a Hamiltonian flow line of $p$ in $\overline{\mathcal U}$ and
$\gamma(0),\gamma(T)\in U_\delta$, then $\gamma([0,T])\subset U_\delta$.
\end{enumerate}
\end{lemm}
\begin{proof} Since $\Gamma_\pm$ are orientable, there exist defining functions $\tilde\varphi_\pm$
of $\Gamma_\pm$ near $K$; that is, $\tilde\varphi_\pm$ are smooth, defined in
some neighborhood $U$ of $K$, and
$d\tilde\varphi_\pm\neq 0$ on $U$ and $\Gamma_\pm\cap U=\overline{\mathcal U}\cap \{\tilde\varphi_\pm=0\}$.
Since $K$ is symplectic, by changing the sign of $\tilde\varphi_-$ if necessary, we can moreover
assume that $\{\tilde\varphi_+,\tilde\varphi_-\}>0$ on $K$.

Since $e^{tH_p}(\Gamma_\pm)\subset\Gamma_\pm$ for $\mp t\geq 0$, we have $H_p\tilde\varphi_\pm=0$ on $\Gamma_\pm$;
therefore,
$$
H_p\tilde\varphi_\pm=\mp \tilde c_\pm\tilde\varphi_\pm,
$$
where $\tilde c_\pm$ are smooth functions on $U$. The functions $\tilde c_\pm$
control how fast $\tilde\varphi_\pm$ decays along the flow as $t\to\pm\infty$. The constants
$\nu_{\min}$ and $\nu_{\max}$ control the average decay rate; to construct $\varphi_\pm$,
we will modify $\tilde\varphi_\pm$ by averaging along the flow for a large time.

For any $\rho\in \Gamma_\pm\cap U$, the kernel of $d\tilde\varphi_\pm(\rho)$ is equal to $T_\rho\Gamma_\pm$;
therefore, the Hamiltonian fields $H_{\tilde\varphi_\pm}$ span $\mathcal V_\pm$ on $\Gamma_\pm\cap U$. We then see from
the definitions~\eqref{e:nu-min}, \eqref{e:nu-max} of $\nu_{\min},\nu_{\max}$ that there exists a constant $C$ such that,
with $(e^{\mp tH_p})_*H_{\tilde\varphi_\pm}\in\mathcal V_\pm$ denoting the push-forward of the vector
field $H_{\tilde\varphi_\pm}$ by the diffeomorphism $e^{\mp tH_p}$,
$$
C^{-1}e^{-(\nu_{\max}+\varepsilon/2)t}\leq
{(e^{\mp tH_p})_*H_{\tilde\varphi_\pm}\over H_{\tilde\varphi_\pm}}\leq Ce^{-(\nu_{\min}-\varepsilon/2)t}\quad\text{on }
K,\quad
t\geq 0.
$$
Now, we calculate on $K$,
$$
\begin{gathered}
\partial_t ((e^{\mp tH_p})_* H_{\tilde\varphi_\pm})
=\pm (e^{\mp tH_p})_*[H_p,H_{\tilde\varphi_\pm}]\\
=-(e^{\mp tH_p})_*H_{\tilde c_\pm\tilde\varphi_\pm}
=-(\tilde c_\pm\circ e^{\pm tH_p})(e^{\mp tH_p})_* H_{\tilde\varphi_\pm}.
\end{gathered}
$$
Combining these two facts, we get for $T>0$ large enough,
$$
\nu_{\min}-\varepsilon<\langle \tilde c_\pm\rangle_T<\nu_{\max}+\varepsilon\quad\text{on }K,
$$
where $\langle \cdot\rangle_T$ stands for the ergodic average on $K$:
$$
\langle f\rangle_T:={1\over T}\int_0^T f\circ e^{tH_p}\,dt.
$$
Fix $T$. We now put
$\varphi_\pm:=e^{\mp f_\pm}\cdot\tilde\varphi_\pm$,
where $f_\pm$ are smooth functions on $U$ with
$$
f_\pm={1\over T}\int_0^T (T-t)\tilde c_\pm\circ e^{tH_p}\,dt\quad \text{on }K,
$$
so that
$H_p f_\pm=\langle \tilde c_\pm\rangle_T-\tilde c_\pm$ on $K$.
Then $\varphi_\pm$ satisfy conditions~\eqref{c:defining}--\eqref{c:poisson}, with
$$
c_\pm=\mp{H_p \varphi_\pm\over \varphi_\pm}=\langle\tilde c_\pm \rangle_T\in
(\nu_{\min}-\varepsilon,\nu_{\max}+\varepsilon)
$$
on $K$, and thus on $U_\delta$ for $\delta$ small enough.

To verify condition~\eqref{c:convex},
fix $\delta_0>0$ small enough so that $\pm H_p\varphi_\pm^2\leq 0$
on $U_{\delta_0}$. By Lemma~\ref{l:the-flow-2},
for $\delta$ small enough
depending on $\delta_0$, for each Hamiltonian flow line $\gamma(t)$, $0\leq t\leq T$,
of $p$ in $\overline{\mathcal U}$, if $\gamma(0),\gamma(T)\in U_\delta$,
then $\gamma([0,T])\subset U_{\delta_0}$. Since
$\pm \partial_t\varphi_\pm(\gamma(t))^2\leq 0$ for $0\leq t\leq T$
and $|\varphi_\pm(\gamma(t))|\leq \delta$ for $t=0,T$, we see
that $\gamma([0,T])\subset U_\delta$.
\end{proof}

\subsection{The canonical relation \texorpdfstring{$\Lambda^\circ$}{Lambda}}
  \label{s:projections}

We next construct the projections $\pi_\pm$ from subsets $\Gamma_\pm^\circ\subset\Gamma_\pm$
to $K$. Fix $\delta_0,\delta_1>0$ small enough so
that Lemma~\ref{l:phi-pm} holds with $\delta_0$ in place of $\delta$
and $K\cap p^{-1}([\alpha_0-\delta_1,\alpha_1+\delta_1])$ is a compact subset
of $\mathcal U$ (the latter is possible by assumption~\eqref{a:k-compact} in~\S\ref{s:framework-assumptions}),
consider the functions $\varphi_\pm$ from Lemma~\ref{l:phi-pm} and put
\begin{equation}
  \label{e:k-circ}
\Gamma_\pm^\circ:=\Gamma_\pm\cap p^{-1}(\alpha_0-\delta_1,\alpha_1+\delta_1)\cap \{|\varphi_\mp|<\delta_0\},\quad
K^\circ:=K\cap p^{-1}(\alpha_0-\delta_1,\alpha_1+\delta_1),
\end{equation}
so that $K^\circ=\Gamma_+^\circ\cap\Gamma_-^\circ$ and, for $\delta_0$ small enough, $\Gamma_\pm^\circ\subset\mathcal U$.
Note that, by part~\eqref{c:c-pm} of Lemma~\ref{l:phi-pm}, the level sets of $p$ on $\Gamma_\pm$ are invariant
under $H_{\varphi_\pm}$ and $e^{tH_p}(\Gamma_\pm^\circ)\subset\Gamma_\pm^\circ$
for $\mp t\geq 0$.

By part~\eqref{c:poisson} of Lemma~\ref{l:phi-pm}, $\Gamma_\pm^\circ$
is foliated by trajectories of $H_{\varphi_\pm}$ (or equivalently,
by trajectories of $\mathcal V_\pm$), moreover each trajectory intersects $K$ at a single point.
This defines projection maps
$$
\pi_\pm:\Gamma^\circ_\pm\to K^\circ,
$$
mapping
each trajectory to its intersection with $K$. The flow $e^{tH_p}$ preserves the
subbundle $\mathcal V_\pm$ generated by $H_{\varphi_\pm}$, therefore
\begin{equation}
  \label{e:pi-pm-commute}
\pi_\pm\circ e^{\mp tH_p}=e^{\mp tH_p}\circ\pi_\pm,\quad
t\geq 0.
\end{equation}
Now, define the $2n$-dimensional submanifold $\Lambda^\circ\subset T^*X\times T^*X$ by
\begin{equation}
  \label{e:the-Lambda}
\Lambda^\circ:=\{(\rho_-,\rho_+)\in\Gamma_-^\circ\times\Gamma_+^\circ\mid
\pi_-(\rho_-)=\pi_+(\rho_+)\}.
\end{equation}
We claim that $\Lambda^\circ$ is a canonical relation. Indeed, it
is enough to prove that $\sigma_S|_{T\Gamma^\circ_\pm}=\pi_\pm^* (\sigma_S|_{TK^\circ})$,
where $\sigma_S$ is the symplectic form on $T^*M$. This is true since
the Hamiltonian flow $e^{tH_{\varphi_\pm}}$ preserves $\sigma_S$ and
$\mathcal V_\pm|_K$ is symplectically orthogonal to $TK$.

\subsection{The transport equations}
  \label{s:transport}

Finally, we use $r$-normal hyperbolicity to establish existence of solutions to the transport
equations, needed in the construction of the projector $\Pi$
in~\S\ref{s:construction-1}. We start by estimating higher derivatives of the flow.
Take $\delta_0,\Gamma_\pm^\circ,K^\circ$ from~\S\ref{s:projections} and
identify $\Gamma_\pm^\circ\sim K^\circ\times (-\delta_0,\delta_0)$ by the map
\begin{equation}
  \label{e:identify}
\rho_\pm\in \Gamma_\pm^\circ\mapsto (\pi_\pm(\rho_\pm),\varphi_\mp(\rho_\pm)).
\end{equation}
Denote elements of $K^\circ\times (-\delta_0,\delta_0)$ by $(\theta,s)$ and the flow $e^{tH_p}$ on
$\Gamma_\pm^\circ$, $\mp t\geq 0$, by (recall~\eqref{e:pi-pm-commute})
$$
e^{tH_p}:(\theta,s)\mapsto (e^{tH_p}(\theta),\psi^t_\pm(\theta,s)).
$$
Note that $\psi^t_\pm(\theta,0)=0$.
We have the following estimate on higher derivatives of the flow on $K^\circ$ (in any fixed coordinate system),
see for example~\cite[Lemma~C.1]{qeefun}
(which is stated for geodesic flows, but the proof applies to any smooth flow):
\begin{equation}
  \label{e:k-ders}
\sup_{\theta\in K^\circ}|\partial^\alpha_\theta e^{tH_p}(\theta)|\leq C_\alpha e^{(|\alpha|\mu_{\max}+\tilde\varepsilon)|t|},\quad
t\in\mathbb R.
\end{equation}
Here $\mu_{\max}$ is defined by~\eqref{e:mu-max},
$\tilde\varepsilon>0$ is any fixed constant, and $C_\alpha$ depends on $\tilde\varepsilon$.
We choose $\tilde\varepsilon$ small enough in~\eqref{e:r-nh-2} below and the constant $\varepsilon>0$
in Lemma~\ref{l:phi-pm} is small depending on $\tilde\varepsilon$.

Next, we estimate the derivatives of $\psi^t_\pm$. We have, with $c_\pm$ defined
in part~\eqref{c:c-pm} of Lemma~\ref{l:phi-pm},
$$
\partial_t\psi^t_\pm(\theta,s)=\pm c_\mp(e^{tH_p}(\theta),\psi^t_\pm(\theta,s))\psi^t_\pm(\theta,s).
$$
Then
$$
\partial_t (\partial^k_s\partial^\alpha_\theta\psi^t_\pm(\theta,s))
=\pm c_\mp(e^{tH_p}(\theta),0)\partial^k_s\partial^\alpha_\theta\psi^t_\pm(\theta,s)
+\dots,
$$
where $\dots$ is a linear combination, with uniformly bounded variable coefficients depending
on the derivatives of $c_\mp$, of expressions of the form
$$
\partial^{\beta_1}_\theta e^{tH_p}(\theta)\cdots\partial_\theta^{\beta_m}e^{tH_p}(\theta)\,
\partial^{\gamma_1}_\theta\partial^{k_1}_s\psi^t_\pm(\theta,s)\cdots\partial_\theta^{\gamma_l}\partial^{k_l}_s\psi^t_\pm(\theta,s),
$$
where $\beta_1+\dots+\beta_m+\gamma_1+\dots+\gamma_l=\alpha$,
$k_1+\dots+k_l=k$, and $|\beta_j|,|\gamma_j|+k_j>0$.
Moreover, if $l=0$ or $l+m=1$,
then the corresponding coefficient is a bounded multiple of $\psi^t_\pm(\theta,s)$.
It now follows by induction from~\eqref{e:nu-bound} that
\begin{equation}
  \label{e:s-ders}
\sup_{\theta\in K^\circ,\, |s|<\delta_0}|\partial^k_s\partial^\alpha_\theta \psi^{\mp t}_\pm(\theta,s)|\leq C_{\alpha k}e^{(|\alpha|\mu_{\max}-\nu_{\min}+\tilde\varepsilon)t},\quad t\geq 0.
\end{equation}

We can now prove the following
\begin{lemm}
  \label{l:ode}
Assume that~\eqref{e:r-nh} is satisfied, with some integer $r>0$.
Let $f\in C^{r+1}(\Gamma_\pm^\circ)$ be such that $f|_K=0$. Then
there exists unique solution $a\in C^r(\Gamma_\pm^\circ)$ to the equation
\begin{equation}
  \label{e:transport}
H_p a=f,\quad
a|_{K^\circ}=0.
\end{equation}
\end{lemm}
\begin{proof}
Using~\eqref{e:r-nh}, choose $\tilde\varepsilon>0$ so that
\begin{equation}
  \label{e:r-nh-2}
r\mu_{\max}-\nu_{\min}+\tilde\varepsilon<0.
\end{equation}
Any solution to~\eqref{e:transport} satisfies for each $T>0$,
$$
a=a\circ e^{\mp TH_p}\pm\int_0^T f\circ e^{\mp tH_p}\,dt.
$$
Since $a|_{K^\circ}=0$, by letting $T\to +\infty$ we see that
the unique solution to~\eqref{e:transport} is
\begin{equation}
  \label{e:transport-sol}
a=\pm\int_0^\infty f\circ e^{\mp tH_p}\,dt.
\end{equation}
The integral~\eqref{e:transport-sol} converges exponentially, as
$$
|f\circ e^{\mp tH_p}(\theta,s)|\leq C|\psi^{\mp t}_\pm(\theta,s)|\leq Ce^{-(\nu_{\min}-\varepsilon)t}.
$$
To show that $a\in C^r$, it suffices to prove that when
$|\alpha|+k\leq r$, the integral
$$
\int_0^\infty \partial^k_s\partial^\alpha_\theta(f\circ e^{\mp tH_p})\,dt
$$
converges uniformly in $s,\theta$. Given~\eqref{e:r-nh-2}, it is enough to show that
\begin{equation}
  \label{e:tr-int}
\sup_{\theta,s}|\partial^k_s\partial^\alpha_\theta (f\circ e^{\mp tH_p})(\theta,s)|\leq C_{\alpha k}e^{(|\alpha|\mu_{\max}
-\nu_{\min}+\tilde\varepsilon)t},\quad t>0.
\end{equation}
To see~\eqref{e:tr-int}, we use the chain rule to estimate the left-hand side by a sum
of terms of the form
$$
\partial^m_\theta\partial^l_s f(e^{\mp tH_p}(\theta,s))
\partial_\theta^{\beta_1}e^{\mp tH_p}(\theta)\cdots\partial_\theta^{\beta_m}e^{\mp tH_p}(\theta)
\partial^{\gamma_1}_\theta\partial^{k_1}_s\psi^{\mp t}_\pm(\theta,s)\cdots\partial_\theta^{\gamma_l}\partial^{k_l}_s\psi^{\mp t}_\pm(\theta,s)
$$
where $\beta_1+\dots+\beta_m+\gamma_1+\dots+\gamma_l=\alpha$, $k_1+\dots+k_l=k$, and
$|\beta_j|,|\gamma_j|+k_j>0$. For $l=0$, we have $|\partial^m_\theta f\circ e^{\mp tH_p}|=\mathcal O(e^{-(\nu_{\min}-\varepsilon)t})$ and~\eqref{e:tr-int} follows from~\eqref{e:k-ders}.
For $l>0$, \eqref{e:tr-int} follows from~\eqref{e:k-ders} and~\eqref{e:s-ders}.
\end{proof}

\section{Calculus of microlocal projectors}
  \label{s:calculus}

In this section, we develop tools for handling Fourier integral operators
associated to the canonical relation~$\Lambda^\circ$ introduced in~\S\ref{s:projections}.
We will not use theoperator $P$ or the global dynamics of the
flow $e^{tH_p}$; we will only assume that $X$ is an $n$-dimensional manifold and
\begin{itemize}
\item $\Gamma_\pm^\circ\subset T^*X$ are smooth orientable hypersurfaces;
\item $\Gamma_\pm^\circ$ intersect transversely and $K^\circ:=\Gamma_+^\circ\cap\Gamma_-^\circ$
is symplectic;
\item if $\mathcal V_\pm\subset T\Gamma_\pm^\circ$ is the symplectic complement of $T\Gamma_\pm^\circ$ in $T(T^*X)$, then
each maximally extended flow line of $\mathcal V_\pm$ on $\Gamma_\pm^\circ$ intersects $K^\circ$ at precisely
one point, giving rise to the projection maps $\pi_\pm:\Gamma_\pm^\circ\to K^\circ$;
\item the canonical relation $\Lambda^\circ\subset T^*(X\times X)$ is defined by
$$
\Lambda^\circ=\{(\rho_-,\rho_+)\in\Gamma_-^\circ\times\Gamma_+^\circ\mid
\pi_-(\rho_-)=\pi_+(\rho_+)\};
$$
\item the projections $\tilde\pi_\pm:\Lambda^\circ\to\Gamma_\pm^\circ$
are defined by
\begin{equation}
  \label{e:tilde-pi}
\tilde\pi_\pm(\rho_-,\rho_+)=\rho_\pm.
\end{equation}
\end{itemize}
If we only consider a bounded number of terms in the asymptotic expansions of the studied
symbols, and require existence of a fixed number of derivatives of these symbols,
then the smoothness requirement above can be replaced by $C^r$ for $r$ large enough depending
only on $n$.

We will study the operators in the class $\II$ considered in \S\ref{s:prelim-fio}.
The antiderivative on $\Lambda^\circ$ (see~\S\ref{s:prelim-fio}) is fixed so
that it vanishes on the image of the embedding
\begin{equation}
  \label{e:j-k}
j_K:K^\circ\to\Lambda^\circ,\quad
j_K(\rho)=(\rho,\rho);
\end{equation}
this is possible since $j_K^*(\eta\,dy-\xi\,dx)=0$ and the image of $j_K$
is a deformation retract of $\Lambda^\circ$.

We are particularly interested in defining invariantly the
principal symbol $\sigma_\Lambda(A)$ of an operator $A\in\II$. This could be done
using the global theory of Fourier integral operators; we take instead a more
direct approach based on the model case studied in~\S\ref{s:model}.
The principal symbols on a neighborhood $\widetilde\Lambda$ of a
compact subset $\widehat K\subset K^\circ$
are defined as sections of certain vector
bundles in~\S\ref{s:general}.

We are also interested in the symbol
of a product of two operators in $\II$. Note
that such a product lies again in $\II$, since
$\Lambda^\circ$ satisfies the transversality condition with itself
and, with the composition defined as in~\eqref{e:comp-fio}, $\Lambda^\circ\circ\Lambda^\circ=\Lambda^\circ$.
To study the principal symbol of the product, we again use the model case~-- see Proposition~\ref{l:calculus}.

Next, in~\S\ref{s:idempotents}, we study idempotents in $\II$,
microlocally near $\widehat K$, proving technical lemmas need in the construction
of the microlocal projector $\Pi$ in~\S\ref{s:global-construction}.
Finally, in~\S\ref{s:ideals}, we consider left and right ideals of pseudodifferential
operators annihilating a microlocal idempotent, which are key for proving resolvent
estimates in~\S\ref{s:resolvent-bounds}.

\subsection{Model case}
  \label{s:model}

We start with the model case
\begin{equation}
  \label{e:model-gamma}
X:=\mathbb R^n,\quad
\Gamma^0_+:=\{\xi_n=0\},\quad
\Gamma^0_-:=\{x_n=0\}.
\end{equation}
Then $K^0=\{x_n=\xi_n=0\}$ is canonically diffeomorphic to $T^*\mathbb R^{n-1}$.
If we denote elements of $\mathbb R^{2n}\simeq T^*\mathbb R^n$ by $(x',x_n,\xi',\xi_n)$, with
$x',\xi'\in\mathbb R^{n-1}$, then the projection maps $\pi_\pm:\Gamma_\pm^0\to K^0$ take the form
$$
\pi_+(x,\xi',0)=(x',0,\xi',0),\quad
\pi_-(x',0,\xi)=(x',0,\xi',0),
$$
and the map
\begin{equation}
  \label{e:phi}
\phi:(x,\xi)\mapsto (x',0,\xi;x,\xi',0)\in T^*(\mathbb R^n\times\mathbb R^n)
\end{equation}
gives a diffeomorphism of $\mathbb R^{2n}$ onto the corresponding canonical relation $\Lambda^0$.

\smallsection{Basic calculus}
For a Schwartz function
$a(x,\xi)\in\mathscr S(\mathbb R^{2n})$, define its $\Lambda^0$-quantization
$\Op_h^\Lambda(a):\mathscr S'(\mathbb R^n)\to \mathscr S(\mathbb R^n)$
by the formula
\begin{equation}
  \label{e:lambda-quant}
\Op_h^\Lambda(a)u(x)=(2\pi h)^{-n}
\int_{\mathbb R^{2n}}e^{{i\over h}(x'\cdot\xi'-y\cdot\xi)}a(x,\xi)u(y)\,dyd\xi.
\end{equation}
The operator $\Op_h^\Lambda(a)$ will be a Fourier integral operator associated to $\Lambda^0$,
see below for details.
We also use the standard quantization for pseudodifferential operators~\cite[\S 4.1.1]{e-z},
where $a(x,\xi;h)\in C^\infty(\mathbb R^{2n})$ and all derivatives of $a$ are bounded uniformly
in $h$ by a fixed power of $1+|x|^2+|\xi|^2$:
\begin{equation}
  \label{e:op-h}
\Op_h(a)u(x)=(2\pi h)^{-n}\int_{\mathbb R^{2n}}e^{{i\over h}(x-y)\cdot\xi}a(x,\xi)u(y)\,dyd\xi.
\end{equation}
The symbol $a$ can be extracted from $\Op_h^\Lambda(a)$
or $\Op_h(a)$ by the following oscillatory testing formulas, see~\cite[Theorem~4.19]{e-z}:
\begin{gather}
  \label{e:osc-test}
\Op_h^\Lambda(a)(e^{{i\over h}x\cdot\xi})=e^{{i\over h}x'\cdot\xi'}a(x,\xi),\quad
\xi\in\mathbb R^n,\\
  \label{e:osc-test-2}
\Op_h(a)(e^{{i\over h}x\cdot\xi})=e^{{i\over h}x\cdot\xi}a(x,\xi),\quad
\xi\in\mathbb R^n.
\end{gather}
From here, using stationary phase expansions similarly to~\cite[Theorems~4.11 and~4.12]{e-z}, we get
(where the symbols quantized by $\Op_h^\Lambda$ are Schwartz)
\begin{gather}
  \label{e:calc-model-1}
\Op_h^\Lambda(a)\Op_h^\Lambda(b)=\Op_h^\Lambda(a\#^\Lambda b),\\
  \label{e:calc-model-2}
\Op_h^\Lambda(a)\Op_h(b)=\Op_h^\Lambda(a_{\#b}),\\
  \label{e:calc-model-3}
\Op_h(b)\Op_h^\Lambda(a)=\Op_h^\Lambda(a_{b\#}),
\end{gather}
where the symbols $a\#^\Lambda b,a_{\#b},a_{b\#}\in \mathscr S(\mathbb R^{2n})$ have
asymptotic expansions
\begin{gather}
\label{e:exp-model-1}
a\#^\Lambda b(x,\xi)
\sim\sum_{\alpha} {(-ih)^{|\alpha|}\over\alpha!}\partial_\xi^\alpha a(x,\xi',0)\partial_x^\alpha b(x',0,\xi),\\
\label{e:exp-model-2}
a_{\#b}(x,\xi)\sim\sum_{\alpha}{(-ih)^{|\alpha|}\over\alpha!} \partial_\xi^\alpha a(x,\xi)
\partial_x^\alpha b(x',0,\xi),\\
\label{e:exp-model-3}
a_{b\#}(x,\xi)\sim \sum_{\alpha}{(-ih)^{|\alpha|}\over\alpha!} \partial_\xi^\alpha b(x,\xi',0)\partial_x^\alpha a(x,\xi).
\end{gather}
Finally, the operators $\Op_h^\Lambda(a)$ are bounded $L^2\to L^2$ with norm $\mathcal O(h^{-1/2})$:
\begin{prop}
  \label{l:bund}
If $a\in\mathscr S(\mathbb R^{2n})$, then there exists a constant $C$ such that
$$
\|\Op_h^\Lambda(a)\|_{L^2(\mathbb R^n)\to L^2(\mathbb R^n)}\leq Ch^{-1/2}.
$$
\end{prop}
\begin{proof}
Define the semiclassical Fourier transform
$$
\hat u(\xi):=(2\pi h)^{-n/2}\int_{\mathbb R^n} e^{-{i\over h}y\cdot\xi}u(y)\,dy,
$$
then $\|\hat u\|_{L^2}=\|u\|_{L^2}$ and
$$
\Op_h^\Lambda(a)u(x)=(2\pi h)^{-1/2}\int_{\mathbb R} v(x,\xi_n)\,d\xi_n,
$$
where
$$
v(x,\xi_n):=(2\pi h)^{-(n-1)/2}\int_{\mathbb R^{n-1}} e^{{i\over h}x'\cdot\xi'}a(x,\xi',\xi_n)\hat u(\xi',\xi_n)\,d\xi'.
$$
Using the $L^2$-boundedness of pseudodifferential operators on $\mathbb R^{n-1}$, we see that
for each $(x_n,\xi_n)\in \mathbb R^2$,
$$
\|v(\cdot,x_n,\xi_n)\|_{L^2_{x'}}\leq F(x_n,\xi_n)\|\hat u(\cdot,\xi_n)\|_{L^2_{\xi'}},
$$
where $F(x_n,\xi_n)$ is bounded by a certain $\mathscr S(\mathbb R^{2n-2})$ seminorm of $a(\cdot,x_n,\cdot,\xi_n)$.
Then $F$ is rapidly decaying on $\mathbb R^2$ and for any $N$,
$$
\|v(\cdot,\xi_n)\|_{L^2_x}\leq C\langle \xi_n\rangle^{-N}\|\hat u(\cdot,\xi_n)\|_{L^2_{\xi'}}.
$$
Therefore,
$$
\|\Op_h^\Lambda(a)u(x)\|_{L^2}\leq Ch^{-1/2}\int_{\mathbb R}\|v(\cdot,\xi_n)\|_{L^2_x}\,d\xi_n
\leq Ch^{-1/2}\|u\|_{L^2}
$$
as required.
\end{proof}

\smallsection{Microlocal properties}
For $a\in\mathscr S(\mathbb R^{2n})$,
the operator $\Op_h^\Lambda(a)$ is $h$-tempered as defined in Section~\ref{s:prelim-basics}.
Moreover, the following analog of~\eqref{e:lagrangian-wf} follows from~\eqref{e:calc-model-2}
and~\eqref{e:calc-model-3}:
\begin{equation}
  \label{e:model-wf}
\WFh(\Op_h^\Lambda(a))\subset \phi(\supp a)\subset\Lambda^0,
\end{equation}
with $\phi$ defined by~\eqref{e:phi}.

For $a\in C_0^\infty(\mathbb R^{2n})$, we use~\eqref{e:lagrangian}
to check that $\Op_h^\Lambda(a)$ is,
modulo an $\mathcal O(h^\infty)_{\mathscr S'\to\mathscr S}$ remainder, 
a Fourier integral operator in the class
$I_{\comp}(\Lambda^0)$ defined in \S\ref{s:prelim-basics}.

We will also use the operator $\Op_h^\Lambda(1):C^\infty(\mathbb R^n)\to C^\infty(\mathbb R^n)$
defined by
\begin{equation}
  \label{e:model-projector}
\Op_h^\Lambda(1)f(x)=f(x',0),\quad
f\in C^\infty(\mathbb R^n).
\end{equation}
Since~\eqref{e:lambda-quant} was defined only for Schwartz symbols, we understand~\eqref{e:model-projector}
as follows: if $a\in C_0^\infty(\mathbb R^{2n})$ is equal to 1 near some open set $U\subset \mathbb R^{2n}$,
then the operator $\Op_h^\Lambda(1)$ defined in~\eqref{e:model-projector} is equal
to the operator $\Op_h^\Lambda(a)$ defined in~\eqref{e:lambda-quant}, microlocally
near $\phi(U)\subset T^*(\mathbb R^n\times\mathbb R^n)$. Moreover,
$\WFh(\Op_h^\Lambda(1))\cap T^*(\mathbb R^n\times\mathbb R^n)\subset\Lambda^0$.
To see this, it is enough to note that for $a\in C_0^\infty(\mathbb R^{2n})$ and
$\chi\in C_0^\infty(\mathbb R^n)$, we have
$\chi\Op_h^\Lambda(1)\Op_h(a)=\Op_h^\Lambda(\tilde a)$, where
$\tilde a(x,\xi)=\chi(x)a(x',0,\xi)\in C_0^\infty(\mathbb R^{2n})$ and
$\Op_h^\Lambda(\tilde a)$ is defined using~\eqref{e:lambda-quant}.

\smallsection{Canonical transformations}
We now study how $\Op^\Lambda_h(a)$ changes under quantized canonical transformations preserving
its canonical relation (see~\S\ref{s:prelim-fio}). Let $U,V\subset\mathbb R^{2n}$ be two bounded open sets and
$\varkappa:U\to V$ a symplectomorphism such that
$$
\varkappa(\Gamma^0_\pm\cap U)=\Gamma^0_\pm\cap V,
$$
with $\Gamma^0_\pm$ given by~\eqref{e:model-gamma}.
We further assume that
for each $(x',\xi')\in T^*\mathbb R^{n-1}$,
the sets $\{x_n\mid (x',x_n,\xi',0)\in U\}$ and
$\{\xi_n\mid (x',0,\xi',\xi_n)\in U\}$, and the corresponding
sets for $V$, are either empty or intervals containing zero,
so that the maps $\pi_\pm:U\cap\Gamma_\pm^0\to U\cap K^0$ are well-defined.
Since
$\varkappa$ preserves the subbundles $\mathcal V_\pm$,
it commutes with the maps $\pi_\pm$ and thus preserves
$\Lambda^0$; using the map $\phi$ from~\eqref{e:phi}, we define
the open sets $\widehat U,\widehat V\subset\mathbb R^{2n}$ and
the diffeomorphism $\widehat\varkappa:\widehat U\to\widehat V$ by
$$
\begin{gathered}
\widehat U:=\phi^{-1}(U\times U),\
\widehat V:=\phi^{-1}(V\times V),\quad
\phi\circ\widehat\varkappa=\varkappa\circ\phi.
\end{gathered}
$$

\begin{prop}
  \label{l:change-of-variables}
Let $B,B':C^\infty(\mathbb R^{n})\to C_0^\infty(\mathbb R^n)$ be two
compactly microlocalized Fourier integral operators associated to
$\varkappa$ and $\varkappa^{-1}$, respectively,%
\footnote{The choice of antiderivative (see~\S\ref{s:prelim-fio}) is irrelevant here,
since the phase factor in $B$ resulting from choosing another antiderivative will be cancelled
by the phase factor in $B'$.}
such that
\begin{equation}
  \label{e:b-inv}
\begin{gathered}
BB'=1+\mathcal O(h^\infty)\quad\text{microlocally near }V',\\
B'B=1+\mathcal O(h^\infty)\quad\text{microlocally near }U',
\end{gathered}
\end{equation}
for some open $U'\Subset U$, $V'\Subset V$ such that $\varkappa(U')=V'$.
Then for each $a\in C_0^\infty(\widehat V)$,
$$
B'\Op^\Lambda_h(a)B=\Op^\Lambda_h(a_\varkappa)+\mathcal O(h^\infty)_{\mathscr S'\to\mathscr S},
$$
for some classical symbol $a_\varkappa$ compactly supported in $\widehat U$, and
\begin{equation}
  \label{e:a-change}
a_\varkappa(x,\xi)=\gamma^+_\varkappa(x,\xi')\gamma^-_\varkappa(x',\xi) a(\widehat\varkappa(x,\xi))+\mathcal O(h)
\quad\text{on }\phi^{-1}(U'\times U'),
\end{equation}
where $\gamma^\pm_\varkappa$ are smooth functions on $U\cap\Gamma_\pm$ depending on $\varkappa,B,B'$
with $\gamma^\pm_\varkappa|_{K^0\cap U'}=1$.
\end{prop}
\begin{proof}
Assume first that $\varkappa$ has a generating function $S(x,\eta)$:
$$
\varkappa(x,\xi)=(y,\eta)\Longleftrightarrow \xi=\partial_x S(x,\eta),\
y=\partial_\eta S(x,\eta).
$$
If $\mathscr D_S\subset\mathbb R^{2n}$ is the domain of $S$,
then for each $(x',\eta')\in T^*\mathbb R^{n-1}$, the sets
$\{x_n\mid (x',x_n,\eta',0)\in\mathscr D_S\}$ and
$\{\eta_n\mid (x',0,\eta',\eta_n)\in\mathscr D_S\}$ are either empty or intervals containing zero.
Since $\varkappa$ preserves $\Gamma_\pm$, we find
$\partial_{\eta_n}S(x',0,\eta)=\partial_{x_n}S(x,\eta',0)=0$ and thus
\begin{equation}
  \label{e:S-eq}
S(x,\eta',0)=S(x',0,\eta)=S(x',0,\eta',0).
\end{equation}
We can write, modulo $\mathcal O(h^\infty)_{\mathscr S'\to\mathscr S}$ errors,
$$
\begin{gathered}
Bu(y)=(2\pi h)^{-n}\int e^{{i\over h}(y\cdot\eta-S(x,\eta))}b(x,\eta;h)u(x)\,dxd\eta,\\
B'u(x)=(2\pi h)^{-n}\int e^{{i\over h}(S(x,\eta)-y\cdot\eta)}b'(x,\eta;h)u(y)\,dyd\eta,
\end{gathered}
$$
where $b,b'$ are compactly supported classical symbols
and by~\eqref{e:b-inv} the principal symbols $b_0$ and $b'_0$ have to satisfy
for $(x,\xi)\in U'$,
\begin{equation}
  \label{e:b-eq}
b_0(x,\eta)b'_0(x,\eta)=|\det \partial^2_{x\eta}S(x,\eta)|.
\end{equation}
We can now use oscillatory testing~\eqref{e:osc-test} to get
$$
\begin{gathered}
a_\varkappa(x,\xi):=e^{-{i\over h}x'\cdot\xi'}B'\Op_h^\Lambda(a) B(e^{{i\over h}x\cdot\xi})\\
=
(2\pi h)^{-2n}\int_{\mathbb R^{4n}} e^{{i\over h}(
-x'\cdot\xi'+S(x,\tilde\eta)-y\cdot\tilde\eta+y'\cdot \eta'-S(\tilde x,\eta)+\tilde x\cdot\xi)}
b'(x,\tilde\eta;h)a(y,\eta)b(\tilde x,\eta;h)\,dyd\tilde \eta d\eta d\tilde x.
\end{gathered}
$$
We analyse this integral by the method of stationary phase; this will yield
that $a_\varkappa$ is a classical symbol in $h$, compactly supported in $\widehat U$
modulo an $\mathcal O(h^\infty)_{\mathscr S(\mathbb R^{2n})}$ error, and thus
$B'\Op_h^\Lambda(a)B=\Op_h^\Lambda(a_\varkappa)$.

The stationary points are given by
$$
\tilde\eta=(\eta',0),\
\tilde x=(x',0),\
(y,\eta)=\widehat\varkappa(x,\xi).
$$
The value of the phase at stationary points is zero due to~\eqref{e:S-eq}.
To compute the Hessian, we make the change of variables
$\tilde\eta=\check\eta+(\eta',0)$. We can then remove the variables $y,\check\eta$
and pass from the original Hessian
to $\partial^2_{\eta'\eta'} S(x,\eta',0)-\partial^2 S(x',0,\eta)$, where the first matrix
is padded with zeros. Since $\partial_{\eta_n}S(x',0,\eta)=0$, we have
$\partial^2_{\eta_n\eta_n}S=\partial^2_{\eta_n x'}S=\partial^2_{\eta_n \eta'}S=0$
at $(x',0,\eta)$, therefore we can remove the $x_n,\eta_n$ variables, with a multiplicand
of $(\partial^2_{x_n\eta_n}S(x',0,\eta))^2$ in the determinant. Next,
by~\eqref{e:S-eq} $\partial^2_{\eta'\eta'}S(x',0,\eta)=\partial^2_{\eta'\eta'}S(x,\eta',0)$;
therefore, the Hessian has signature zero and determinant
$$
(\partial^2_{x_n\eta_n}S(x',0,\eta)\det \partial^2_{x'\eta'}S(x',0,\eta))^2.
$$
Since $\partial^2_{x'\eta_n}S(x',0,\eta)=0$, this is equal to $(\det \partial^2_{x\eta}S(x',0,\eta))^2$.
Therefore, we get~\eqref{e:a-change} with
$$
\gamma^+_\varkappa(x,\xi')\gamma^-_\varkappa(x',\xi)={b_0'(x,\eta',0)b_0(x',0,\eta)\over |\det \partial^2_{x\eta}S(x',0,\eta)|}
={b'_0(x,\eta',0)\over b'_0(x',0,\eta)};
$$
here $(y,\eta)=\widehat\varkappa(x,\xi)$ and
the last equality follows from~\eqref{e:b-eq}. We then find
\begin{equation}
  \label{e:transition-coeff}
\gamma^+_\varkappa(x,\xi')=b'_0(x,\eta',0)/b'_0(x',0,\eta',0),\
\gamma^-_\varkappa(x',\xi)=b'_0(x',0,\eta',0)/b'_0(x',0,\eta).
\end{equation}

We now consider the case of general $\varkappa$. Using a partition of
unity for $a$, we may assume that the intersection $U\cap K^0$ is
arbitrary small.  We now represent $\varkappa$ as a product of several
canonical relations, each of which satisfies the conditions of this
Proposition and has a generating function; this will finish the proof.

First of all, consider a canonical transformation of the form
\begin{equation}
  \label{e:tan-can}
(x,\xi)\mapsto (y,\eta),\quad
(y',\eta')=\widetilde\varkappa(x',\xi'),\
(y_n,\eta_n)=(x_n,\xi_n),
\end{equation}
with $\widetilde\varkappa$ a canonical transformation on $T^*\mathbb R^{n-1}\simeq K^0$. 
We can write $\widetilde\varkappa$ locally as a product of canonical transformations
close to the identity, each of which has a generating function~-- see~\cite[Theorems~10.4 and~11.4]{e-z}.
If $\widetilde S(x',\eta')$ is a generating function for $\widetilde\varkappa$,
then $\widetilde S(x',\eta')+x_n\eta_n$ is a generating function for~\eqref{e:tan-can}.

Multiplying our $\varkappa$ by a transformation of the form~\eqref{e:tan-can}
with $\widetilde\varkappa=(\varkappa|_{K^0})^{-1}$, we reduce
to the case
$$
\varkappa(x',0,\xi',0)=(x',0,\xi',0)\quad\text{for }
(x',0,\xi',0)\in U\cap K^0.
$$
If $\varkappa(x,\xi)=(y(x,\xi),\eta(x,\xi))$, since $\varkappa$ commutes with $\pi_\pm$ we have
\begin{equation}
  \label{e:can-int}
\begin{gathered}
y'(x,\xi',0)=y'(x',0,\xi)=x',\\
\eta'(x,\xi',0)=\eta'(x',0,\xi)=\xi'.
\end{gathered}
\end{equation}
We now claim that $\varkappa$ has a generating function, if we shrink $U$
to be a small neighborhood of $U\cap (\Gamma^0_+\cup\Gamma^0_-)$ (which does not change
anything since $\Op_h^\Lambda(a)$ is microlocalized in $\Gamma^0_-\times\Gamma^0_+$).
For that, it is enough to show that the map
$$
\psi:(x,\xi)\mapsto (x,\eta(x,\xi))
$$
is a diffeomorphism from $U$ onto some open subset $\mathscr D_S\subset\mathbb R^{2n}$.

We first show that $\psi$ is a local diffeomorphism near
$\Gamma^0_\pm$; that is, the differential $\partial_\xi\eta$ is
nondegenerate on $\Gamma^0_\pm$. By~\eqref{e:can-int},
$\partial_{x',\xi'}(y',\eta')$ equals the identity on
$\Gamma^0_+\cup\Gamma^0_-$; moreover, on $\Gamma^0_+$ we have
$\partial_{x,\xi'}\eta_n=0$ and $\partial_{x_n}(y',\eta')=0$ and on
$\Gamma^0_-$, we have $\partial_{x',\xi}y_n=0$ and
$\partial_{\xi_n}(y',\eta')=0$.  It follows that on
$\Gamma^0_+\cup\Gamma^0_-$, $\det
\partial_\xi\eta=\partial_{\xi_n}\eta_n$ and since $\varkappa$ is a
diffeomorphism, $0\neq
\det\partial_{(x,\xi)}(y,\eta)=\partial_{x_n}y_n\cdot\partial_{\xi_n}\eta_n$,
yielding $\det \partial_\xi\eta\neq 0$.

It remains to note that $\psi$ is one-to-one on $\Gamma^0_+\cup\Gamma^0_-$,
which follows immediately from the identities
$\psi(x,\xi',0)=(x,\xi',0)$
and $\psi(x',0,\xi)=\varkappa(x',0,\xi)$.
\end{proof}

\subsection{General case}
  \label{s:general}

We now consider the case of general $\Gamma_\pm^\circ,K^\circ,\Lambda^\circ$, satisfying the
assumptions from the beginning of~\S\ref{s:calculus}.
We start by shrinking $\Gamma_\pm^\circ$
so that our setup can locally be conjugated to the model case of~\S\ref{s:model}.
(The set $\widehat K$ will be chosen in~\S\ref{s:construction-1}.)
\begin{prop}
  \label{l:reduce-model}
Let $\widehat K\subset K^\circ$ be compact. Then
there exist $\tilde\delta>0$ and
\begin{itemize}
\item a finite collection of open sets $U_i\subset T^*X$, such that
$$
\widehat K\subset \widetilde K:=\bigcup_i K_i,\quad
K_i:=K^\circ\cap U_i.
$$
\item symplectomorphisms $\varkappa_i$ defined in a neighborhood of $U_i$ 
and mapping $U_i$ onto
\begin{equation}
  \label{e:v-delta}
V_{\tilde\delta}:=\{|(x',\xi')|<\tilde\delta,\ |x_n|<\tilde\delta,\ |\xi_n|<\tilde\delta\}\subset T^* \mathbb R^n,
\end{equation}
such that, with $\Gamma_\pm^0$ defined in~\eqref{e:model-gamma},
$$
\varkappa_i(U_i\cap\Gamma_\pm^\circ)=V_{\tilde\delta}\cap\Gamma_\pm^0;
$$
\item compactly microlocalized Fourier integral operators
$$
B_i:C^\infty(X)\to C_0^\infty(\mathbb R^n),\
B'_i:C^\infty(\mathbb R^n)\to C_0^\infty(X),
$$
associated to $\varkappa_i$ and $\varkappa_i^{-1}$, respectively, such that
\begin{equation}
  \label{e:b-i-inv}
B_iB'_i=1\quad\text{near }V_{\tilde\delta},\
B'_iB_i=1\quad\text{near }U_i.
\end{equation}
\end{itemize}
\end{prop}
\begin{proof}
It is enough to show that each point $\rho\in K^\circ$ has a neighborhood
$U_\rho$ and a symplectomorphism $\varkappa_\rho:U_\rho\to V_\rho\subset T^*\mathbb R^n$
such that $\varkappa_\rho(U_\rho\cap\Gamma_\pm^\circ)=V_\rho\cap\Gamma_\pm^0$; see for
example~\cite[Theorem~11.5]{e-z} for how to construct the operators $B_i,B'_i$ locally quantizing
the canonical transformations $\varkappa_\rho,\varkappa_\rho^{-1}$.

By the Darboux theorem~\cite[Theorem~12.1]{e-z} (giving
a symplectomorphism mapping an arbitrarily chosen defining function of $\Gamma_-^\circ$
to $x_n$), we can reduce to the case $\rho=0\in T^*\mathbb R^n$ and
$\Gamma_-^\circ=\{x_n=0\}$ near $0$. Since $\Gamma_+^\circ\cap\Gamma_-^\circ=K^\circ$ is symplectic,
the Poisson bracket of the defining function $x_n$ of $\Gamma_-^\circ$ and any defining
function $\varphi_+$ of $\Gamma_+^\circ$ is nonzero at $0$; thus,
$\partial_{\xi_n}\varphi_+(0)\neq 0$ and 
we can write $\Gamma_+^\circ$ locally as the graph of some function:
$$
\Gamma_+^\circ=\{\xi_n=F(x,\xi')\}.
$$
Put $\varphi_+'(x,\xi)=\xi_n-F(x,\xi')$, then
$\{\varphi_+',x_n\}=1$. It remains to apply the Darboux theorem
once again, obtaining a symplectomorphism preserving $x_n$
and mapping $\varphi'_+$ to $\xi_n$.
\end{proof}
We now consider the sets
\begin{equation}
  \label{e:smaller-stuff}
\begin{gathered}
\widetilde \Gamma_\pm:=\bigcup_i\Gamma_\pm^i,\quad
\Gamma_\pm^i:=\Gamma_\pm^\circ\cap U_i,\\
\widetilde\Lambda:=\bigcup_i\Lambda_i,\quad
\Lambda_i:=\{(\rho_-,\rho_+)\in\Lambda^\circ\mid \rho_\pm\in\Gamma_\pm^i\}.
\end{gathered}
\end{equation}
Let $\widehat\Gamma_\pm\subset\widetilde\Gamma_\pm$ be compact,
with $\pi_\pm(\widehat\Gamma_\pm)=\widehat K$ and for each $\rho\in \widehat K$,
the set $\pi_\pm^{-1}(\rho)\cap\widehat\Gamma_\pm$ is a flow line
of $\mathcal V_\pm$ containing $\rho$. Define the compact set
\begin{equation}
  \label{e:hat-lambda}
\widehat\Lambda:=\{(\rho_-,\rho_+)\in\Lambda^\circ\mid \rho_\pm\in\widehat\Gamma_\pm\}
\end{equation}
and assume that $\widehat\Gamma_\pm$ are chosen so that $\widehat\Lambda\subset\widetilde\Lambda$.
The goal of this subsection is to obtain an invariant notion of the principal
symbol of Fourier integral operators in $\II$, microlocally
near $\widehat\Lambda$.

Define the diffeomorphisms $\widehat\varkappa_i:\Lambda_i\to V_{\tilde\delta}$
by the formula
$$
(\varkappa_i(\rho_-),\varkappa_i(\rho_+))=\phi(\widehat\varkappa_i(\rho_-,\rho_+)),\quad
(\rho_-,\rho_+)\in\Lambda_i;
$$
here $\phi$ is defined in~\eqref{e:phi}.

Consider some $A\in \II$, then $B_iAB'_i$ is a Fourier integral operator
associated to the model canonical relation $\Lambda^0$ from \S\ref{s:model}
(with the antiderivatives on $\Lambda^\circ$ and $\Lambda^0$ chosen
in the beginning of~\S\ref{s:calculus}). 
 Therefore,
there exists a compactly supported classical symbol
$\tilde a^i(x,\xi;h)$ on $\mathbb R^{2n}$ such that, with $\Op_h^\Lambda$ defined in~\eqref{e:lambda-quant},
\begin{equation}
  \label{e:model-reduction}
B_iAB'_i=\Op_h^\Lambda(\tilde a^i)+\mathcal O(h^\infty)_{\mathscr S'\to \mathscr S}.
\end{equation}
By~\eqref{e:b-i-inv}, we find
$$
A=B'_i\Op_h^\Lambda(\tilde a^i)B_i+\mathcal O(h^\infty)\quad\text{microlocally near }\Lambda_i.
$$
Define the function $a^i\in C^\infty(\Lambda_i)$ using the principal symbol $\tilde a^i_0$ by
$$
a^i=\tilde a^i_0\circ\widehat\varkappa_i.
$$
By Proposition~\ref{l:change-of-variables}, applied to the Fourier integral operators
$B_jB'_i$ and $B_iB'_j$ quantizing $\varkappa=\varkappa_j\circ\varkappa_i^{-1}$
and $\varkappa^{-1}$, respectively, with
$U'=\varkappa_i(U_i\cap U_j)$, $V'=\varkappa_j(U_i\cap U_j)$
we see that whenever $\Lambda_i\cap\Lambda_j\neq\emptyset$,
we have
\begin{equation}
  \label{e:transition}
a^i|_{\Lambda_i\cap\Lambda_j}=(\gamma_{ij}^-\otimes \gamma_{ij}^+)a^j|_{\Lambda_i\cap\Lambda_j},
\end{equation}
where $\gamma_{ij}^\pm$ are smooth functions on $\Gamma_\pm^i\cap \Gamma_\pm^j$ and $\gamma^\pm_{ij}|_K=1$.
Moreover, $\gamma_{ji}^\pm=(\gamma_{ij}^\pm)^{-1}$ and $\gamma_{ij}^\pm\gamma_{jk}^\pm=\gamma_{ik}^\pm$
on $\Gamma_\pm^i\cap\Gamma_\pm^j\cap\Gamma_\pm^k$
(this can be seen either from the fact that the formulas~\eqref{e:transition} for different $i,j$
have to be compatible with each other, or directly from~\eqref{e:transition-coeff}).
Therefore, we can consider smooth line
bundles $\mathcal E_\pm$ over $\widetilde\Gamma_\pm$ with smooth sections $e^i_\pm$ of $\mathcal E_\pm|_{\Gamma_\pm^i}$
such that $e^j_\pm=\gamma_{ij}^\pm e^i_\pm$ on $\Gamma_\pm^i\cap\Gamma_\pm^j$~-- see for
example~\cite[\S 6.4]{ho1}. 

Define the line bundle $\mathcal E$ over $\widetilde\Lambda$ using the projection
maps from~\eqref{e:tilde-pi}:
$$
\mathcal E=(\tilde\pi_-^*\mathcal E^-)\otimes(\tilde \pi_+^*\mathcal E^+)
$$
and
for $A\in \II$, the symbol $\sigma_\Lambda(A)\in C^\infty(\widetilde\Lambda;\mathcal E)$
by the formula
\begin{equation}
  \label{e:symbol}
\sigma_\Lambda(A)|_{\Lambda_i}=a^i(\tilde\pi_-^*e^i_-\otimes \tilde\pi_+^*e^i_+).
\end{equation}
Note that the bundle $\mathcal E$ can be studied in detail using the global theory
of Fourier integral operators (see for instance~\cite[\S25.1]{ho4}).
However, the situation in our
special case is considerably simplified, since the Maslov bundle does not appear.

We have $\sigma_\Lambda(A)=0$ near $\widehat\Lambda$
if and only if $A\in h\II$ microlocally near
$\widehat\Lambda$. Moreover, for all
$a\in C^\infty(\widetilde\Lambda;\mathcal E)$, there exists $A\in \II$ such that
$\sigma_\Lambda(A)=a$ near $\widehat\Lambda$.

The restrictions $\mathcal E_\pm|_{\widetilde K}$ are canonically trivial; that is,
for $a_\pm\in C^\infty(\widetilde\Gamma_\pm;\mathcal E_\pm)$, we can view
$a_\pm|_{\widetilde K}$ as a function on $\widetilde K$, by taking $e_\pm^i|_{K_i}=1$.
The bundles $\mathcal E_\pm$ are trivial:
\begin{prop}
  \label{l:trivial}
There exist sections $a_\pm\in C^\infty(\widetilde\Gamma_\pm;\mathcal E_\pm)$, 
nonvanishing near $\widehat\Gamma_\pm$ and such that
$a_\pm|_{\widetilde K}=1$ near $\widehat K$.
\end{prop}
\begin{proof}
Since $\gamma_{ij}^\pm$ is a nonvanishing smooth function on $\Gamma_\pm^i\cap\Gamma_\pm^j$ such that
$\gamma_{ij}^\pm|_{K_i\cap K_j}=1$, we can write
$$
\gamma_{ij}^\pm=\exp(f_{ij}^\pm),
$$
where $f_{ij}^\pm$ is a uniquely defined function on $\Gamma_\pm^i\cap\Gamma_\pm^j$,
such that $f_{ij}^\pm|_{K_i\cap K_j}=0$. We now put near $\widehat\Gamma_\pm$,
$$
a_\pm|_{\Gamma_\pm^i}=\exp(b_\pm^i)e_\pm^i,
$$
where $b_\pm^i\in C^\infty(\Gamma_\pm^i)$ are such that near $\widehat\Gamma_\pm$
and $\widehat K$ respectively,
$$
(b_\pm^i-b_\pm^j)|_{\Gamma_\pm^i\cap\Gamma_\pm^j}=f_{ij}^\pm,\quad
b_\pm^i|_{K_i}=0.
$$
Such functions exist since $f_{ij}^\pm$ is a cocycle:
$$
f_{ii}^\pm=f_{ij}^\pm+f_{ji}^\pm=0;\quad
f_{ij}^\pm+f_{jk}^\pm=f_{ik}^\pm\quad\text{on }\Gamma_\pm^i\cap\Gamma_\pm^j\cap\Gamma_\pm^k
$$
and since the sheaf of smooth functions is fine; more precisely, if
$1=\sum_i \chi_i$ is a partition of unity on $\widehat\Gamma_\pm$, with
$\supp\chi_i\subset\Gamma_\pm^i$, we put
$$
b_\pm^i=\sum_k \chi_k f_{ik}^\pm.
\qedhere
$$
\end{proof}
We now state the properties of the calculus, following directly
from~\eqref{e:calc-model-1}--\eqref{e:calc-model-3}, the general
theory of Fourier integral operators, and Egorov's Theorem~\cite[Theorem~11.1]{e-z}
(see the beginning of~\S\ref{s:calculus}
for multiplying two elements of $\II$):
\begin{prop}
  \label{l:calculus}
Assume that $A_1,A_2\in \II,P\in \Psi^k(X)$. Then
$A_1A_2,A_1P,PA_1$ lie in $\II$, and 
\begin{gather}
\label{e:good-calc-1}
\sigma_\Lambda(A_1A_2)(\rho_-,\rho_+)=\sigma_\Lambda(A_2)(\rho_-,\pi_-(\rho_-))
\otimes\sigma_\Lambda(A_1)(\pi_+(\rho_+),\rho_+),\\
\label{e:good-calc-2}
\sigma_\Lambda(A_1P)(\rho_-,\rho_+)=\sigma(P)(\rho_-)\cdot\sigma_\Lambda(A_1)(\rho_-,\rho_+),\\
\label{e:good-calc-3}
\sigma_\Lambda(PA_1)(\rho_-,\rho_+)=\sigma(P)(\rho_+)\cdot\sigma_\Lambda(A_1)(\rho_-,\rho_+).
\end{gather}
Here in~\eqref{e:good-calc-1},
$\sigma_\Lambda(A_2)(\rho_-,\pi_-(\rho_-))$ and $\sigma_\Lambda(A_1)(\pi_+(\rho_+),\rho_+)$
are considered as sections of $\mathcal E_-$ and $\mathcal E_+$, respectively.
\end{prop}
We next give a parametrix construction for operators of the form
$1-A$, with $A\in \II$, needed in~\S\ref{s:grushin}:
\begin{prop}
  \label{l:funny-parametrix}
Let $A\in \II$ and assume that
$$
\WFh(A)\subset\widehat\Lambda;\quad
\sigma_\Lambda(A)|_{\widetilde K}\neq 1\text{ everywhere}.
$$
Then there exists $B\in \II$ with $\WFh(B)\subset\widehat\Lambda$,
and such that
$$
(1-A)(1-B)=1+\mathcal O(h^\infty),\quad
(1-B)(1-A)=1+\mathcal O(h^\infty).
$$
Moreover, $B$ is uniquely defined modulo $\mathcal O(h^\infty)$ and
\begin{equation}
  \label{e:funny-symbol}
\sigma_\Lambda(B)(\rho_-,\rho_+)=
{\sigma_\Lambda(A)(\rho_-,\pi_-(\rho_-))\otimes \sigma_\Lambda(A)(\pi_+(\rho_+),\rho_+)
\over \sigma_\Lambda(A)(\pi_-(\rho_-),\pi_+(\rho_+))-1}-\sigma_\Lambda(A)(\rho_-,\rho_+).
\end{equation}
\end{prop}
\begin{proof}
Take any $B_1\in \II$ with $\WFh(B_1)\subset\widehat\Lambda$
and symbol given by~\eqref{e:funny-symbol}.
By~\eqref{e:good-calc-1}, $(1-A)(1-B_1)=1-hR$,
for some $R\in \II$ with $\WFh(R)\subset\widehat\Lambda$.
Define $B_2\in \II$ by the asymptotic Neumann series
$$
-B_2\sim\sum_{j\geq 1}h^jR^j.
$$
Define $B\in \II$ by the identity $1-B=(1-B_1)(1-B_2)$, then
$(1-A)(1-B)=1+\mathcal O(h^\infty)$.
Similarly, we construct $B'\in \II$ such that $(1-B')(1-A)=1+\mathcal O(h^\infty)$.
A standard algebraic argument, see for example
the proof of~\cite[Theorem~18.1.9]{ho3}, shows that $B'=B+\mathcal O(h^\infty)$
and both are determined uniquely modulo $\mathcal O(h^\infty)$.
\end{proof}
We finish this subsection with a trace formula for operators in $\II$,
used in~\S\ref{s:trace}:
\begin{prop}
  \label{l:trace-basic}
Assume that $A\in \II$ and $\WFh(A)\subset\widehat\Lambda$. Then,
with $d\Vol_\sigma=\sigma_S^{n-1}/(n-1)!$ denoting the symplectic volume form
and $j_K:K^\circ\to\Lambda^\circ$ defined in~\eqref{e:j-k},
$$
(2\pi h)^{n-1}\Tr A=\int_{\widehat K} \sigma_\Lambda(A)\circ j_K \,d\Vol_\sigma+\mathcal O(h).
$$
\end{prop}
\begin{proof}
By a microlocal partition of unity, we reduce to the case when $\WFh(A)$ lies
entirely in one of the sets $\Lambda_i$ defined in~\eqref{e:smaller-stuff}.
If $\tilde a_i$ is defined by~\eqref{e:model-reduction}, then by the
cyclicity of the trace,
$\Tr A=\Tr \Op_h^\Lambda(\tilde a_i)+\mathcal O(h^\infty)$. It remains to note that
for any $a(x,\xi)\in C_0^\infty(\mathbb R^{2n})$,
$$
(2\pi h)^{n-1}\Tr \Op_h^\Lambda(a)=\int_{\mathbb R^{2n-2}}a(x',0,\xi',0)\,dx'd\xi'+\mathcal O(h),
$$
seen directly from~\eqref{e:lambda-quant} by the method of stationary phase
in the $x_n,\xi_n$ variables.
\end{proof}

\subsection{Microlocal idempotents}
\label{s:idempotents}

In this subsection, we establish properties of microlocal idempotents associated
to the Lagrangian $\Lambda^\circ$ considered in \S\ref{s:general}, microlocally on the compact
set $\widehat\Lambda$ defined in~\eqref{e:hat-lambda}. We use
the principal symbol $\sigma_\Lambda$ constructed in~\eqref{e:symbol}.
\begin{defi}
  \label{d:microlocal-idempotent}
We call $A\in \II$
a \emph{microlocal idempotent} of order $k>0$ near $\widehat\Lambda$, if
$A^2=A+\mathcal O(h^k)_{\II}$ microlocally near $\widehat\Lambda$
and $\sigma_\Lambda(A)$ does not vanish on $\widehat\Lambda$.
\end{defi}
In the following Proposition, part~1 is concerned with the principal part of
the idempotent equation; part~2 establishes a normal form for microlocal idempotents,
making it possible to conjugate them microlocally to the operator $\Op_h^\Lambda(1)$
from~\eqref{e:model-projector}. Part~3 is used to construct a global idempotent of
all orders in Proposition~\ref{l:idempotent-exists} below, while part~4 establishes
properties of commutators used in the construction of~\S\ref{s:global-construction}.
\begin{prop}
  \label{l:idempotents}
1. $A\in \II$ is a microlocal idempotent of order 1 near $\widehat\Lambda$ if and only if
near $\widehat\Lambda$,
\begin{equation}
  \label{e:idempotent-principal}
\sigma_\Lambda(A)(\rho_-,\rho_+)=a_0^-(\rho_-)\otimes a_0^+(\rho_+)
\end{equation}
for some sections $a_0^\pm\in C^\infty(\widetilde\Gamma_\pm;\mathcal E_\pm)$ nonvanishing
near $\widehat\Gamma_\pm$ and
such that $a_0^\pm|_{\widetilde K}=1$ near $\widehat K$. Moreover, $a_0^\pm$ are uniquely determined by $A$
on $\widehat \Gamma_\pm$.

2. If $A,B\in \II$ are two microlocal idemptotents of order $k>0$ near $\widehat\Lambda$,
then there exists an operator $Q\in\Psi^{\comp}(X)$, elliptic on $\widehat\Gamma_+\cup\widehat\Gamma_-$
and such that $B=Q A Q^{-1}+\mathcal O(h^k)_{\II}$ microlocally near $\widehat\Lambda$.
Here $Q^{-1}$ denotes an elliptic parametrix of~$Q$
constructed in Proposition~\ref{l:eparametrix}.

3. If $A\in \II$ is a microlocal idempotent of order $k>0$ near $\widehat\Lambda$, and
$A^2-A=h^kR_k+\mathcal O(h^\infty)$ microlocally near $\widehat\Lambda$ for some $R_k\in \II$, then
for $\rho_\pm\in\widetilde\Gamma_\pm$ near $\widehat\Gamma_\pm$,
\begin{equation}
  \label{e:r-eq}
\begin{gathered}
\sigma_\Lambda(R_k)(\pi_+(\rho_+),\rho_+)=\sigma_\Lambda(R_k)(\pi_+(\rho_+),\pi_+(\rho_+))\cdot
a_0^+(\rho_+),\\
\sigma_\Lambda(R_k)(\rho_-,\pi_-(\rho_-))=\sigma_\Lambda(R_k)(\pi_-(\rho_-),\pi_-(\rho_-))\cdot
a_0^-(\rho_-),
\end{gathered}
\end{equation}
with $a_0^\pm$ defined in~\eqref{e:idempotent-principal}.

4. If $A\in \II$ is a microlocal idempotent of all orders near $\widehat\Lambda$,
$P\in\Psi^{\comp}(X)$ is compactly supported, and
$[P,A]=h^k S_k+\mathcal O(h^\infty)$ microlocally near $\widehat\Lambda$ for some $S_k\in \II$, then
near $\widehat\Lambda$,
$$
\sigma_\Lambda(S_k)(\rho_-,\rho_+)=a_0^-(\rho_-)\otimes \sigma_\Lambda(S_k)(\pi_+(\rho_+),\rho_+)
+\sigma_\Lambda(S_k)(\rho_-,\pi_-(\rho_-))\otimes a_0^+(\rho_+).
$$
In particular, $\sigma_\Lambda(S_k)\circ j_K=0$ near $\widehat K$,
with $j_K:K^\circ\to\Lambda^\circ$ defined in~\eqref{e:j-k}.
\end{prop}
\begin{proof} In this proof, all the equalities of operators in $\II$
and the corresponding symbols are presumed to hold microlocally near $\widehat\Lambda$.

1. By~\eqref{e:good-calc-1}, we have $A^2=A+\mathcal O(h)$
if and only if
$$
\sigma_\Lambda(A)(\rho_-,\rho_+)=\sigma_\Lambda(A)(\rho_-,\pi_-(\rho_-))\otimes
\sigma_\Lambda(A)(\pi_+(\rho_+),\rho_+).
$$
In particular, restricting to $\widetilde K$, we obtain $\sigma_\Lambda(A)=\sigma_\Lambda(A)^2$
near $\widehat K$. Since $\sigma_\Lambda(A)$ is nonvanishing, we get $\sigma_\Lambda(A)|_{\widetilde K}=1$
near $\widehat K$.
It then remains to put $a^-_0(\rho_-)=\sigma_\Lambda(A)(\rho_-,\pi_-(\rho_-))$
and $a^+_0(\rho_+)=\sigma_\Lambda(A)(\pi_+(\rho_+),\rho_+)$.

2. We use induction on $k$. For $k=1$, we have by~\eqref{e:good-calc-2} and~\eqref{e:good-calc-3},
$$
\sigma_\Lambda(QAQ^{-1})(\rho_-,\rho_+)={\sigma(Q)(\rho_+)\over\sigma(Q)(\rho_-)}
\sigma_\Lambda(A)(\rho_-,\rho_+).
$$
If $a_0^\pm$ and $b_0^\pm$ are given by~\eqref{e:idempotent-principal}, then it is enough to take
any $Q$ with
\begin{equation}
  \label{e:q-principal}
\sigma(Q)|_{\widetilde\Gamma_-}=a_0^-/b_0^-,\
\sigma(Q)|_{\widetilde\Gamma_+}=b_0^+/a_0^+,
\end{equation}
this is possible since the restrictions of $a_0^\pm$ and $b_0^\pm$ to $\widetilde K$ are equal to 1.

Now, assuming the statement is true for $k\geq 1$, we prove it for $k+1$.
We have $B=\widetilde Q A\widetilde Q^{-1}+\mathcal O(h^k)$ for some $\widetilde Q\in\Psi^{\comp}$
elliptic on $\widehat\Gamma_+\cup\widehat\Gamma_-$;
replacing $A$ by $\widetilde Q A\widetilde Q^{-1}$, we may assume that $B=A+\mathcal O(h^k)$. Then $B-A=h^k R_k$
for some $R_k\in \II$; since both~$A$ and~$B$ are microlocal idempotents
of order $k+1$, we find $R_k=AR_k+R_kA+\mathcal O(h)$ and thus by~\eqref{e:good-calc-1},
\begin{equation}
  \label{e:r-eqn}
\sigma_\Lambda(R_k)(\rho_-,\rho_+)=a_0^-(\rho_-)\otimes \sigma_\Lambda(R_k)(\pi_+(\rho_+),\rho_+)
+\sigma_\Lambda(R_k)(\rho_-,\pi_-(\rho_-))\otimes a_0^+(\rho_+).
\end{equation}
Take $Q=1+h^kQ_k$ for some $Q_k\in\Psi^{\comp}$, then $Q^{-1}=1-h^kQ_k+\mathcal O(h^{k+1})$ and
$$
QAQ^{-1}=A+h^k[Q_k,A]+\mathcal O(h^{k+1}).
$$
Now, $B=QAQ^{-1}+\mathcal O(h^{k+1})$ if and only if
$$
(\sigma(Q_k)(\rho_+)-\sigma(Q_k)(\rho_-))a_0^-(\rho_-)\otimes a_0^+(\rho_+)=\sigma_\Lambda(R_k)(\rho_-,\rho_+).
$$
By~\eqref{e:r-eqn}, it is enough to choose $Q_k$ such that for $\rho_\pm\in\widetilde \Gamma_\pm$,
$$
\sigma(Q_k)(\rho_-)=-{\sigma_\Lambda(R_k)(\rho_-,\pi_-(\rho_-))\over a_0^-(\rho_-)},\quad
\sigma(Q_k)(\rho_+)={\sigma_\Lambda(R_k)(\pi_+(\rho_+),\rho_+)\over a_0^+(\rho_+)},
$$
this is possible since $\sigma_\Lambda(R_k)\circ j_K=0$ (with $j_K$ defined in~\eqref{e:j-k})
as follows from~\eqref{e:r-eqn}.

3. Since this is a local statement, we can use~\eqref{e:model-reduction} to reduce
to the model case of \S\ref{s:model}.
Using part~2 and the fact that the operator $\Op_h^\Lambda(1)$ 
considered in~\eqref{e:model-projector} is a microlocal idemptotent of all orders,
we can write
$$
A=Q\Op_h^\Lambda(1)Q^{-1}+h^kA_k,
$$
for some elliptic $Q\in\Psi^{\comp}$ and $A_k\in I_{\comp}(\Lambda^0)$. Then
$$
R_k=Q\Op_h^\Lambda(1)Q^{-1}A_k+A_kQ\Op_h^\Lambda(1)Q^{-1}-A_k+\mathcal O(h);
$$
\eqref{e:r-eq} follows by~\eqref{e:good-calc-1}
since $\sigma_\Lambda(Q\Op_h^\Lambda(1)Q^{-1})=\sigma_\Lambda(A)$ is given by~\eqref{e:idempotent-principal}.

4. As in part~3, we reduce to the model case of~\S\ref{s:model} and use part~2 to write
$A=Q\Op_h^\Lambda(1)Q^{-1}+\mathcal O(h^\infty)$; then
$$
[P,A]=Q[Q^{-1}PQ,\Op_h^\Lambda(1)]Q^{-1}+\mathcal O(h^\infty).
$$
Put $\widetilde P=Q^{-1}PQ$; by~\eqref{e:exp-model-2} and~\eqref{e:exp-model-3}
we have $[\widetilde P,\Op_h^\Lambda(1)]=\Op_h^\Lambda(s\circ\phi)$, where
$\phi$ is given by~\eqref{e:phi} and
$$
s(\rho_-,\rho_+;h)=\tilde p(\rho_+;h)-\tilde p(\rho_-;h),
$$
where $\widetilde P=\Op_h(\tilde p)$; thus
$$
s(\rho_-,\rho_+;h)=s(\pi_+(\rho_+),\rho_+;h)+s(\rho_-,\pi_-(\rho_-);h).
$$
It remains to conjugate by $Q$, keeping in mind~\eqref{e:q-principal}.
\end{proof}
We can use part~3 of Proposition~\ref{l:idempotents}, together with
the triviality of the bundles $\mathcal E_\pm$, to show existence
of a global idempotent, which is the starting point of the construction in~\S\ref{s:global-construction}.
\begin{prop}
  \label{l:idempotent-exists}
There exists a microlocal idempotent $\widetilde\Pi\in \II$ of all orders near $\widehat\Lambda$.
\end{prop}
\begin{proof}
We argue inductively, constructing microlocal idempotents $\widetilde\Pi_k$ of order $k$ for each $k$
and taking the asymptotic limit. To construct $\widetilde\Pi_1$, we use part~1 of Proposition~\ref{l:idempotents};
the existence of symbols $a_0^\pm$ was shown in Proposition~\ref{l:trivial}.

Now, assume that $\widetilde\Pi_k$ is a microlocal idempotent of order $k>0$.
By part~3 of Proposition~\ref{l:idempotents}, we have
$\widetilde\Pi_k^2-\widetilde\Pi_k=h^kR_k+\mathcal O(h^\infty)$ microlocally near $\widehat\Lambda$,
where $R_k\in \II$ and $r_k=\sigma_\Lambda(R_k)$ satisfies~\eqref{e:r-eq}.
Put $\widetilde\Pi_{k+1}=\widetilde \Pi_k+h^k B_k$, for
some $B_k\in \II$. We need to choose $B_k$ so that
microlocally near $\widehat\Lambda$,
$$
R_k+\widetilde\Pi_k  B_k+ B_k\widetilde \Pi_k- B_k=\mathcal O(h).
$$
Taking $b_k=\sigma_\Lambda(B_k)$, by~\eqref{e:good-calc-1} this translates to
$$
b_k(\rho_-,\rho_+)=a_0^-(\rho_-)\otimes b_k(\pi_+(\rho_+),\rho_+)+
b_k(\rho_-,\pi_-(\rho_-))\otimes a_0^+(\rho_+)
+r_k(\rho_-,\rho_+).
$$
By~\eqref{e:r-eq}, it is enough to take any
$b_k^\pm\in C^\infty(\widetilde \Gamma_\pm;\mathcal E_\pm)$ such that near $\widehat K$,
$b_k^\pm|_{\widetilde K}=-r_k\circ j_K$, with
$j_K$ defined in~\eqref{e:j-k} 
(for example, $b_k^\pm=-(r_k\circ j_K\circ \pi_\pm)a_0^\pm$)
and put
$$
b_k(\rho_-,\rho_+):=a_0^-(\rho_-)\otimes b_k^+(\rho_+)+b_k^-(\rho_-)\otimes a_0^+(\rho_+)
+r_k(\rho_-,\rho_+).\qedhere
$$
\end{proof}

\subsection{Annihilating ideals}
\label{s:ideals}

Assume that $\Pi\in \II$ is a microlocal idempotent of all orders
near the set $\widehat\Lambda$ introduced in~\eqref{e:hat-lambda},
see Definition~\ref{d:microlocal-idempotent}.
We are interested in the following equations:
\begin{gather}
\label{e:right-ideal}
\Pi\Theta_-=\mathcal O(h^\infty)\quad\text{microlocally near }\widehat\Lambda,\\
\label{e:left-ideal}
\Theta_+\Pi=\mathcal O(h^\infty)\quad\text{microlocally near }\widehat\Lambda,
\end{gather}
where $\Theta_\pm$ are pseudodifferential operators.
The solutions to~\eqref{e:right-ideal} form a right ideal and
the solutions to~\eqref{e:left-ideal} form a left ideal
in the algebra of pseudodifferential operators.
Moreover, by~\eqref{e:good-calc-2}, \eqref{e:good-calc-3},
each solution $\Theta_\pm$ to the equations~\eqref{e:right-ideal},
\eqref{e:left-ideal} satisfies $\sigma(\Theta_\pm)|_{\Gamma_\pm}=0$ near $\widehat\Gamma_\pm$
and each $\Theta_\pm$ such that $\WFh(\Theta_\pm)\cap\widehat\Gamma_\pm=\emptyset$
solves these equations.

Note that in the model case of~\S\ref{s:model}, with $\Pi$ equaling the operator $\Op_h^\Lambda(1)$
from~\eqref{e:model-projector}, and with the
quantization procedure $\Op_h$ defined in~\eqref{e:op-h},
the set of solutions to~\eqref{e:right-ideal}
is the set of operators $\Op_h(\theta_-)$ with $\theta_-|_{x_n=0}=0$; that is,
the right ideal generated by the operator $x_n$. The set of solutions to~\eqref{e:left-ideal}
is the set of operators $\Op_h(\theta_+)$ with $\theta_+|_{\xi_n=0}=0$; that is,
the left ideal generated by the operator $hD_{x_n}$. This follows from the multiplication
formulas~\eqref{e:exp-model-2} and~\eqref{e:exp-model-3}, together with the
multiplication formulas for the standard quantization~\cite[(4.3.16)]{e-z}.

We start by showing that our ideals are principal in the general setting:
\begin{prop}
  \label{l:principal-ideal}
1. For each defining functions $\varphi_\pm$ of $\Gamma_\pm^\circ$
near $\widehat\Gamma_\pm$,
there exist operators $\Theta_\pm$ solving~\eqref{e:right-ideal},
\eqref{e:left-ideal}, such that $\sigma(\Theta_\pm)=\varphi_\pm$
near $\widehat\Gamma_\pm$. Such operators
are called \emph{basic} solutions of the corresponding equations.

2. If $\Theta_\pm,\Theta'_\pm$ are solutions to~\eqref{e:right-ideal},
\eqref{e:left-ideal}, and moreover $\Theta_\pm$ are basic solutions,
then there exist $Z_\pm\in\Psi^{\comp}$ such that
$\Theta'_-=\Theta_-Z_-+\mathcal O(h^\infty)$ microlocally near $\widehat\Gamma_-$ and
$\Theta'_+=Z_+\Theta_++\mathcal O(h^\infty)$ microlocally near $\widehat\Gamma_+$.
\end{prop}
\begin{proof}
We concentrate on the equation~\eqref{e:right-ideal}; \eqref{e:left-ideal} is handled similarly.
Since the equations~\eqref{e:right-ideal} and $\Theta'=\Theta_-Z_-$ are linear in $\Theta_-$
and $\Theta',Z_-$, respectively, we can use~\eqref{e:model-reduction} 
and a pseudodifferential partition of unity to reduce to the model case
of~\S\ref{s:model}. Using part~2 of Proposition~\ref{l:idempotents},
we can furthermore assume that $\Pi=\Op_h^\Lambda(1)$. 

To show part~1, in the model case, we can take $\Theta_-=\Op_h(\varphi_-)$, where $\varphi_-(x,\xi)$
is the given defining function of $\{x_n=0\}$. For part~2, if $\Theta_-=\Op_h(\varphi_-)$
and $\Theta'_-=\Op_h(\theta'_-)$, then we can write microlocally near $\widehat\Gamma_-$,
$\Theta_-=x_n Y_-+\mathcal O(h^\infty)$, where $Y_-\in\Psi^{\comp}$ is elliptic on $\widehat\Gamma_-$;
in fact, $Y_-=\Op_h(\varphi_-/x_n)$.
Similarly we can write $\Theta'_-=x_nY'_-+\mathcal O(h^\infty)$ microlocally near $\widehat\Gamma_-$,
for some $Y'_-\in\Psi^{\comp}$;
it remains to put $Z_-=Y_-^{-1}Y'_-$ microlocally near $\widehat\Gamma_-$.
\end{proof}
For the microlocal estimate on the kernel of $\Pi$ in~\S\ref{s:estimate-kernel},
we need an analog of the following fact:
\begin{equation}
  \label{e:division}
f\in C^\infty(\mathbb R^n)\ \Longrightarrow\
f(x)-f(x',0)=x_ng(x),\quad
g\in C^\infty(\mathbb R^n),
\end{equation}
where $f(x',0)$ is replaced by $\Pi f$ and multiplication by $x_n$ is replaced by a basic
solution to~\eqref{e:right-ideal}.
We start with a technical lemma for the model case:
\begin{lemm}
  \label{l:Xi-0}
Consider the operator $\Xi_0:C^\infty(\mathbb R^n)\to C^\infty(\mathbb R^n)$ defined by
$$
\Xi_0f(x',x_n)={f(x',x_n)-f(x',0)\over x_n}=\int_0^1 (\partial_{x_n}f)(x',tx_n)\,dt.
$$
Then:

1. $\Xi_0$ is bounded $H^1(\mathbb R^n)\to L^2(\mathbb R^n)$ and thus
$\|\Xi_0\|_{H^1_h\to L^2}=\mathcal O(h^{-1})$.

2. The wavefront set $\WFh(\Xi_0)$ defined in~\S\ref{s:prelim-basics} satisfies%
\footnote{It would be interesting to understand the microlocal structure of $\Xi_0$, starting
from the fact that its wavefront set lies in the union of three Lagrangian submanifolds.}
$$
\begin{gathered}
\WFh(\Xi_0)\cap T^*(\mathbb R^n\times\mathbb R^n)\,\subset\,\Delta(T^*\mathbb R^n)\cup\Lambda^0\\
\cup \{(x',0,\xi,x',0,\xi',t\xi_n)\mid (x',\xi)\in\mathbb R^{2n-1},t\in [0,1]\},
\end{gathered}
$$
where $\Delta(T^*\mathbb R^n)\subset T^*\mathbb R^n\times T^*\mathbb R^n$ is the diagonal
and $\Lambda^0$ is defined using~\eqref{e:model-gamma}.
\end{lemm}
\begin{proof}
1. Put $\lambda_t f(x',x_n)=(\partial_{x_n}f)(x',tx_n)$; then
$$
\|\Xi_0 f\|_{L^2}\leq \int_0^1\|\lambda_t f\|_{L^2}\,dt
\leq \int_0^1 t^{-1/2}\|f\|_{H^1}\,dt
\leq 2\|f\|_{H^1}.
$$

2. Denote elements of $T^*(\mathbb R^n\times\mathbb R^n)$
by $(x,\xi,y,\eta)$.
If $\chi\in C_0^\infty(\mathbb R)$ is supported away from zero, then,
with $\Op_h^\Lambda(1)$ defined in~\eqref{e:model-projector},
$$
\chi(x_n) \Xi_0={\chi(x_n)\over x_n} (1-\Op_h^\Lambda(1)).
$$
Since $\chi(x_n)/x_n$ is a smooth function,
the identity operator has wavefront set on the diagonal,
and $\WFh(\Op_h^\Lambda(1))\cap T^*(\mathbb R^n\times \mathbb R^n)\subset\Lambda^0$,
we find
$$
\begin{gathered}
\WFh(\Xi_0)\cap T^*(\mathbb R^n\times\mathbb R^n)\cap\{y_n\neq 0\}\,\subset\,
\Delta(T^*\mathbb R^n)
\cup \Lambda^0.
\end{gathered}
$$
Similarly, one has
$\Xi_0\chi(x_n)=\chi(x_n)/x_n$;
therefore,
$$
\WFh(\Xi_0)\cap\{x_n\neq 0\}\,\subset\, \Delta(T^*\mathbb R^n).
$$
To handle the remaining part of the wavefront set, take $a,b\in C_0^\infty(T^*\mathbb R^n)$
such that
$$
(x',tx_n,\xi)\in\supp a,\ t\in [0,1]\Longrightarrow
(x,\xi',t\xi_n)\not\in\supp b.
$$
We claim that for any $\psi\in C_0^\infty(\mathbb R^n)$,
\begin{equation}
  \label{e:wf-internal}
\Op_h(b)\psi\,\Xi_0\Op_h(a)\psi=\mathcal O(h^\infty);
\end{equation}
indeed, the Schwartz kernel of this operator is
$$
\begin{gathered}
\mathcal K(y,x)=(2\pi h)^{-2n}\int_{\mathbb R^{3n}\times [0,1]} e^{{i\over h}((y-z)\cdot\eta+(z'-x')\cdot\xi'+(tz_n-x_n)\xi_n)}\\
b(y,\eta)\psi(z)(ih^{-1}\xi_n a(z',tz_n,\xi)+(\partial_{z_n}a)(z',tz_n,\xi))\psi(x)\,d\xi d\eta dz dt.
\end{gathered}
$$
The stationary points of the phase in the $(\xi,\eta,z)$ variables are given by
$$
z=y,\
x'=y',\
x_n=ty_n,\
\eta'=\xi',\
\eta_n=t\xi_n
$$
and lie outside of the support of the amplitude; by the method of nonstationary phase
in the $(\xi,\eta,z)$ variables, the integral is $\mathcal O(h^\infty)_{C^\infty}$. Now, \eqref{e:wf-internal} implies that
$$
\begin{gathered}
\WFh(\Xi_0)\cap T^*(\mathbb R^n\times\mathbb R^n)\cap\{x_n=y_n=0\}\\\subset\,
\{(x',0,\xi,x',0,\xi',t\xi_n)\mid (x',\xi)\in\mathbb R^{2n-1},\ t\in [0,1]\},
\end{gathered}
$$
which finishes the proof.
\end{proof}
The microlocal analog of~\eqref{e:division} in the general case is now given by
\begin{prop}
  \label{l:Xi}
Let $\Pi\in \II$ be a microlocal idempotent of all orders near $\widehat\Lambda$
and $\Theta_-$ be a basic solution to~\eqref{e:right-ideal}, see Proposition~\ref{l:principal-ideal}.
Then there exists an operator $\Xi:C^\infty(X)\to C_0^\infty(X)$ such that:

1. $\WFh(\Xi)$ is a compact subset of $T^*(X\times X)$ and
$\|\Xi\|_{L^2\to L^2}=\mathcal O(h^{-1})$;

2. $\WFh(\Xi)\subset \Delta(T^*X)\cup \Lambda^\circ\cup \Upsilon$,
where $\Delta(T^*X)\subset T^*X\times T^*X$ is the diagonal and
 $\Upsilon$ consists of all $(\rho_-,\rho'_-)$ such that
$\rho_-,\rho'_-\in\Gamma_-^\circ$ and $\rho'_-$ lies on the
segment of the flow line of $\mathcal V_-$ between $\rho_-$ and $\pi_-(\rho_-)$;

3. $1-\Pi=\Theta_-\Xi+\mathcal O(h^\infty)$ microlocally near $\widehat K\times\widehat K$.
\end{prop}
\begin{proof}
By~\eqref{e:model-reduction} and a microlocal partition of unity, we can reduce to the model case of~\S\ref{s:model}.
Moreover, by part~2 of Proposition~\ref{l:idempotents}, we may conjugate
by a pseudodifferential operator to make $\Pi=\Op_h^\Lambda(1)$.
Finally, by part~2 of Proposition~\ref{l:principal-ideal} we can multiply $\Theta_-$
on the right by an elliptic pseudodifferential operator to
make $\Theta_-=\Op_h(x_n)$. Then we can take
$\Xi=A\Xi_0 A$, with $\Xi_0$ defined in Lemma~\ref{l:Xi-0}
and $A\in\Psi^{\comp}(\mathbb R^n)$ compactly supported, with $A=1+\mathcal O(h^\infty)$ microlocally near $\widehat K$.
\end{proof}

\section{The projector \texorpdfstring{$\Pi$}{Pi}}
  \label{s:global-construction}

In this section, we construct the microlocal projector $\Pi$ near a neighborhood
$\widehat W$ of $K\cap p^{-1}([\alpha_0,\alpha_1])$ discussed
in the introduction (Theorem~\ref{t:our-Pi} in~\S\ref{s:construction-1}).
In~\S\ref{s:construction-2}, we study the annihilating ideals for $\Pi$ in $\widehat W$
using~\S\ref{s:ideals}.

\subsection{Construction of \texorpdfstring{$\Pi$}{Pi}}
  \label{s:construction-1}

Assume that the conditions of \S\S\ref{s:framework-assumptions} and~\ref{s:dynamics}
hold. Consider the sets $\Gamma_\pm^\circ$ and $K^\circ=\Gamma_+^\circ\cap\Gamma_-^\circ$
defined in~\eqref{e:k-circ} and let $\Lambda^\circ$ be given by~\eqref{e:the-Lambda}. Put
$$
\widehat K := K\cap p^{-1}([\alpha_0-\delta_1/2,\alpha_1+\delta_1/2])\subset K^\circ,
$$
here $\delta_1$ is defined in~\S\ref{s:projections}. The sets
$\Gamma_\pm^\circ$ satisfy the assumptions listed in the beginning of~\S\ref{s:calculus},
as follows from~\S\S\ref{s:dynamics} and~\ref{s:projections}.

We choose $\delta>0$ small enough so that Lemma~\ref{l:phi-pm}
holds (we will impose more conditions on $\delta$ in~\S\ref{s:construction-2})
and consider the sets
\begin{equation}
  \label{e:w-hat}
\begin{gathered}
\widehat W:=U_\delta\cap p^{-1}([\alpha_0-\delta_1/2,\alpha_1+\delta_1/2]),\\
\widehat\Gamma_\pm:=\Gamma_\pm^\circ\cap\widehat W,\quad
\widehat\Lambda:=\Lambda^\circ\cap (\widehat W\times\widehat W).
\end{gathered}
\end{equation}
Here $U_\delta$ is defined in~\eqref{e:u-delta}.
We now apply Proposition~\ref{l:reduce-model};
for $\delta$ small enough, $\widehat W,\widehat\Gamma_\pm$ are compact and
$\widehat\Gamma_\pm,\widehat \Lambda$
satisfy the conditions listed after~\eqref{e:smaller-stuff}.
Then~\eqref{e:symbol}
defines the principal symbol $\sigma_\Lambda(A)$ on a neighborhood
of $\widehat\Lambda$ in $\Lambda^\circ$ for each $A\in\II$.
\begin{theo}
\label{t:our-Pi}
Let the assumptions of~\S\S\ref{s:framework-assumptions} and~\ref{s:dynamics} hold for all $r$,
let $\Lambda^\circ$ be defined in~\eqref{e:the-Lambda} and
$\widehat\Lambda\subset\Lambda^\circ$ be given by~\eqref{e:w-hat}. Then there exists $\Pi\in \II$,
uniquely defined modulo $\mathcal O(h^\infty)$ microlocally near $\widehat\Lambda$,
such that the principal symbol of $\Pi$ is nonvanishing on $\widehat\Lambda$ and,
with $P\in\Psi^{\comp}(X)$ defined in Lemma~\ref{l:resolution},
\begin{align}
\label{e:th-idempotent}
\Pi^2-\Pi&=\mathcal O(h^\infty)\quad\text{microlocally near }\widehat\Lambda,\\
\label{e:th-commute}
[P,\Pi]&=\mathcal O(h^\infty)\quad\text{microlocally near }\widehat\Lambda.
\end{align}
Same can be said if we replace $\mathcal O(h^\infty)$ above by $\mathcal O(h^N)$,
require that the full symbol of $\Pi$ lies in $C^{3N}$ for some large $N$ (rather than being smooth),
and the assumptions of~\S\ref{s:dynamics} hold for $r$ large enough depending on $N$.
\end{theo}
\begin{proof}
We argue by induction, finding a family $\Pi_k$, $k\geq 1$, of microlocal idempotents of all orders
near $\widehat\Lambda$ (see Definition~\ref{d:microlocal-idempotent})
such that $[P,\Pi_k]=\mathcal O(h^{k+1})$ microlocally near $\widehat\Lambda$,
and taking their asymptotic limit to obtain $\Pi$.

We first construct $\Pi_1$.
Take the microlocal idempotent of all orders $\widetilde\Pi\in \II$ near $\widehat\Lambda$
constructed in Proposition~\ref{l:idempotent-exists}.
Since the Hamilton field of $p=\sigma(P)$ is tangent to $\Gamma_\pm$, $dp$ is annihilated by
the subbundles $\mathcal V_\pm$ from~\S\ref{s:projections}; therefore, 
$$
p(\rho_\pm)=p(\pi_\pm(\rho_\pm)),\quad \rho_\pm\in\Gamma_\pm^\circ;
$$
by~\eqref{e:good-calc-2} and~\eqref{e:good-calc-3},
$[P,\widetilde\Pi]=\mathcal O(h)$ microlocally near $\widehat\Lambda$.
We write $[P,\widetilde\Pi]=hS_0$ microlocally near $\widehat\Lambda$, where
$S_0\in \II$ and by part~4 of Proposition~\ref{l:idempotents},
\begin{equation}
  \label{e:sigma-s-0}
\sigma_\Lambda(S_0)(\rho_-,\rho_+)=\tilde a_0^-(\rho_-)\otimes
s_0^+(\rho_+)+s_0^-(\rho_-)\otimes\tilde a_0^+(\rho_+),
\end{equation}
with $s_0^\pm\in C^\infty(\widetilde\Gamma_\pm;\mathcal E_\pm)$ vanishing on $K$ near $\widehat K$ and
$\tilde a_0^\pm\in C^\infty(\widetilde\Gamma_\pm;\mathcal E_\pm)$ giving the principal symbol of
$\widetilde\Pi$ by~\eqref{e:idempotent-principal}. Here $\widetilde\Gamma_\pm$ are the neighborhoods of $\widehat\Gamma_\pm$
in $\Gamma_\pm^\circ$ defined in~\eqref{e:smaller-stuff}.

We look for $\Pi_1$ in the form
\begin{equation}
  \label{e:pi-1}
\Pi_1=e^{Q_0}\widetilde \Pi e^{-Q_0},
\end{equation}
where $Q_0\in\Psi^{\comp}(X)$ is compactly supported and
thus $e^{\pm Q_0}$ are pseudodifferential (see for example~\cite[Proposition~2.7]{zeeman}).
We calculate microlocally near $\widehat\Lambda$,
$$
e^{-Q_0}[P,\Pi_1]e^{Q_0}=[e^{-Q_0}Pe^{Q_0},\widetilde\Pi]=hS_0+[[P,Q_0],\widetilde\Pi]+\mathcal O(h^2).
$$
Here we use that $e^{-Q_0}Pe^{Q_0}=P+[P,Q_0]+\mathcal O(h^2)$. By~\eqref{e:sigma-s-0},
\eqref{e:good-calc-2}, \eqref{e:good-calc-3},
$$
\begin{gathered}
\sigma_\Lambda(S_0+h^{-1}[[P,Q_0],\widetilde \Pi])(\rho_-,\rho_+)\\
=\tilde a_0^-(\rho_-)\otimes(s_0^+(\rho_+)-iH_p\sigma(Q_0)(\rho_+)\tilde a_0^+(\rho_+))\\
+(s_0^-(\rho_-)+iH_p\sigma(Q_0)(\rho_-)\tilde a_0^-(\rho_-))\otimes \tilde a_0^+(\rho_+).
\end{gathered}
$$
It is thus enough to take any $Q_0$ such that for the restrictions
$q_0^\pm=\sigma(Q_0)|_{\widetilde\Gamma_\pm}$, the following transport equations
hold near $\widehat\Gamma_\pm$:
\begin{equation}
  \label{e:transport-0}
H_pq_0^\pm=\mp is_0^\pm/\tilde a_0^\pm,\quad
q_0^\pm|_{\widetilde K}=0.
\end{equation}
Such $q_0^\pm$ exist and are unique and smooth enough by Lemma~\ref{l:ode}, giving $\Pi_1$.
Note that Lemma~\ref{l:ode} can be applied near $\widehat\Gamma_\pm$, instead of the whole $\Gamma_\pm^\circ$,
since $e^{tH_p}(\widehat\Gamma_\pm)\subset \widehat\Gamma_\pm$ for $\mp t\geq 0$ by part~\eqref{c:c-pm}
of Lemma~\ref{l:phi-pm}.

Now, assume that we have constructed $\Pi_k$ for some $k>0$.
Let $a_0^\pm$ be the components of the principal symbol of $\Pi_k$ given
by~\eqref{e:idempotent-principal}. Then microlocally
near $\widehat\Lambda$,
$[P,\Pi_k]=h^{k+1}S_k$, where $S_k\in \II$ and
by part~4 of Proposition~\ref{l:idempotents},
$$
\sigma_\Lambda(S_k)(\rho_-,\rho_+)=a_0^-(\rho_-)\otimes s_k^+(\rho_+)
+s_k^-(\rho_-)\otimes a_0^+(\rho_+),
$$
where $s_k^\pm\in C^\infty(\widetilde\Gamma_\pm;\mathcal E_\pm)$ vanish on $K$ near $\widehat K$. 
We then take
\begin{equation}
  \label{e:pi-k}
\Pi_{k+1}=(1+h^kQ_k)\Pi_k(1+h^kQ_k)^{-1}
\end{equation}
where $Q_k$ is a compactly supported pseudodifferential operator. Microlocally near $\widehat\Lambda$,
$$
[P,\Pi_{k+1}]=h^{k+1}S_k+h^k[[P,Q_k],\Pi_k]+\mathcal O(h^{k+2}).
$$
Therefore, $q_k^\pm=\sigma(Q_k)|_{\widetilde\Gamma_\pm}$ need to satisfy the transport equations
near $\widehat\Gamma_\pm$
\begin{equation}
  \label{e:transport-k}
H_p q_k^\pm=\mp is_k^\pm/a_0^\pm,\quad
q_k^\pm|_{\widetilde K}=0.
\end{equation}
Such $q_k^\pm$ exist and are unique and smooth enough again by Lemma~\ref{l:ode}, giving $\Pi_{k+1}$.

To show that the operator $\Pi$ satisfying~\eqref{e:th-idempotent}
and~\eqref{e:th-commute} is unique microlocally near $\widehat \Lambda$,
we show by induction that each such $\Pi$ satisfies
$\Pi=\Pi_k+\mathcal O(h^k)$ microlocally near $\widehat\Lambda$.
First of all, $\Pi$ has the form~\eqref{e:pi-1}
for some operator $Q_0$ microlocally near $\widehat\Lambda$, by part~2
of Proposition~\ref{l:idempotents}; moreover, by the proof of this fact,
we can take $\sigma(Q_0)|_K=0$ near $\widehat K$.
 Now, $\sigma(Q_0)|_{\widehat\Gamma_\pm}$
are determined uniquely
by the transport equations~\eqref{e:transport-0}, and this gives
$\Pi=\Pi_1+\mathcal O(h)$ microlocally near $\widehat\Lambda$.
Next, if $\Pi=\Pi_k+\mathcal O(h^k)$ for some $k>0$, then,
as follows from the proof of Part~2 of Proposition~\ref{l:idempotents},
$\Pi$ has the form~\eqref{e:pi-k} for some operator $Q_k$ microlocally
near $\widehat\Lambda$, such that $\sigma(Q_k)|_K=0$
near $\widehat K$. Then $\sigma(Q_k)|_{\widehat\Gamma_\pm}$
are determined uniquely by the transport equations~\eqref{e:transport-k},
and this gives $\Pi=\Pi_{k+1}+\mathcal O(h^{k+1})$ microlocally near $\widehat\Lambda$.
\end{proof}

\subsection{Annihilating ideals}
  \label{s:construction-2}

Let $\Pi\in \II$ be the operator constructed in Theorem~\ref{t:our-Pi}.
In this section, we construct pseudodifferential operators $\Theta_\pm$
annihilating $\Pi$ microlocally near $\widehat\Lambda$; they are key
for the microlocal estimates in~\S\ref{s:resolvent-bounds}.
More precisely, we obtain
\begin{prop}
  \label{l:ideals}
If $\delta>0$ in the definition~\eqref{e:w-hat} of $\widehat W$
is small enough, then there exist compactly supported $\Theta_\pm\in\Psi^{\comp}(X)$ such that:
\begin{enumerate}
\item\label{i:ideals-1} $\Pi\Theta_-=\mathcal O(h^\infty)$ and $\Theta_+\Pi=\mathcal O(h^\infty)$
microlocally near $\widehat\Lambda$;
\item\label{i:ideals-2} $\sigma(\Theta_\pm)=\varphi_\pm$ near $\widehat W$, with
$\varphi_\pm$ defined in Lemma~\ref{l:phi-pm};
\item\label{i:ideals-3} if $P$ is the operator constructed in Lemma~\ref{l:resolution},
then
\begin{equation}
  \label{e:ideals-3}
[P,\Theta_-]=-ih\Theta_-Z_-+\mathcal O(h^\infty),\quad
[P,\Theta_+]=ih Z_+\Theta_++\mathcal O(h^\infty)
\end{equation}
microlocally near $\widehat W$, where $Z_\pm\in\Psi^{\comp}(X)$ are compactly
supported and $\sigma(Z_\pm)=c_\pm$ near $\widehat W$, with
$c_\pm$ defined in Lemma~\ref{l:phi-pm};
\item\label{i:ideals-4} if $\Im\Theta_+={1\over 2i}(\Theta_+-\Theta_+^*)$ and
$\zeta=\sigma(h^{-1}\Im\Theta_+)$, then
\begin{equation}
  \label{e:zeta}
H_p\zeta=- c_+\zeta- {1\over 2}\{\varphi_+,c_+\}\quad \text{on }\Gamma_+\text{ near }\widehat W;
\end{equation}
\item\label{i:ideals-5} there exists an operator $\Xi:C^\infty(X)\to C_0^\infty(X)$, satisfying
parts~1 and~2 of Proposition~\ref{l:Xi} and such that
\begin{equation}
  \label{e:Xi}
1-\Pi=\Theta_-\Xi+\mathcal O(h^\infty)\quad\text{microlocally near }\widehat W\times\widehat W.  
\end{equation}
\end{enumerate}
\end{prop}
\begin{proof}
The operators~$\Theta_\pm$ satisfying conditions~\eqref{i:ideals-1} and~\eqref{i:ideals-2} exist
by part~1 of Proposition~\ref{l:principal-ideal}. Next, since
$[P,\Pi]=\mathcal O(h^\infty)$ microlocally near $\widehat\Lambda$,
we find
$$
\Pi[P,\Theta_-]=\mathcal O(h^\infty),\quad
[P,\Theta_+]\Pi=\mathcal O(h^\infty)
$$
microlocally near $\widehat\Lambda$; condition~\eqref{i:ideals-3} now follows
from part~2 of Proposition~\ref{l:principal-ideal}. The symbols
$\sigma(Z_\pm)$ can be computed using the identity $H_p\varphi_\pm=\mp c_\pm \varphi_\pm$
from part~\eqref{c:c-pm} of Lemma~\ref{l:phi-pm}.
Condition~\eqref{i:ideals-5} follows immediately from Proposition~\ref{l:Xi}, keeping in mind
that by making $\delta$ small we can make $\widehat W$ contained in an arbitrary neighborhood of $\widehat K$.

Finally, we verify condition~\eqref{i:ideals-4}. Taking the adjoint of the identity $[P,\Theta_+]=ihZ_+\Theta_++\mathcal O(h^\infty)$
and using
that $P$ is self-adjoint, we get microlocally near $\widehat W$,
$$
[P,\Theta_+^*]=ih\Theta_+^*Z_+^*.
$$
Therefore, microlocally near $\widehat W$
$$
2[P,h^{-1}\Im\Theta_+]=Z_+\Theta_+-\Theta_+^*Z_+^*
=[Z_+,\Theta_+]+2i((\Im\Theta_+)Z_+^*+\Theta_+\Im Z_+).
$$
By comparing the principal symbols, we get~\eqref{e:zeta}.
\end{proof}

\section{Resolvent estimates}
  \label{s:resolvent-bounds}

In this section we give various estimates on the resolvent $\mathcal R(\omega)$,
in particular proving Theorem~\ref{t:gaps}.
In~\S\ref{s:estimate-full}, we reduce Theorem~\ref{t:gaps} to
a microlocal estimate in a neighborhood of the trapped set, which
is further split into two estimates: on the kernel
of the projector $\Pi$ given by Theorem~\ref{t:our-Pi}, proved in~\S\ref{s:estimate-kernel},
and on the image of $\Pi$, proved in~\S\ref{s:estimate-image}.
In~\S\ref{s:estimate-spectral}
we obtain a restriction on the wavefront set $\mathcal R(\omega)$ in $\omega$ on the image of $\Pi$,
needed in~\S\ref{s:trace}. Finally, in~\S\ref{s:resonant-states}, we discuss
the consequences of our methods for microlocal concentration of resonant states
and the corresponding semiclassical measures.

\subsection{Reduction to the trapped set}
  \label{s:estimate-full}

We take $\delta>0$ small enough so that the results
of~\S\ref{s:construction-1},\ref{s:construction-2} hold, and define following~\eqref{e:w-hat}
(with $\delta_1$ chosen in~\S\ref{s:projections}),
\begin{equation}
  \label{e:w'}
\widehat W:=U_\delta\cap p^{-1}([\alpha_0-\delta_1/2,\alpha_1+\delta_1/2]),\quad
W':=U_{\delta/2}\cap p^{-1}([\alpha_0-\delta_1/4,\alpha_1+\delta_1/4]),
\end{equation}
so that $W'$ is a neighborhood of $K\cap p^{-1}([\alpha_0,\alpha_1])$ compactly contained in $\widehat W$.
Here $U_\delta$ is defined in~\eqref{e:u-delta}.

For the reductions of this subsection, it is enough to assume that $\omega$ satisfies~\eqref{e:omega-region}.
The region~\eqref{e:omega-gaps} will arise as the intersection of the regions~\eqref{e:assumption-kernel}
and~\eqref{e:assumption-image} where the two components of the estimate will hold.

To prove Theorem~\ref{t:gaps}, it is enough to show the estimate
\begin{equation}
  \label{e:gaps2}
\|\tilde u\|_{\mathcal H_1}\leq Ch^{-2}\|\tilde f\|_{\mathcal H_2}+\mathcal O(h^\infty)
\end{equation}
for each $\tilde u=\tilde u(h)\in\mathcal H_1$
with $\|\tilde u\|_{\mathcal H_1}$ bounded polynomially in $h$ and for
$\tilde f=\mathcal P(\omega)\tilde u$, where $\omega=\omega(h)$ satisfies~\eqref{e:omega-gaps}.

Subtracting from $\tilde u$ the function $v$ constructed in Lemma~\ref{l:smart-parametrix},
we may assume that
$$
\WFh(\tilde f)\subset W'.
$$
Let $\mathcal S(\omega)$ be the operator constructed in Lemma~\ref{l:resolution},
$\mathcal S'(\omega)$ be its elliptic parametrix near $\mathcal U\supset\widehat W$
constructed in Lemma~\ref{l:eparametrix}, and put
$$
u:=\mathcal S(\omega)\tilde u,\quad
f:=\mathcal S'(\omega)\tilde f,
$$
so by~\eqref{e:conjugated}, for the operator $P$ constructed in Lemma~\ref{l:resolution},
\begin{equation}
  \label{e:assumption-general-1}
(P-\omega)u=f\quad\text{microlocally near }\widehat W,\quad
\WFh(f)\subset \widehat W.
\end{equation}
By ellipticity (Proposition~\ref{l:elliptic}) and since $\WFh(f)\subset W'$,
\begin{equation}
  \label{e:assumption-general-1.5}
\WFh(u)\cap\widehat W\subset p^{-1}([\alpha_0-\delta_1/4,\alpha_1+\delta_1/4]).
\end{equation}
Let $\varphi_\pm$ be the functions constructed in Lemma~\ref{l:phi-pm}.
By Lemma~\ref{l:propagate-outgoing}, $u$
satisfies the conditions (see Figure~\ref{f:reduction})
\begin{gather}
  \label{e:assumption-general-2}
\WFh(u)\cap \widehat W\subset \{|\varphi_+|\leq\delta/2\},\\
  \label{e:assumption-general-2.5}
\WFh(u)\cap\Gamma_-^\circ\subset W'.
\end{gather} 
Indeed, if $\rho\in\WFh(u)\cap\mathcal U$, then either
$\rho\in\Gamma_+$ (in which case~\eqref{e:assumption-general-2}
and~\eqref{e:assumption-general-2.5} follow immediately) or there exists $T\geq 0$ such that
for $\gamma(t)=e^{tH_p}(\rho)$, $\gamma([-T,0])\subset\overline{\mathcal U}$
and $\gamma(-T)\in\WFh(\tilde f)\subset W'$. In the second case, if $\rho\in\widehat W$, then
by convexity of $U_{\delta}$
(part~\eqref{c:convex} of Lemma~\ref{l:phi-pm}) we have
$\gamma([-T,0])\subset\widehat W$. To show~\eqref{e:assumption-general-2}, it remains to use that
$H_p\varphi_+^2\leq 0$ on $\widehat W$, following from part~\eqref{c:c-pm} of Lemma~\ref{l:phi-pm}.
For~\eqref{e:assumption-general-2.5}, note that if $\rho\in\Gamma_-$, then
$\gamma(-T)\in\Gamma_-\cap W'$; however, $e^{tH_p}(\Gamma_-\cap W')\subset \Gamma_-\cap W'$
for all $t\geq 0$ and thus $\rho\in W'$.
\begin{figure}
\includegraphics{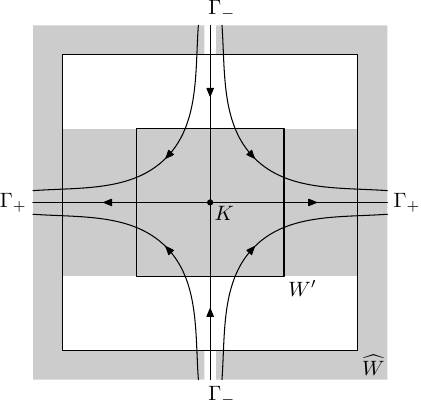}
\caption{A phase space picture of the geodesic flow near $\widehat W$.
The shaded region corresponds to~\eqref{e:assumption-general-2} and~\eqref{e:assumption-general-2.5}.}
  \label{f:reduction}
\end{figure}

By Lemma~\ref{l:smart-bound}, we reduce~\eqref{e:gaps2} to
\begin{equation}
  \label{e:gaps3}
\|A_1u\|_{L^2}\leq Ch^{-2}\|f\|_{L^2}+\mathcal O(h^\infty),
\end{equation}
where $A_1\in\Psi^{\comp}(X)$ is any compactly supported operator elliptic on $W'$.

Now, let $\Pi\in \II$ be the operator constructed in Theorem~\ref{t:our-Pi} in~\S\ref{s:construction-1}.
Note that
\begin{equation}
  \label{e:commutation-yay}
(P-\omega)\Pi u=\Pi f+\mathcal O(h^\infty)\quad\text{microlocally near }
\widehat W,
\end{equation}
since $[P,\Pi]=\mathcal O(h^\infty)$ microlocally near $\widehat W\times\widehat W$,
$\WFh(\Pi)\subset\Lambda^\circ\subset\Gamma_-^\circ\times\Gamma_+^\circ$,
and by~\eqref{e:assumption-general-2.5}.

We finally reduce~\eqref{e:gaps3} to the following two estimates,
which are proved in the following subsections:
\begin{prop}
  \label{l:estimate-kernel}
Assume that $u,f$ are $h$-tempered families
satisfying~\eqref{e:assumption-general-1}--\eqref{e:assumption-general-2.5} and
\begin{equation}
  \label{e:assumption-kernel}
\Re\omega\in [\alpha_0,\alpha_1],\quad
\Im\omega\in [-(\nu_{\min}-\varepsilon)h,C_0h].
\end{equation}
Then there exists compactly supported $A_1\in\Psi^{\comp}(X)$
elliptic on $W'$ such that
\begin{equation}
  \label{e:estimate-kernel}
\|A_1(1-\Pi)u\|_{L^2}\leq Ch^{-1}\|\Xi f\|_{L^2}+\mathcal O(h^\infty),
\end{equation}
where $\Xi$ is the operator from part~\eqref{i:ideals-5} of Proposition~\ref{l:ideals};
note that by part~1 of Proposition~\ref{l:Xi}, $\|\Xi\|_{L^2\to L^2}=\mathcal O(h^{-1})$.
\end{prop}
%
\begin{prop}
  \label{l:estimate-image}
Assume that $u,f$ are $h$-tempered families
satisfying~\eqref{e:assumption-general-1}--\eqref{e:assumption-general-2.5}
and
\begin{equation}
  \label{e:assumption-image}
\Re\omega\in[\alpha_0,\alpha_1],\quad
\Im\omega\in [-C_0h,C_0h]\setminus\Big(-{\nu_{\max}+\varepsilon\over 2}h,-{\nu_{\min}-\varepsilon\over 2}h\Big),
\end{equation}
Then there exists compactly supported $A_1\in\Psi^{\comp}(X)$ elliptic on $W'$
such that
\begin{equation}
  \label{e:estimate-image}
\|A_1\Pi u\|_{L^2}\leq Ch^{-1}\|\Pi f\|_{L^2}+\mathcal O(h^\infty).
\end{equation}
Note that by Proposition~\ref{l:bund} and the reduction to the model case
of~\S\ref{s:general}, we have $\|\Pi\|_{L^2\to L^2}=\mathcal O(h^{-1/2})$.
\end{prop}

\subsection{Estimate on the kernel of \texorpdfstring{$\Pi$}{Pi}}
  \label{s:estimate-kernel}

In this section, we prove Proposition~\ref{l:estimate-kernel}, which
is a microlocal estimate on the kernel of $\Pi$ (or equivalently, on
the image of $1-\Pi$).  We will use the identity~\eqref{e:Xi} together
with the commutator formula~\eqref{e:ideals-3} to effectively shift
the spectral parameter to the upper half-plane, where a standard
positive commutator argument gives us the estimate.

By~\eqref{e:commutation-yay}, we have microlocally near $\widehat W$,
\begin{equation}
  \label{e:kernel-int-1}
(P-\omega)(1-\Pi)u=(1-\Pi)f+\mathcal O(h^\infty)
\end{equation}
Let $\Theta_-\in\Psi^{\comp}(X)$ and $\Xi$ be the operators constructed in Proposition~\ref{l:ideals},
and denote
$$
v:=\Xi u,\quad
g:=\Xi f.
$$
Then microlocally near $\widehat W$,
\begin{equation}
  \label{e:kernel-int-1.5}
(1-\Pi)u=\Theta_- v,\quad
(1-\Pi)f=\Theta_- g.
\end{equation}
Indeed, by part~2 of Proposition~\ref{l:Xi},
\eqref{e:assumption-general-2.5}, and the fact that
$\WFh(\Pi)\subset\Lambda^\circ\subset\Gamma_-^\circ\times\Gamma_+^\circ$,
we see that $1-\Pi=\Theta_-\Xi+\mathcal O(h^\infty)$ microlocally near
$(\WFh(u)\setminus\widehat W)\times\widehat W$, since each of the
featured operators is microlocalized away from this region. Combining
this with~\eqref{e:Xi}, we see that $1-\Pi=\Theta_-\Xi+\mathcal
O(h^\infty)$ microlocally near $\WFh(u)\times\widehat W$, and thus
also near $\WFh(f)\times\widehat W$,
yielding~\eqref{e:kernel-int-1.5}.

By part~2 of Proposition~\ref{l:Xi} together
with~\eqref{e:assumption-general-1.5}--\eqref{e:assumption-general-2.5}
and part~\eqref{c:poisson} of Lemma~\ref{l:phi-pm},
\begin{gather}
  \label{e:wf-v-0}
\WFh(v)\cup\WFh(g)\subset p^{-1}([\alpha_0-\delta_1/4,\alpha_1+\delta_1/4]),\\
  \label{e:wf-v}
(\WFh(v)\cup\WFh(g))\cap\widehat W\subset \{|\varphi_+|\leq\delta/2\}.
\end{gather}

We now obtain a differential equation on $v$; the favorable imaginary part of the operator
in this equation, coming from commuting $\Theta_-$ with $P$, is the key component of the proof.
\begin{prop}
  \label{l:estimate-kernel-int}
Let $Z_-$ be the operator from~\eqref{e:ideals-3}. Then microlocally near $\widehat W$,
\begin{equation}
  \label{e:kernel-int-2}
(P-ihZ_--\omega)v=g+\mathcal O(h^\infty).
\end{equation}
\end{prop}
\begin{proof}
Given~\eqref{e:kernel-int-1.5}, the equation~\eqref{e:kernel-int-1} becomes
$(P-\omega)\Theta_-v=\Theta_- g+\mathcal O(h^\infty)$
microlocally near $\widehat W$. Using~\eqref{e:ideals-3}, we get
microlocally near $\widehat W$,
$$
\Theta_-(P-ihZ_--\omega)v=\Theta_-g+\mathcal O(h^\infty).
$$
To show~\eqref{e:kernel-int-2}, it remains
to apply propagation of singularities (part~2 of Proposition~\ref{l:microhyperbolic}),
for the operator $\Theta_-$. Indeed, by part~\eqref{c:poisson} of Lemma~\ref{l:phi-pm},
for each $\rho\in\widehat W$, there exists $t\geq 0$ such that
the Hamiltonian trajectory $\{e^{sH_{\varphi_-}}(\rho)\mid 0\leq s\leq t\}$
lies entirely inside $\widehat W$ and $e^{tH_{\varphi_-}}(\rho)$ lies
in $\{\varphi_+=-\delta\}$ and by~\eqref{e:wf-v} does not lie in $\WFh((P-ihZ_--\omega)v-\Theta_-g)$.
\end{proof}
We now use a positive commutator argument. Take a self-adjoint compactly
supported $\mathcal X_-\in\Psi^{\comp}(X)$ such that $\WFh(\mathcal X_-)$ is compactly
contained in $\widehat W$ and $\sigma(\mathcal X_-)=\chi(\varphi_-)$
near $\widehat W\cap\WFh(v)$, where $\varphi_-$ is defined in Lemma~\ref{l:phi-pm},
$\chi\in C_0^\infty(-\delta,\delta)$, $s\chi'(s)\leq 0$ everywhere,
and $\chi=1$ near $[-\delta/2,\delta/2]$.
This is possible by~\eqref{e:wf-v-0} and~\eqref{e:wf-v}.
Put $\Im\omega=h\nu$; by~\eqref{e:kernel-int-2} and since $P$ is self-adjoint,
\begin{equation}
  \label{e:pos-comm-1}
\begin{gathered}
\Im\langle\mathcal X_- v,g\rangle={h\over 2}\langle (Z_-^*\mathcal X_-+
\mathcal X_-Z_-+2\nu\mathcal X_-)v,v\rangle\\
+{1\over 2i}\langle[P,\mathcal X_-]v,v\rangle+\mathcal O(h^\infty)
=h\langle\mathcal  Y_- v,v\rangle+\mathcal O(h^\infty),
\end{gathered}
\end{equation}
where $\mathcal Y_-\in\Psi^{\comp}(X)$ is compactly supported, $\WFh(\mathcal Y_-)\subset\WFh(\mathcal X_-)\subset\widehat W$
and, using the function $c_-$ from part~\eqref{c:c-pm} of
Lemma~\ref{l:phi-pm} together with part~\eqref{i:ideals-3} of Proposition~\ref{l:ideals}, we write near $\widehat W\cap \WFh(v)$,
$$
\begin{gathered}
\sigma(\mathcal Y_-)=(c_-+\nu)\chi(\varphi_-)-{1\over 2}H_p\chi(\varphi_-)
=(c_-+\nu)\chi(\varphi_-)-{1\over 2}c_-\varphi_-\chi'(\varphi_-).
\end{gathered}
$$
However, $\nu\geq-(\nu_{\min}-\varepsilon)$ by~\eqref{e:assumption-kernel}
and by~\eqref{e:nu-bound},
$c_->(\nu_{\min}-\varepsilon)$ on $\widehat W$;
therefore
\begin{equation}
  \label{e:y-minus-positive}
\sigma(\mathcal Y_-)\geq 0\quad\text{near }\WFh(v),\qquad
\sigma(\mathcal Y_-)>0\quad\text{near }\WFh(v)\cap W'.
\end{equation}
To take advantage of~\eqref{e:y-minus-positive}, we use the following
combination of sharp G\r arding inequality with propagation of singularities:
\begin{lemm}
  \label{l:kernel-garding}
Assume that $Z,Q\in\Psi^{\comp}(X)$ are compactly supported,
$\WFh(Z),\WFh(Q)$ are compactly contained in $\widehat W$, 
$Z^*=Z$, and
$$
\sigma(Z)\geq 0\quad\text{near }\WFh(v),\quad
\sigma(Z)>0\quad\text{near }\WFh(v)\cap W'.
$$
Then
\begin{equation}
  \label{e:kernel-garding}
\|Qv\|_{L^2}^2\leq C\langle Zv,v\rangle
+Ch^{-2}\|g\|^2_{L^2}+\mathcal O(h^\infty).
\end{equation}
\end{lemm}
\begin{proof}
Without loss of generality, we may assume that $Q$ is elliptic on $\WFh(Z)\cup W'$.
There exists compactly supported $Q_1\in\Psi^{\comp}(X)$, elliptic on $W'$,
such that $\sigma(Z-Q_1^*Q_1)\geq 0$ near $\WFh(v)$
and $Q$ is elliptic on $\WFh(Q_1)$. Applying sharp G\r arding inequality (Proposition~\ref{l:garding})
to $Z-Q_1^*Q_1$, we get
\begin{equation}
  \label{e:kernel-garding-1}
\|Q_1 v\|_{L^2}^2\leq C\langle Z v,v\rangle+Ch\|Q v\|_{L^2}^2+\mathcal O(h^\infty).
\end{equation}
Now, by~\eqref{e:wf-v-0}, \eqref{e:wf-v}, and since $H_p\varphi_-^2>0$
on $\widehat W\setminus \Gamma_-$ by part~\eqref{c:c-pm}
of~Lemma~\ref{l:phi-pm}, each backwards flow line of $H_p$ starting on
$\WFh(Q)$ reaches either $\WFh(Q_1)$ or the complement of $\WFh(v)$,
while staying in $\widehat W$; by propagation of singularities
(Proposition~\ref{l:microhyperbolic}) applied
to~\eqref{e:kernel-int-2},
\begin{equation}
  \label{e:kernel-garding-2}
\|Qv\|_{L^2}\leq C\|Q_1 v\|_{L^2}+Ch^{-1}\|g\|_{L^2}+\mathcal O(h^\infty).
\end{equation}
Combining~\eqref{e:kernel-garding-1} and~\eqref{e:kernel-garding-2}, we get~\eqref{e:kernel-garding}.
\end{proof}
Now, there exists $A_1\in\Psi^{\comp}(X)$ compactly supported, elliptic on $W'$
and with $\WFh(A_1)$ compactly contained in $\widehat W$, such that the estimate
\begin{equation}
  \label{e:product-estimate}
|\langle \mathcal X_- v, g\rangle|\leq \tilde \varepsilon h \|A_1 v\|_{L^2}^2
+C_{\tilde\varepsilon}h^{-1}\|g\|_{L^2}^2+\mathcal O(h^\infty)
\end{equation}
holds for each $\tilde\varepsilon>0$ and constant $C_{\tilde\varepsilon}$
dependent on $\tilde\varepsilon$. Taking~$\tilde\varepsilon$ small enough and combining~\eqref{e:pos-comm-1},
\eqref{e:kernel-garding} (for $Z=\mathcal Y_-$ and $Q=A_1$), and~\eqref{e:product-estimate}, we arrive to
$$
\|A_1v\|_{L^2}\leq Ch^{-1}\|g\|_{L^2}+\mathcal O(h^\infty).
$$
Since $(1-\Pi)u=\Theta_-v$ microlocally near $\widehat W$, we get~\eqref{e:estimate-kernel}.

\subsection{Estimate on the image of \texorpdfstring{$\Pi$}{Pi}}
  \label{s:estimate-image}

In this section, we prove Proposition~\ref{l:estimate-image},
which is a microlocal estimate on the image of $\Pi$. We will
use the pseudodifferential operator $\Theta_+$ microlocally solving $\Theta_+\Pi=\mathcal O(h^\infty)$
to obtain an additional pseudodifferential equation satisfied by elements of the image of $\Pi$.
This will imply that for a pseudodifferential operator $A$ microlocalized near $K$,
the principal part of the expression $\langle A\Pi u(h),\Pi u(h)\rangle$ depends
only on the integral of $\sigma(A)$ over the flow lines of $\mathcal V_+$,
with respect to an appropriately chosen measure. A positive commutator estimate finishes the proof.

By~\eqref{e:commutation-yay}, we have microlocally near $\widehat W$,
\begin{equation}
  \label{e:eeq}
(P-\omega)\Pi u=\Pi f+\mathcal O(h^\infty).
\end{equation}
Let $\Theta_+\in\Psi^{\comp}(X)$ be the operator constructed in Proposition~\ref{l:ideals}, then
by~\eqref{e:assumption-general-2.5},
\begin{equation}
  \label{e:theta-yay}
\Theta_+\Pi u=\mathcal O(h^\infty)\quad\text{microlocally near }\widehat W.
\end{equation}
We start with
\begin{lemm}
  \label{l:estimate-image-int}
Let $\zeta:=\sigma(h^{-1}\Im\Theta_+)$. Take the function $\psi$ on $\Gamma_+\cap \widehat W$ such that
\begin{equation}
  \label{e:psi}
\{\varphi_+,\psi\}=2\zeta,\quad
\psi|_K=0.
\end{equation}
Assume that $A\in\Psi^{\comp}(X)$ satisfies $\WFh(A)\Subset\widehat W$ and
\begin{equation}
  \label{e:condiition}
\int(e^\psi\sigma(A))\circ e^{sH_{\varphi_+}}\,ds=0\quad\text{on }K.
\end{equation}
The integral in~\eqref{e:condiition}, and all similar integrals in this subsection,
is taken over the interval corresponding to a maximally extended flow line
of $H_{\varphi_+}$ in $\Gamma_+\cap\widehat W$.

Then there exists compactly supported $A_0\in\Psi^{\comp}(X)$ with $\WFh(A_0)\Subset\widehat W$ such that
$$
|\langle A\Pi u,\Pi u\rangle|\leq Ch\|A_0\Pi u\|_{L^2}^2+\mathcal O(h^\infty).
$$
\end{lemm}
\begin{proof}
By~\eqref{e:condiition}, there exists $q\in C_0^\infty(\widehat W)$
such that
$
\{\varphi_+,e^\psi q\}=e^\psi\sigma(A)$ on $\Gamma_+$.
We can rewrite this as
\begin{equation}
  \label{e:yay-int}
\{\varphi_+,q\}+2\zeta q=\sigma(A)\quad\text{on }\Gamma_+.
\end{equation}
Take $Q,Y\in\Psi^{\comp}(X)$ microlocalized inside $\widehat W$ and such that
$\sigma(Q)=q$ and
$$
\sigma(A)=\{\varphi_+,q\}+2\zeta q+\sigma(Y)\varphi_+.
$$
Then $A=(ih)^{-1}(Q\Theta_+-\Theta_+^*Q)+Y\Theta_++\mathcal O(h)_{\Psi^{\comp}}$
and thus for some $A_0$,
$$
\begin{gathered}
\langle A\Pi u,\Pi u\rangle={\langle Q\Theta_+\Pi u,\Pi u\rangle
-\langle Q\Pi u,\Theta_+\Pi u\rangle\over ih}+\langle Y\Theta_+\Pi u,\Pi u\rangle\\
+\mathcal O(h)\|A_0\Pi u\|_{L^2}^2+\mathcal O(h^\infty).
\end{gathered}
$$
The first three terms on the right-hand side
are $\mathcal O(h^\infty)$ by~\eqref{e:theta-yay}.
\end{proof}
Now, take compactly supported self-adjoint $\mathcal X_+\in\Psi^{\comp}(X)$
such that $\WFh(\mathcal X_+)$ is compactly contained in $\widehat W$ and
the symbol $\chi_+:=\sigma(\mathcal X_+)$ satisfies
$\chi_+\geq 0$ everywhere,
$\chi_+>0$ on $W'$, and 
\begin{equation}
  \label{e:normal}
\int (e^\psi\chi_+)\circ e^{sH_{\varphi_+}}\,ds=1\quad\text{on }K\cap p^{-1}([\alpha_0-\delta_1/4,
\alpha_1+\delta_1/4]).
\end{equation}
Putting $\Im\omega=h\nu$, we have by~\eqref{e:eeq}
\begin{equation}
  \label{e:iimage}
\begin{gathered}
\Im\langle \mathcal X_+\Pi u,\Pi f\rangle=h\nu\langle\mathcal X_+\Pi u,\Pi u\rangle
+{1\over 2i}\langle[P,\mathcal X_+]\Pi u,\Pi u\rangle+\mathcal O(h^\infty)\\
=h\langle \mathcal Y_+\Pi u,\Pi u\rangle+\mathcal O(h^\infty),
\end{gathered}
\end{equation}
where $\mathcal Y_+\in\Psi^{\comp}(X)$ is compactly supported, $\WFh(\mathcal Y_+)\subset \WFh(\mathcal X_+)\subset\widehat W$,
and
$$
\sigma(\mathcal Y_+)=\nu\chi_+-H_p\chi_+/2.
$$
We now want to use Lemma~\ref{l:estimate-image-int} together with G\r arding inequality to show that
$\langle\mathcal Y_+\Pi u,\Pi u\rangle$ has fixed sign, positive for $\nu\geq -(\nu_{\min}-\varepsilon)/2$
and negative for $\nu\leq -(\nu_{\max}+\varepsilon)/2$. For that, we need to integrate $\sigma(\mathcal Y_+)$
over the Hamiltonian flow lines of $\varphi_+$ on $\Gamma_+$,
with respect to the measure from~\eqref{e:condiition}. This relies on
\begin{lemm}
  \label{l:yay-int-2}
If $c_+$ is defined in Lemma~\ref{l:phi-pm}, then
\begin{equation}
  \label{e:yay-int-2}
\int (e^\psi H_p\chi_+)\circ e^{sH_{\varphi_+}}\,ds
=-c_+\quad\text{on }K\cap p^{-1}([\alpha_0-\delta_1/4,\alpha_1+\delta_1/4]).
\end{equation}
\end{lemm}
\begin{proof}
By part~\eqref{c:c-pm} of Lemma~\ref{l:phi-pm}, we have on $\Gamma_+\cap \widehat W$
$$
(e^{sH_{\varphi_+}})_*\partial_s (e^{-sH_{\varphi_+}})_*H_p
=-[H_p,H_{\varphi_+}]=c_+H_{\varphi_+}.
$$
Therefore, we can write (at $\rho\in K$ and $s$ such that $e^{sH_{\varphi_+}}(\rho)\in \widehat W$)
$$
(e^{-sH_{\varphi_+}})_* H_p=H_p+w(s)H_{\varphi_+}\quad\text{on }K
$$
where $w(s)$ is the smooth function on $K\times \mathbb R$ given by
$$
\partial_s w(s)=c_+\circ e^{sH_{\varphi_+}},\
w(0)=0.
$$
Now, differentiating~\eqref{e:normal} along $H_p$
and integrating by parts, we have on $K\cap p^{-1}([\alpha_0-\delta_1/4,\alpha_1+\delta_1/4])$
$$
\begin{gathered}
\int \big(H_p(e^\psi\chi_+)\big)\circ e^{sH_{\varphi_+}}\,ds
=\int (H_p+w(s)\partial_s)\big((e^\psi\chi_+)\circ e^{sH_{\varphi_+}}\big)\,ds\\
=-\int (e^\psi c_+\chi_+)\circ e^{sH_{\varphi_+}}\,ds;
\end{gathered}
$$
therefore,
\begin{equation}
  \label{e:yayay}
\int (e^\psi H_p\chi_+)\circ e^{sH_{\varphi_+}}\,ds
=-\int (e^\psi (c_++H_p\psi)\chi_+)\circ e^{sH_{\varphi_+}}\,ds.
\end{equation}
Now, we find on $\Gamma_+\cap \widehat W$ by~\eqref{e:psi} and~\eqref{e:zeta},
$$
H_{\varphi_+}H_p \psi=(H_p+c_+)H_{\varphi_+}\psi=2(H_p+c_+)\zeta
=-H_{\varphi_+}c_+.
$$
We have on $K\cap \widehat W$, $H_p\psi=0$; thus
$$
c_++H_p\psi=c_+\circ\pi_+\quad\text{on }\Gamma_+\cap \widehat W
$$
and by~\eqref{e:yayay} and~\eqref{e:normal},
on $K\cap p^{-1}([\alpha_0-\delta_1/4,\alpha_1+\delta_1/4])$,
$$
\int (e^\psi H_p\chi_+)\circ e^{sH_{\varphi_+}}\,ds
=-c_+\int (e^\psi\chi_+)\circ e^{sH_{\varphi_+}}\,ds=-c_+.
$$
This finishes the proof of~\eqref{e:yay-int-2}.
\end{proof}
Using~\eqref{e:normal}, \eqref{e:yay-int-2}, and Lemma~\ref{l:estimate-image-int}
(taking into account~\eqref{e:assumption-general-1.5}), we find
for some compactly supported $A_1\in\Psi^{\comp}$ with $\WFh(A_1)\subset\widehat W$
and $A_1$ elliptic on $W'\cup \WFh(\mathcal X_+)$,
$$
\langle \mathcal Y_+\Pi u,\Pi u\rangle=\langle\mathcal Z_+\Pi u,\Pi u\rangle+\mathcal O(h)\|A_1\Pi u\|_{L^2}^2
+\mathcal O(h^\infty)
$$
where $\mathcal Z_+\in\Psi^{\comp}(X)$ is any self-adjoint compactly supported operator
with $\WFh(\mathcal Z_+)\subset\widehat W$ and
$$
\sigma(\mathcal Z_+)=(\nu+(c_+\circ\pi_+)/2)\chi_+\quad\text{on }\Gamma_+.
$$
Then by~\eqref{e:iimage},
\begin{equation}
  \label{e:image-ultimate}
\Im\langle\mathcal X_+\Pi u,\Pi f\rangle=h\langle\mathcal Z_+\Pi u,\Pi u\rangle+\mathcal O(h^2)\|A_1\Pi u\|_{L^2}^2
+\mathcal O(h^\infty).
\end{equation}
Now, by~\eqref{e:nu-bound}, $\nu_{\min}-\varepsilon<c_+< \nu_{\max}+\varepsilon$ on $K$, therefore,
keeping in mind that $\WFh(\Pi u)\subset\Gamma_+^\circ$, we find
\begin{gather}
  \label{e:image-ineq-1}
\sigma(\mathcal Z_+)\geq 0\quad\text{near }\WFh(\Pi u)\quad\text{for }\nu\geq-(\nu_{\min}-\varepsilon)/2,
\\
  \label{e:image-ineq-2}
\sigma(\mathcal Z_+)\leq 0\quad\text{near }\WFh(\Pi u)\quad\text{for }\nu\leq-(\nu_{\max}+\varepsilon)/2.
\end{gather}
Moreover,
in both cases $\sigma(\mathcal Z_+)\neq 0$ on $\WFh(\Pi u)\cap W'$.
We now combine sharp G\r arding inequality and propagation of singularities
for the operator $\Theta_+$:
\begin{lemm}
  \label{l:image-garding}
Assume that $Z,Q\in\Psi^{\comp}(X)$ are compactly supported,
$\WFh(Z),\WFh(Q)$ are compactly contained in $\widehat W$, 
$Z^*=Z$, and
$$
\sigma(Z)\geq 0\quad\text{near }\WFh(\Pi u),\quad
\sigma(Z)>0\quad\text{near }\WFh(\Pi u)\cap W'.
$$
Then
\begin{equation}
  \label{e:image-garding}
\|Q\Pi u\|_{L^2}^2\leq C\langle Z\Pi u,\Pi u\rangle
+\mathcal O(h^\infty).
\end{equation}
\end{lemm}
\begin{proof}
We argue similarly to the proof of Lemma~\ref{l:kernel-garding}, with~\eqref{e:kernel-garding-2}
replaced by
\begin{equation}
  \label{e:image-garding-2}
\|Q\Pi u\|_{L^2}\leq C\|Q_1\Pi u\|_{L^2}+\mathcal O(h^\infty). 
\end{equation}
The estimate~\eqref{e:image-garding-2} follows from propagation of singularities (Proposition~\ref{l:microhyperbolic})
applied to~\eqref{e:theta-yay}. Indeed, by part~\eqref{c:poisson} of Lemma~\ref{l:phi-pm}
together with~\eqref{e:assumption-general-1.5}, for each $\rho\in\widehat W\cap\WFh(\Pi u)\subset\Gamma_+$,
there exists $t\in \mathbb R$ such that $e^{tH_{\varphi_+}}(\rho)\in W'$ and $e^{sH_{\varphi_+}}(\rho)\in \widehat W$
for each $s$ between $0$ and $t$.
\end{proof}
Using~\eqref{e:image-garding} (for $Z=\pm \mathcal Z_+,Q=A_1$), \eqref{e:image-ultimate},
and an analog of~\eqref{e:product-estimate}, we complete the proof of~\eqref{e:estimate-image}.

\subsection{Microlocalization in the spectral parameter}
  \label{s:estimate-spectral}
  
In this section, we provide a restriction on the wavefront set of solutions
to the equation $(P-\omega)u=f$ in the spectral parameter $\omega$, needed in~\S\ref{s:trace}.
We use the method of~\S\ref{s:estimate-image}, however since $\Re\omega$ is now a variable,
we will get an extra term coming from commutation with the multiplication operator by $\omega$.
Because of the technical difficulties of studying operators on product spaces (namely, a pseudodifferential
operator on $X$ does not give rise to a pseudodifferential operator on $X\times (\alpha_0,\alpha_1)$
since the corresponding symbol does not decay under differentiation in $\xi$ and thus does
not lie in the class $S^k$ of~\S\ref{s:prelim-basics}), we use the Fourier transform in the $\omega$ variable.
\begin{prop}
  \label{l:res-mic}
Fix $\nu\in [-C_0,C_0]$ and put $\omega=\alpha+ih\nu$, where $\alpha\in (\alpha_0,\alpha_1)$
is regarded as a variable. Assume that $u(x,\alpha;h)\in C([\alpha_0,\alpha_1];\mathcal H_1)$,
$f(x,\alpha;h)\in C([\alpha_0,\alpha_1];\mathcal H_2)$ have norms bounded polynomially in $h$,
satisfying~\eqref{e:assumption-general-1}--\eqref{e:assumption-general-2.5} uniformly in $\alpha$.
Define the semiclassical Fourier transform
\begin{equation}
  \label{e:fourier-transform}
\hat u(x,s;h)=\int_{\alpha_0}^{\alpha_1} e^{-{is\alpha\over h}}u(x,\alpha;h)\,d\alpha,
\end{equation}
and $\hat f(x,s;h)$ accordingly. Then there exists $A_1\in\Psi^{\comp}(X)$
elliptic on $W'$ such that:

1. If $\nu\geq-(\nu_{\min}-\varepsilon)/2$, then for any fixed $s_0\in\mathbb R$,
\begin{equation}
  \label{e:res-mic-1}
\qquad\|\Pi\hat f\|_{L^2_s((-\infty,s_0])L^2_{x}(X)}=\mathcal O(h^\infty)\ \Longrightarrow\ 
\|A_1\Pi\hat u(s_0)\|_{L^2_x(X)}=\mathcal O(h^\infty).
\end{equation}

2. If $\nu\leq-(\nu_{\max}+\varepsilon)/2$, then for any fixed $s_0\in\mathbb R$,
\begin{equation}
  \label{e:res-mic-2}
\|\Pi\hat f\|_{L^2_s([s_0,\infty))L^2_{x}(X)}=\mathcal O(h^\infty)\ \Longrightarrow\ 
\|A_1\Pi\hat u(s_0)\|_{L^2_x(X)}=\mathcal O(h^\infty).
\end{equation}
\end{prop}
\begin{proof}
We consider case~1; case~2 is handled similarly using~\eqref{e:image-ineq-2} instead of~\eqref{e:image-ineq-1}.
Since $u(\alpha),f(\alpha)$ are $h$-tempered uniformly in $\alpha$, their Fourier transforms
$\hat u(s),\hat f(s)$ are $h$-tempered and satisfy~\eqref{e:assumption-general-1}--\eqref{e:assumption-general-2.5}
in the $L^2$ sense in $s$; therefore, the corresponding $\mathcal O(h^\infty)$
errors will be bounded in $L^2_s$ for expressions linear in $\hat u,\hat f$ and
in $L^1_s$ for expressions quadratic in $\hat u,\hat f$. We also note that
for each $j$, the derivatives $\partial_s^j\hat u(s)$, $\partial_s^j\hat f(s)$
are $h$-tempered uniformly in $s\in\mathbb R$ and also in the $L^2$ sense in $s$. 

Taking the Fourier transform of~\eqref{e:commutation-yay}, we get
\begin{equation}
  \label{e:commutation-yayf}
(hD_s+P-ih\nu)\Pi\hat u(s)=\Pi \hat f(s)+\mathcal O(h^\infty)_{L^2_s(\mathbb R)}\quad\text{microlocally near }\widehat W.
\end{equation}
 We use the operators $\mathcal X_+, \mathcal Z_+,A_1$
from~\S\ref{s:estimate-image}.
Similarly to~\eqref{e:image-ultimate}, we find
$$
\begin{gathered}
\Im\langle \mathcal X_+\Pi \hat u(s),\Pi\hat f(s)\rangle
={h\over 2}\partial_s\langle\mathcal X_+\Pi\hat u(s),\Pi\hat u(s)\rangle\\
+h\langle \mathcal Z_+\Pi\hat u(s),\Pi\hat u(s)\rangle
+\mathcal O(h^2)\|A_1\Pi\hat u(s)\|_{L^2_x}^2
+\mathcal O(h^\infty)_{L^1_s(\mathbb R)}.
\end{gathered}
$$
Integrating this over $s\in (-\infty,s_0]$, by the assumption of~\eqref{e:res-mic-1}, we find
\begin{equation}
  \label{e:bananas}
\begin{gathered}
\langle\mathcal X_+\Pi\hat u(s_0),\Pi\hat u(s_0)\rangle+2\int_{-\infty}^{s_0}
\langle\mathcal Z_+\Pi\hat u(s),\Pi\hat u(s)\rangle\,ds\\
\leq Ch\|A_1\Pi \hat u(s)\|_{L^2_s((-\infty,s_0])L^2_x}^2+\mathcal O(h^\infty).
\end{gathered}
\end{equation}
Applying Lemma~\ref{l:image-garding} to $Q=A_1$ and $Z=\mathcal Z_+,\mathcal X_+$,
and using~\eqref{e:image-ineq-1}, we get
\begin{align}
  \label{e:zebra-garding}
\|A_1\Pi\hat u(s)\|_{L^2_x}^2&\leq C\langle \mathcal Z_+\Pi\hat u(s),\Pi\hat u(s)\rangle+\mathcal O(h^\infty)_{L^1_s(\mathbb R)},\\
  \label{e:zebra-garding2}
\|A_1\Pi\hat u(s_0)\|_{L^2_x}^2&\leq C\langle \mathcal X_+\Pi\hat u(s_0),\Pi\hat u(s_0)\rangle+\mathcal O(h^\infty).
\end{align}
Combining~\eqref{e:bananas} with~\eqref{e:zebra-garding}, integrated over $s\in (-\infty,s_0]$,
and~\eqref{e:zebra-garding2},
we get the conclusion of~\eqref{e:res-mic-1}.
\end{proof}

\subsection{Localization of resonant states}
  \label{s:resonant-states}
  
In this section, we study an application of the estimates of the preceding subsections to
microlocal behavior of resonant states, namely elements of
the kernel of $\mathcal P(\omega)$ for a resonance $\omega$. Assume that we are given a sequence $h_j\to 0$,
and $\omega(h)\in\mathbb C,\tilde u(h)\in\mathcal H_1$, defined
for $h$ in this sequence, such that
\begin{equation}
  \label{e:resonant-state}
\begin{gathered}
\mathcal P(\omega)\tilde u=0,\quad
\|\tilde u\|_{\mathcal H_1}=1;\\
\Re\omega\in [\alpha_0,\alpha_1],\quad
\Im\omega\in [-(\nu_{\min}-\varepsilon)h,C_0h];
\end{gathered}
\end{equation}
the condition on $\omega$ is just~\eqref{e:assumption-kernel}.
We also use the operators $\mathcal S(\omega)$ and $P$ from Lemma~\ref{l:resolution} and put
\begin{equation}
  \label{e:resonant-state-u}
u:=\mathcal S(\omega)\tilde u,
\end{equation}
so that
\begin{equation}
  \label{e:resonant-state-p}
(P-\omega)u=\mathcal O(h^\infty)\quad\text{microlocally near }
\mathcal U.
\end{equation}
We say that the sequence $u(h_j)$ converges to some Radon measure $\mu$ on
$T^*X$, and we call $\mu$ the semiclassical defect
measure of $u$ (see~\cite[Chapter~5]{e-z})
if for each compactly supported $A\in\Psi^0(M)$, we have
\begin{equation}
  \label{e:measures}
\langle Au,u\rangle\to \int_{T^*M}\sigma(A)\,d\mu\quad\text{as }h_j\to 0.
\end{equation}
Such $\mu$ is necessarily a nonnegative measure, see~\cite[Theorem~5.2]{e-z}.
\begin{theo}
  \label{t:resonant-states}
Let $\tilde u(h)$ be a sequence of resonant states corresponding
to some resonances $\omega(h)$, as in~\eqref{e:resonant-state},
and $u$ defined in~\eqref{e:resonant-state-u}. Take the
neighborhood $\widehat W$ of $K\cap p^{-1}([\alpha_0,\alpha_1])$ defined in~\eqref{e:w-hat}. Then:
\begin{enumerate}
\item\label{rs:1}
$\WFh(\tilde u)\cap \mathcal U\,\subset\, \Gamma_+\cap p^{-1}([\alpha_0,\alpha_1])$;
\item\label{rs:2}
for each $A_1\in \Psi^{\comp}(X)$ elliptic
on $K\cap p^{-1}([\alpha_0,\alpha_1])$, there exists
a constant $c>0$ independent of $h$ such that $\|A_1u\|_{L^2}\geq c$;
\item\label{rs:3}
$u=\Pi u+\mathcal O(h^\infty)$
and $\Theta_+u=\mathcal O(h^\infty)$ microlocally near $\widehat W$, where
$\Pi$ is constructed in Theorem~\ref{t:our-Pi} in~\S\ref{s:construction-1}
and $\Theta_+$ is the pseudodifferential operator from Proposition~\ref{l:ideals};
\item\label{rs:5}
there exists a smooth family of smooth measures $\mu_\rho$, $\rho\in K\cap p^{-1}([\alpha_0,\alpha_1])$, on
the flow line segments $\pi_+^{-1}(\rho)\cap\widehat W\subset\Gamma_+$ of $\mathcal V_+$,
independent of the choice of $u$, such that
if $u$ converges to some measure $\mu$ on $T^*M$ in the sense of~\eqref{e:measures},
and $\Re\omega(h_j)\to\omega_\infty$, $h^{-1}\Im\omega(h_j)\to\nu$ as $h_j\to 0$,
then $\mu|_{\widehat W}$ has the form
\begin{equation}
  \label{e:measure-form}
\mu|_{\widehat W}=\int_{K\cap p^{-1}(\omega_\infty)} \mu_\rho\,d\hat\mu(\rho),
\end{equation}
for some nontrivial measure $\hat\mu$ on $K\cap p^{-1}(\omega_\infty)$, such that
for each $b\in C^\infty(K)$, 
\begin{equation}
  \label{e:measure-eqn}
\int_{K\cap p^{-1}(\omega_\infty)} H_pb -(2\nu+c_+)b\,d\hat\mu=0,
\end{equation}
with the function $c_+$ defined in Lemma~\ref{l:phi-pm}.
\end{enumerate}
\end{theo}
\noindent\textbf{Remark}. The equation~\eqref{e:measure-eqn} is similar
to the equation satisfied by semiclassical defect measures for eigenstates
for the damped wave equation, see~\cite[(5.3.21)]{e-z}. 
\begin{proof}
Part~\eqref{rs:1} follows immediately from Lemma~\ref{l:propagate-outgoing},
part~\eqref{rs:2} follows from Lemma~\ref{l:smart-bound}
and implies that $\mu|_{\widehat W}$ is a nontrivial measure in part~\eqref{rs:5}.
By the discussion in \S\ref{s:estimate-full}, $u$ satisfies~\eqref{e:assumption-general-1}--\eqref{e:assumption-general-2.5},
with $f=0$.
The first statement of part~\eqref{rs:3} then follows from Proposition~\ref{l:estimate-kernel}. Indeed, we have
$(1-\Pi)u=\mathcal O(h^\infty)$ microlocally near the set $W'$ introduced
in~\eqref{e:w'}; it remains to apply propagation of singularities (Proposition~\ref{l:microhyperbolic})
to~\eqref{e:kernel-int-1}, using Lemma~\ref{l:the-flow}. The second statement of part~\eqref{rs:3} now follows from~\eqref{e:theta-yay}.

Finally, we prove part~\eqref{rs:5}. First of all, $\mu|_{\mathcal U}$ is supported on $\Gamma_+$
by part~\eqref{rs:1}, and on $p^{-1}(\omega_\infty)$ by~\eqref{e:resonant-state-p} and
the elliptic estimate (Proposition~\ref{l:elliptic}; see also~\cite[Theorem~5.3]{e-z}). 
Next, note that by Lemma~\ref{l:estimate-image-int} and since $u=\Pi u+\mathcal O(h^\infty)$ microlocally
near $\widehat W$,
we have for each $a\in C_0^\infty(\widehat W)$ and the function $\psi$ given by~\eqref{e:psi},
$$
\int(e^{\psi}a)(e^{s H_{\varphi_+}}(\rho))\,ds=0\quad\text{for all }\rho\in K\cap p^{-1}(\omega_\infty)\
\Longrightarrow\
\int a\,d\mu=0.
$$
This implies~\eqref{e:measure-form}, with
$$
\int a\,d\mu_\rho:=\int(e^{\psi}a)(e^{sH_{\varphi_+}}(\rho))\,ds,\quad
a\in C_0^\infty(\widehat W),\
\rho\in K\cap p^{-1}(\omega_\infty).
$$
To see~\eqref{e:measure-eqn}, we note that by~\eqref{e:resonant-state-p},
for each $a\in C_0^\infty(\widehat W)$ we have (see the derivation of~\cite[(5.3.21)]{e-z})
\begin{equation}
  \label{e:texas}
\int H_pa-2\nu a\,d\mu=0.
\end{equation}
Put $b(\rho)=\int a\,d\mu_\rho$ for $\rho\in K\cap p^{-1}(\omega_\infty)$. Similarly
to Lemma~\ref{l:yay-int-2} (replacing 1 by $b(\rho)$ on the right-hand side
of~\eqref{e:normal}), we compute
$$
\int H_pa\,d\mu_\rho=H_p b(\rho)-c_+(\rho)b(\rho),\quad
\rho\in K\cap p^{-1}(\omega_\infty) 
$$
and~\eqref{e:measure-eqn} follows by~\eqref{e:texas}.
\end{proof}

\section{Grushin problem}
  \label{s:grushin}

In this section, we construct a well-posed Grushin problem for the scattering resolvent,
representing resonances in the region~\eqref{e:assumption-kernel}
as zeroes of a certain determinant $F(\omega)$ defined in~\eqref{e:S-omega} below.
Together
with the trace formulas of~\S\ref{s:trace},
this makes possible the proof of the Weyl law in~\S\ref{s:weyl-law}.

We assume that the conditions of~\S\S\ref{s:framework-assumptions} and~\ref{s:dynamics}
hold, fix $\varepsilon>0$ (to be chosen in Theorem~\ref{t:weyl-law}),
and use the neighborhoods $W'\subset \widehat W$ of $K\cap p^{-1}([\alpha_0,\alpha_1])$
defined in~\eqref{e:w'}; let $\delta,\delta_1>0$ be the constants used to define these neighborhoods.
Take compactly supported
$Q_1,Q_2\in\Psi^{\comp}(X)$ such that (with $U_\delta$ defined in Lemma~\ref{l:phi-pm})
\begin{equation}
  \label{e:qz}
\begin{gathered}
Q_1=1+\mathcal O(h^\infty)\quad\text{microlocally near }U_{\delta/4}\cap p^{-1}([\alpha_0-\delta_1/6,\alpha_1+\delta_1/6]),\\
Q_2=1+\mathcal O(h^\infty)\quad\text{microlocally near }U_{\delta/3}\cap p^{-1}([\alpha_0-\delta_1/5,\alpha_1+\delta_1/5]),\\
\WFh(Q_1)\Subset U_{\delta/3}\cap p^{-1}([\alpha_0-\delta_1/5,\alpha_1+\delta_1/5]),\quad
\WFh(Q_2)\Subset W'.
\end{gathered}
\end{equation}
We will impose more restrictions on $Q_1$ later in Lemma~\ref{l:grushin-l}.

Using the operator $\mathcal P(\omega):\mathcal H_1\to\mathcal H_2$ from~\S\ref{s:framework-assumptions}
and the operator $\mathcal S(\omega)$ constructed in Lemma~\ref{l:resolution},
define the holomorphic family of operators
$$
\mathcal G(\omega):=\begin{pmatrix}
\mathcal P(\omega) & \mathcal S(\omega)Q_1\Pi Q_2 \\
Q_1 \Pi Q_2\mathcal S(\omega) & 1-Q_1\Pi Q_2
\end{pmatrix}
:\mathcal H_1\oplus L^2(X)\to \mathcal H_2\oplus L^2(X).
$$
Here $\Pi\in \II$ is the operator constructed in~Theorem~\ref{t:our-Pi} in~\S\ref{s:construction-1};
it is a microlocal idempotent commuting with the operator $P$ from Lemma~\ref{l:resolution}
microlocally near the set $\widehat\Lambda=\Lambda^\circ\cap (\widehat W\cap\widehat W)$.
Note that, since $Q_1,Q_2$ are microlocalized away from fiber infinity,
$\mathcal G(\omega)$ is a compact perturbation of
$\mathcal P(\omega)\oplus 1$, and therefore a Fredholm operator of index zero.

In this section, we will prove
\begin{prop}
  \label{l:grushin}
There exists a global constant%
\footnote{A more careful analysis, as in~\S\ref{s:resolvent-bounds},
could give the optimal value of $N$; we do not pursue this here since
the value of $N$ is irrelevant for our application in~\S\ref{s:weyl-law}.}
$N$ such that
for $\omega$ satisfying~\eqref{e:assumption-kernel},
\begin{equation}
  \label{e:grushin-bound}
\|\mathcal G(\omega)^{-1}\|_{\mathcal H_2\oplus L^2\to \mathcal H_1\oplus L^2}=\mathcal O(h^{-N}).
\end{equation}
Moreover, if
\begin{equation}
  \label{e:grushin-inverse}
\mathcal G(\omega)^{-1}=\begin{pmatrix}
\mathcal R_{11}(\omega) & \mathcal R_{12}(\omega) \\
\mathcal R_{21}(\omega) & \mathcal R_{22}(\omega)
\end{pmatrix},
\end{equation}
then $\mathcal R_{22}(\omega)=1-L_{22}(\omega)+\mathcal O(h^\infty)_{\mathcal D'\to C_0^\infty}$, where
$L_{22}(\omega)\in \II$ is microlocalized inside $\widehat\Lambda$
and the symbol $\sigma_\Lambda(L_{22})$ defined in~\eqref{e:symbol} satisfies
\begin{equation}
  \label{e:l-22-symbol}
\sigma_\Lambda(L_{22}(\omega))(\rho,\rho)={\sigma(Q_1)(\rho)^2+(p(\rho)-\omega)\sigma(Q_1)(\rho)\over
\sigma(Q_1)(\rho)^2+(p(\rho)-\omega)(\sigma(Q_1)(\rho)-1)},\quad
\rho\in \widehat K.
\end{equation}
\end{prop}
To prove~\eqref{e:grushin-bound},
we consider families of distributions $ u(h)\in\mathcal H_1$,
$f(h)\in \mathcal H_2$, $ v(h), g(h)\in L^2(X)$, bounded
polynomially in $h$ in the indicated spaces and satisfying
$\mathcal G(u, v)=(f, g)$, namely
\begin{align}
  \label{e:grushin-eq-1}
\mathcal P(\omega) u+\mathcal S(\omega) Q_1\Pi Q_2 v&= f,\\
  \label{e:grushin-eq-2}
Q_1\Pi Q_2 \mathcal S(\omega) u+(1-Q_1\Pi Q_2) v&= g.
\end{align}
Note that by~\eqref{e:conjugated}, \eqref{e:grushin-eq-1} implies
\begin{equation}
  \label{e:grushin-eq-1'}
(P-\omega)\mathcal S(\omega) u+ Q_1\Pi Q_2  v=\mathcal S'(\omega) f+\mathcal O(h^\infty)\quad
\text{microlocally near }\mathcal U.
\end{equation}
Here $\mathcal S'(\omega)$ is an elliptic parametrix of $\mathcal S(\omega)$ near $\mathcal U$
constructed in Proposition~\ref{l:eparametrix}.

To show~\eqref{e:grushin-bound}, it is enough to establish the bound
\begin{equation}
  \label{e:grushin-bound-2}
\| u\|_{\mathcal H_1}+\| v\|_{L^2}\leq
Ch^{-N}(\| f\|_{\mathcal H_2}+\| g\|_{L^2})+\mathcal O(h^\infty).
\end{equation}
We start with a technical lemma:
\begin{lemm}
  \label{l:grushin-l}
There exists $Q_1\in\Psi^{\comp}(X)$ satisfying~\eqref{e:qz}
and such that
\begin{equation}
  \label{e:grushin-l}
\sigma(Q_1)^2+(p-\omega)(\sigma(Q_1)-1)\neq 0\quad\text{on }K\text{ for all }
\omega\in [\alpha_0,\alpha_1].
\end{equation}
\end{lemm}
\begin{proof}
It suffices to take $Q_1$ such that
$\sigma(Q_1)|_K=\psi(p)$, where $\psi\in C_0^\infty(\alpha_0-\delta_1/5,\alpha_1+\delta_1/5)$ is equal to
1 near $[\alpha_0-\delta_1/6,\alpha_1+\delta_1/6]$ and 
\begin{equation}
  \label{e:egg-2}
\psi(\lambda)^2+(\lambda-\omega)(\psi(\lambda)-1)\neq 0,\quad
\lambda\in\mathbb R,\ \omega\in [\alpha_0,\alpha_1].
\end{equation}
We now show that such $\psi$ exists. 
The equation~\eqref{e:egg-2} holds automatically for $\lambda\not\in (\alpha_0-\delta_1/5,\alpha_1+\delta_1/5)$,
as $\psi=0$ there and the left-hand side of~\eqref{e:egg-2} equals
$\omega-\lambda\neq 0$. This however also shows that a real-valued $\psi$
with the desired properties does not exist. We take $\Re\psi\in C_0^\infty(\alpha_0-\delta_1/5,\alpha_1+\delta_1/5)$
equal to 1 near $[\alpha_0-\delta_1/6,\alpha_1+\delta_1/6]$ and take values in $[0,1]$ and
$\Im\psi\in C_0^\infty(\alpha_1+\delta_1/6,\alpha_1+\delta_1/5)$
a nonnegative function to be chosen later.
Then the left-hand side of~\eqref{e:egg-2} is equal to $1$ for $\lambda\in [\alpha_0-\delta_1/6,\alpha_1+\delta_1/6]$
and is positive for $\lambda\in [\alpha_0-\delta_1/5,\alpha_0-\delta_1/6]$. Next, the imaginary part of~\eqref{e:egg-2}
is
$$
\Im\psi(\lambda)(2\Re\psi(\lambda)+\lambda-\omega).
$$
Since $2\Re\psi(\lambda)+\lambda-\omega>0$ for $\lambda\in [\alpha_1+\delta_1/6,\alpha_1+\delta_1/5]$,
it remains to take $\Im\psi(\lambda)>0$ on a large compact
subinterval of $(\alpha_1+\delta_1/6,\alpha_1+\delta_1/5)$; then $\psi$ satisfies~\eqref{e:egg-2}.
\end{proof}
Using Lemma~\ref{l:grushin-l}, we determine $v$ microlocally outside of the elliptic region:
\begin{prop}
  \label{l:grushin-1}
Let $Q_1$ be chosen in Lemma~\ref{l:grushin-l}.
Then there exist $L_{21}^e(\omega),L_{22}^e(\omega)\in \II$ holomorphic in $\omega$,
microlocalized inside $\widehat\Lambda$, and
such that for all $ u, v, f, g$ satisfying~\eqref{e:grushin-eq-1},
\eqref{e:grushin-eq-2},
\begin{equation}
  \label{e:grushin-1-eq}
 v=L_{21}^e f+(1-L_{22}^e) g
\end{equation}
microlocally outside of $\Gamma_+\cap\widehat W\cap p^{-1}([\alpha_0-\delta_1/8,\alpha_1+\delta_1/8])$.
Moreover, $\sigma_\Lambda(L_{22}^e)$ satisfies~\eqref{e:l-22-symbol} for $\rho\not\in p^{-1}([\alpha_0-\delta_1/8,\alpha_1+\delta_1/8])$.
\end{prop}
\begin{proof}
Using Proposition~\ref{l:eparametrix}, construct compactly supported $R^e(\omega)\in\Psi^{\comp}(X)$
such that $R^e(\omega)(P-\omega)=1+\mathcal O(h^\infty)$ microlocally
near $\widehat W\setminus p^{-1}(\alpha_0-\delta_1/8,\alpha_1+\delta_1/8)$.
By~\eqref{e:grushin-eq-1'}, we get
$$
\mathcal S(\omega) u=R^e(\omega)(\mathcal S'(\omega) f-Q_1\Pi Q_2 v)+\mathcal O(h^\infty)
$$
microlocally near $\widehat W\setminus p^{-1}(\alpha_0-\delta_1/8,\alpha_1+\delta_1/8)$.
Substituting this into~\eqref{e:grushin-eq-2}, we get
\begin{equation}
  \label{e:egg-0}
(1-L') v= g-Q_1\Pi Q_2 R^e(\omega)\mathcal S'(\omega) f+\mathcal O(h^\infty)
\end{equation}
microlocally outside of $\Gamma_+\cap \widehat W\cap p^{-1}([\alpha_0-\delta_1/8,\alpha_1+\delta_1/8])$, where
$L'=Q_1\Pi Q_2(1+R^e(\omega)Q_1\Pi Q_2)\in \II$ and $\WFh(L')\subset\widehat\Lambda$.

Let $\sigma_\Lambda(L')$ be the symbol of $L'$, defined in~\eqref{e:symbol}.
By~\eqref{e:good-calc-1}--\eqref{e:good-calc-3}, and since $\sigma_\Lambda(\Pi)|_K=1$
near $\widehat W$ (see part~1 of Proposition~\ref{l:idempotents}
or~\S\ref{s:construction-1}), we find
for $\rho\in K\setminus p^{-1}(\alpha_0-\delta_1/8,\alpha_1+\delta_1/8)$,
$$
\sigma_\Lambda(L')(\rho,\rho)=\sigma(Q_1)(\rho)(1+\sigma(Q_1)(\rho)/(p(\rho)-\omega));
$$
it follows from~\eqref{e:grushin-l} that
\begin{equation}
  \label{e:egg-1}
\sigma_\Lambda(L')|_K\neq 1\quad\text{outside of }
p^{-1}(\alpha_0-\delta_1/8,\alpha_1+\delta_1/8).
\end{equation}
By Proposition~\ref{l:funny-parametrix}, there exists $L_{22}^e(\omega)\in \II$,
with $\WFh(L_{22}^e)\subset \widehat\Lambda$, such that
$(1-L_{22}^e)(1-L')=1+\mathcal O(h^\infty)$ microlocally outside of $p^{-1}([\alpha_0-\delta_1/8,\alpha_1+\delta_1/8])$,
and note that the symbol $\sigma_\Lambda(L_{22}^e)$ satisfies~\eqref{e:l-22-symbol}
for $\rho\not\in p^{-1}([\alpha_0-\delta_1/8,\alpha_1+\delta_1/8])$ by~\eqref{e:funny-symbol}.
By~\eqref{e:egg-0}, we get~\eqref{e:grushin-1-eq} with
$L_{12}^e(\omega)=-(1-L_{22}^e(\omega))Q_1\Pi Q_2 R^e(\omega)\mathcal S'(\omega)$.
\end{proof}
By Proposition~\ref{l:grushin-1},
replacing $ v$ by $A_e v$, where $A_e\in\Psi^{\comp}(X)$ is compactly
supported, $\WFh(A_e)\subset \mathcal U\cap p^{-1}(\alpha_0-\delta_1/7,\alpha_1+\delta_1/7)$,
and $A_e=1+\mathcal O(h^\infty)$ near $\widehat W\cap p^{-1}([\alpha_0-\delta_1/8,\alpha_1+\delta_1/8])$,
we see that it is enough to prove~\eqref{e:grushin-bound-2} in the case
\begin{equation}
  \label{e:gcase-1}
\WFh(v)\subset \mathcal U\cap p^{-1}([\alpha_0-\delta_1/7,\alpha_1+\delta_1/7]).
\end{equation}
Using Lemma~\ref{l:smart-parametrix}, consider $u'\in\mathcal H_1$ such that
$\|u'\|_{\mathcal H_1}\leq Ch^{-1}\| f\|_{\mathcal H_2}$
and $\WFh(\mathcal P(\omega)u'- f)\subset \WFh(Q_1)\cap p^{-1}([\alpha_0-\delta_1/7,\alpha_1+\delta_1/7])$. Subtracting $u'$ from $ u$,
we see that is suffices to prove~\eqref{e:grushin-bound-2} for the case
\begin{equation}
  \label{e:gcase-2}
\WFh(f)\subset \WFh(Q_1)\cap p^{-1}([\alpha_0-\delta_1/7,\alpha_1+\delta_1/7]).
\end{equation}
By~\eqref{e:gcase-1}, the wavefront set of $\mathcal P(\omega) u= f-\mathcal S(\omega)Q_1\Pi Q_2 v$
satisfies~\eqref{e:gcase-2}. Arguing as in~\S\ref{s:estimate-full}, and keeping
in mind~\eqref{e:grushin-eq-1'}, we see
that $ u$ satisfies~\eqref{e:assumption-general-1.5}--\eqref{e:assumption-general-2.5};
in fact, \eqref{e:assumption-general-1.5} can be strengthened to
\begin{equation}
  \label{e:assumption-general-1.75}
\WFh(u)\cap \widehat W\subset p^{-1}([\alpha_0-\delta_1/7,\alpha_1+\delta_1/7]).
\end{equation}
and~\eqref{e:assumption-general-2.5} can be strengthened to
\begin{equation}
  \label{e:assumption-general-2.75}
\WFh(u)\cap\Gamma_-^\circ\subset U_{\delta/3}\cap p^{-1}([\alpha_0-\delta_1/7,\alpha_1+\delta_1/7]).
\end{equation}
We can now solve for $v$:
\begin{prop}
  \label{l:horseradish}
Assume that $u,v,f,g$ satisfy~\eqref{e:grushin-eq-1}, \eqref{e:grushin-eq-2}, \eqref{e:gcase-1},
\eqref{e:gcase-2}. Then
\begin{equation}
  \label{e:horseradish}
v=Q_1\Pi\mathcal S'(\omega)f+(1-Q_1(P-\omega+1)\Pi Q_2)g+\mathcal O(h^\infty)_{C_0^\infty}.
\end{equation}
\end{prop}
\begin{proof}
Since $\Pi^2=\Pi+\mathcal O(h^\infty)$ microlocally near $\widehat W\times\widehat W$
and $Q_1=1+\mathcal O(h^\infty)$ microlocally near $K\cap p^{-1}([\alpha_0-\delta_1/6,\alpha_1+\delta_1/6]$, we have
\begin{equation}
  \label{e:kinda-i}
\Pi Q_1 \Pi=\Pi+\mathcal O(h^\infty)\quad
\text{microlocally near }(\widehat W\cap p^{-1}([\alpha_0-\delta_1/6,\alpha_1+\delta_1/6]))\times\widehat W.
\end{equation}
We rewrite~\eqref{e:grushin-eq-2} as
\begin{equation}
  \label{e:grushin-eq-2.0}
Q_1\Pi Q_2(\mathcal S(\omega)  u- g)+(1-Q_1\Pi Q_2)( v- g)=0.
\end{equation}
It follows immediately that $\WFh( v- g)\subset\WFh(Q_1)$
and thus $Q_2(v-g)=v-g+\mathcal O(h^\infty)_{C_0^\infty}$.
Also, by~\eqref{e:grushin-eq-2}, \eqref{e:gcase-1}, and~\eqref{e:assumption-general-1.75},
$\WFh(g)\subset\mathcal U\cap p^{-1}([\alpha_0-\delta_1/7,\alpha_1+\delta_1/7])$.
Applying $\Pi$ to~\eqref{e:grushin-eq-2.0} and using~\eqref{e:gcase-1}, \eqref{e:assumption-general-1.75},
and~\eqref{e:kinda-i},
we get $\Pi Q_2\mathcal S(\omega) u-\Pi Q_2 g=\mathcal O(h^\infty)$
microlocally near $\widehat W$. By~\eqref{e:assumption-general-2.75},
we have $\Pi Q_2\mathcal S(\omega)u=\Pi\mathcal S(\omega)u+\mathcal O(h^\infty)_{C_0^\infty}$;
therefore,
\begin{equation}
  \label{e:pi-u-bound}
\Pi \mathcal S(\omega) u=\Pi Q_2 g+\mathcal O(h^\infty)\quad
\text{microlocally near }\widehat W.
\end{equation}
Then~\eqref{e:grushin-eq-2.0} becomes
\begin{equation}
  \label{e:grushin-eq-2.1}
v=Q_1\Pi Q_2 v+(1-Q_1\Pi Q_2)g+\mathcal O(h^\infty)_{C_0^\infty}.
\end{equation}
Applying $\Pi$ to~\eqref{e:grushin-eq-1'},
using that $[P,\Pi]=\mathcal O(h^\infty)$ microlocally near $\widehat W\times\widehat W$,
and keeping in mind~\eqref{e:assumption-general-2.75},
we get
\begin{equation}
  \label{e:giraffe}
(P-\omega)\Pi\mathcal S(\omega) u+\Pi Q_2 v=\Pi\mathcal S'(\omega) f+\mathcal O(h^\infty)
\quad\text{microlocally near }\widehat W.
\end{equation}
Together, \eqref{e:pi-u-bound} and~\eqref{e:giraffe} give
$$
\Pi Q_2 v=\Pi \mathcal S'(\omega)f-(P-\omega)\Pi Q_2 g+\mathcal O(h^\infty)
\quad\text{microlocally near }\widehat W.
$$
By~\eqref{e:grushin-eq-2.1}, we now get~\eqref{e:horseradish}.
\end{proof}
By Proposition~\ref{l:horseradish}, we see that
\begin{equation}
  \label{e:vbund}
\|v\|_{L^2}\leq Ch^{-N}(\|f\|_{\mathcal H_2}+\|g\|_{L^2})+\mathcal O(h^\infty).
\end{equation}
By Proposition~\ref{l:estimate-kernel} (using~\eqref{e:grushin-eq-1'} instead of~\eqref{e:assumption-general-1}), we get
for some $A_1\in\Psi^{\comp}(X)$ elliptic near $W'$,
$$
\|A_1(1-\Pi) \mathcal S(\omega)u\|_{L^2}\leq Ch^{-N}(\|f\|_{\mathcal H_2}+\|g\|_{L^2})+\mathcal O(h^\infty).
$$
Combining this with~\eqref{e:pi-u-bound}, we estimate $\|A_1u\|_{L^2}$
by the right-hand side of~\eqref{e:vbund}.  Applying
Lemma~\ref{l:smart-bound} to~\eqref{e:grushin-eq-1}, we can estimate
$\|u\|_{\mathcal H_1}$ by the same quantity, completing the proof
of~\eqref{e:grushin-bound-2}.

It remains to describe the operator $\mathcal R_{22}$
from~\eqref{e:grushin-inverse}.  We assume that $u,v,f,g$
satisfy~\eqref{e:grushin-eq-1}, \eqref{e:grushin-eq-2} and $f=0$; then
$\mathcal R_{22}g=v$. By Proposition~\ref{l:grushin-1},
$v=(1-L_{22}^e)g+\mathcal O(h^\infty)$ microlocally outside of
$\widehat W\cap p^{-1}([\alpha_0-\delta_1/8,\alpha_1+\delta_1/8])$; it
then suffices to describe $v$ microlocally near $\widehat W\cap
p^{-1}([\alpha_0-\delta_1/8,\alpha_1+\delta_1/8])$.  Let $A_e$ be the
operator introduced before~\eqref{e:gcase-1} and $R^e(\omega)$ be an
elliptic parametrix for $P-\omega$ constructed in the proof of
Proposition~\ref{l:grushin-1}.  Replacing $(u,v)$ by $(u+\mathcal
S'(\omega)R^e(\omega)Q_1\Pi Q_2 (1-A_e)v,A_e v)$, we may assume
that~\eqref{e:gcase-1} and~\eqref{e:gcase-2} hold, and in fact the
resulting $f$ is $\mathcal O(h^\infty)_{C_0^\infty}$
and the resulting $g$ coincides with the original $g$
microlocally near $\widehat W\cap p^{-1}([\alpha_0-\delta_1/8,\alpha_1+\delta_1/8])$.
By Proposition~\ref{l:horseradish}, we now get for the original $v$ and
$g$,
$$
v=(1-Q_1(P-\omega+1)\Pi Q_2)g+\mathcal O(h^\infty)\quad
\text{microlocally near }\widehat W\cap p^{-1}([\alpha_0-\delta_1/8,\alpha_1+\delta_1/8]).
$$
Note that $Q_1(P-\omega+1)\Pi Q_2\in \II$
and its principal symbol satisfies~\eqref{e:l-22-symbol} in $p^{-1}([\alpha_0-\delta_1/8,\alpha_1+\delta_1/8])$,
since $\sigma(Q_1)|_K=1$ in that region. This finishes the proof of Proposition~\ref{l:grushin}.

By Proposition~\ref{l:grushin}, $\mathcal R_{22}(\omega)-1$ is a compactly supported operator
mapping $H^{-N}_h\to H^N_h$ for all $N$, therefore
it is trace class. We can then define the determinant
(see for instance~\cite[(A.6.38)]{tay1})
\begin{equation}
  \label{e:S-omega}
F(\omega):=\det \mathcal R_{22}(\omega),
\end{equation}
which is holomorphic in the region~\eqref{e:assumption-kernel} and
$F(\omega)=0$ if and only if $\mathcal R_{22}(\omega)$ is not invertible
(see~\cite[Proposition~A.6.16]{tay1}).
The key properties of $F$ needed in~\S\ref{s:weyl-law} are established in
\begin{prop}
  \label{l:grushin-ultimate}
1. Resonances in the region~\eqref{e:assumption-kernel} coincide (with the
multiplicities defined in~\eqref{e:multiplicity})
with zeroes of $F(\omega)$.

2. For some constants $C$ and $N$, we have 
$|F(\omega)|\leq e^{Ch^{-N}}$ for $\omega$ in~\eqref{e:assumption-kernel},
and $|F(\omega)|\geq e^{-Ch^{-N}}$ for $\omega$ in
the resonance free region~\eqref{e:omega-gaps}.

3. For $\omega$ in the resonance free region~\eqref{e:omega-gaps}, we have
$$
{\partial_\omega F(\omega)\over F(\omega)}
=-\Tr((1-Q_1\Pi Q_2-Q_1\Pi \mathcal S(\omega)\mathcal R(\omega)\mathcal S(\omega) Q_1\Pi Q_2)\partial_\omega L_{22}(\omega))+\mathcal O(h^\infty).
$$
Here $L_{22}(\omega)$ is defined in Proposition~\ref{l:grushin}.
\end{prop}
\begin{proof}
1. By Schur's complement formula~\cite[(D.1.1)]{e-z}, and since
$\mathcal G(\omega)$ is invertible by Proposition~\ref{l:grushin}, we
know that $\mathcal P(\omega)$ is invertible if and only if $\mathcal
R_{22}(\omega)$ is, and in fact
\begin{equation}
  \label{e:schur}
\mathcal P(\omega)^{-1}=\mathcal R_{11}(\omega)-\mathcal R_{12}(\omega)\mathcal R_{22}(\omega)^{-1}
\mathcal R_{21}(\omega).
\end{equation}
To see that the multiplicity of a resonance $\omega_0$ defined
by~\eqref{e:multiplicity} coincides with the multiplicity of
$\omega_0$ as a zero of the function $F(\omega)$ (and in particular,
to demonstrate that the multiplicity defined by~\eqref{e:multiplicity}
is a positive integer), it is enough to show that
\begin{equation}
  \label{e:multiplicity2}
{1\over 2\pi i}\Tr\oint_{\omega_0}\mathcal P(\omega)^{-1}\partial_\omega \mathcal P(\omega)\,d\omega
={1\over 2\pi i}\Tr\oint_{\omega_0}\mathcal R_{22}(\omega)^{-1}\partial_\omega\mathcal R_{22}(\omega)\,d\omega;
\end{equation}
indeed, since $\partial_\omega\mathcal R_{22}(\omega)$ is trace class,
we can put the trace inside the integral on the right-hand side
of~\eqref{e:multiplicity2}, yielding $\partial_\omega
F(\omega)/F(\omega)$; therefore, the right-hand side gives the
mutliplicity of $\omega_0$ as a zero of $F(\omega)$ by the argument
principle.

Since $\partial_\omega (\mathcal G(\omega)^{-1})=-\mathcal
G(\omega)^{-1}(\partial_\omega \mathcal G(\omega))\mathcal
G(\omega)^{-1}$, we have
$$
\partial_\omega \mathcal R_{22}(\omega)=-\mathcal R_{21}(\omega)(\partial_\omega \mathcal P(\omega))\mathcal R_{12}(\omega)
+\mathcal A(\omega)\mathcal R_{22}(\omega)+\mathcal R_{22}(\omega)\mathcal B(\omega), 
$$
where $\mathcal A(\omega),\mathcal B(\omega):L^2(X)\to L^2(X)$ are bounded operators holomorphic at $\omega_0$.
By~\eqref{e:schur}, \eqref{e:multiplicity2} follows from the two identities
$$
\begin{gathered}
\Tr\oint_{\omega_0} \mathcal R_{12}(\omega)\mathcal R_{22}(\omega)^{-1}\mathcal R_{21}(\omega)\partial_\omega\mathcal P(\omega)\,d\omega
=\Tr\oint_{\omega_0} \mathcal R_{22}(\omega)^{-1}\mathcal R_{21}(\omega) (\partial_\omega \mathcal P(\omega))\mathcal R_{12}(\omega)\,d\omega,\\
\Tr\oint_{\omega_0} \mathcal R_{22}(\omega)^{-1}(\mathcal A(\omega)\mathcal R_{22}(\omega)+\mathcal R_{22}(\omega)\mathcal B(\omega))\,d\omega=0.
\end{gathered}
$$
Both of them follow from the cyclicity of the trace, replacing
$\mathcal R_{22}(\omega)^{-1}$ by its finite-dimensional principal
part at $\omega_0$ and putting the trace inside the integral.

2. By Proposition~\ref{l:grushin}, the trace class norm $\|\mathcal R_{22}(\omega)-1\|_{\Tr}$ is bounded
polynomially in $h$. Using the bound $|\det (1+T)|\leq e^{\|T\|_{\Tr}}$
(see for example~\cite[(A.6.44)]{tay1}), we get
$|F(\omega)|\leq e^{Ch^{-N}}$.
By Theorem~\ref{t:gaps},
we have $\|\mathcal R(\omega)\|_{\mathcal H_2\to \mathcal H_1}\leq Ch^{-2}$ when $\omega$ satisfies~\eqref{e:omega-gaps}.
Using Schur's complement formula again, we get
\begin{equation}
  \label{e:schur-strikes-back}
\mathcal R_{22}(\omega)^{-1}=1-Q_1\Pi Q_2-Q_1\Pi Q_2 \mathcal S(\omega)\mathcal R(\omega)\mathcal S(\omega)Q_1\Pi Q_2.
\end{equation}
Then $\|\mathcal R_{22}(\omega)^{-1}-1\|_{\Tr}\leq Ch^{-N}$ and thus
$|F(\omega)|^{-1}=|\det (\mathcal R_{22}(\omega)^{-1})|\leq e^{Ch^{-N}}$.

3. By Proposition~\ref{l:grushin}, we have $\partial_\omega \mathcal R_{22}(\omega)
=-\partial_\omega L_{22}(\omega)+\mathcal O(h^\infty)_{\mathcal D'\to \mathcal C_0^\infty}$, thus
$$
{\partial_\omega F(\omega)\over F(\omega)}=-\Tr(\mathcal R_{22}(\omega)^{-1}\partial_\omega L_{22}(\omega))+\mathcal O(h^\infty).
$$
By~\eqref{e:schur-strikes-back}, it then suffices to prove that
$$
\Tr(Q_1\Pi (1-Q_2) \mathcal S(\omega) \mathcal R(\omega)\mathcal S(\omega)Q_1\Pi Q_2\partial_\omega L_{22}(\omega))=
\mathcal O(h^\infty).
$$
For that, it suffices to show that the intersection of the wavefront set of the operator
on the left-hand side with the diagonal in $T^*X$ is empty. We assume the contrary,
then there exists $\rho\in T^*X$ such that
$$(\rho,\rho)\in\WFh(Q_1\Pi (1-Q_2) \mathcal S(\omega) \mathcal R(\omega)\mathcal S(\omega)Q_1\Pi Q_2\partial_\omega L_{22}(\omega)).
$$
Since both $\Pi$ and $\partial_\omega L_{22}$ are microlocalized inside
$\Lambda^\circ\subset\Gamma_-^\circ\cap\Gamma_+^\circ$, we see that
$\rho\in K^\circ=\Gamma_+^\circ\cap\Gamma_-^\circ$. There exists $\rho'\in T^*X$ such that
$$
(\rho,\rho')\in\WFh(\mathcal S(\omega)\mathcal R(\omega) \mathcal S(\omega)Q_1\Pi Q_2\partial_\omega L_{22}(\omega)),\quad
(\rho',\rho)\in \WFh(Q_1\Pi (1-Q_2)).
$$
For any $h$-tempered $f\in L^2(X)$, we have $\WFh(\mathcal S(\omega)Q_1\Pi Q_2 \partial_\omega L_{22}(\omega)f)\subset
\Gamma_+^\circ\cap \widehat W$, therefore by Lemma~\ref{l:propagate-outgoing}
we have $\WFh(\mathcal R(\omega)\mathcal S(\omega)Q_1\Pi Q_2 \partial_\omega L_{22}(\omega)f)\cap \mathcal U\subset\Gamma_+$. It follows that $\rho'\in\Gamma_+$. Since $(\rho',\rho)\in\WFh(Q_1\Pi (1-Q_2))$, we see
that $\rho'=\rho\in K^\circ$. However, then $\rho\in\WFh(Q_1)\cap \WFh(1-Q_2)$, which is impossible
since $Q_2=1+\mathcal O(h^\infty)$ microlocally near $\WFh(Q_1)$.
\end{proof}

\section{Trace formula}
  \label{s:trace}

In this section, we establish an asymptotic expansion for contour integrals of the
logarithmic derivative of the determinant $F(\omega)$ of the effective Hamiltonian
of the Grushin problem of~\S\ref{s:grushin}, defined in~\eqref{e:S-omega}. By
Proposition~\ref{l:grushin-ultimate}, this reduces to computing contour integrals
of operators of the form $\Pi\mathcal R(\omega)$, where
$\Pi$ is the projector constructed in Theorem~\ref{t:our-Pi} in~\S\ref{s:construction-1}.
This in turn is done
by approximating $\mathcal R(\omega)$ microlocally on the image of $\Pi$
by pseudodifferential operators, using Schr\"odinger
propagators and microlocalization in the spectral parameter established in~\S\ref{s:estimate-spectral}.

We operate under the pinching condition~\eqref{e:pinching} of Theorem~\ref{t:weyl-law},
namely $\nu_{\max}<2\nu_{\min}$, and choose $\varepsilon>0$
such that $\nu_{\max}+\varepsilon<2(\nu_{\min}-\varepsilon)$.
Take $\chi\in C_0^\infty(\alpha_0,\alpha_1)$
with $\alpha_0,\alpha_1$ from~\eqref{e:omega-region}.
Consider an almost analytic extension $\tilde\chi(\omega)$ of $\chi$,
that is $\tilde\chi\in C^\infty(\mathbb C)$ such that
$\tilde\chi|_{\mathbb R}=\chi$ and $\partial_{\bar\omega} \tilde\chi(\omega)=\mathcal O(|\Im\omega|^{\infty})$.
We may take $\tilde\chi$ such that $\supp(\tilde\chi)\subset \{\Re\omega\in (\alpha_0,\alpha_1)\}$.

The main result of this section is
\begin{prop}
  \label{l:trace}
Take
\begin{equation}
  \label{e:nu-trace}
\nu_-\in \Big[-(\nu_{\min}-\varepsilon),-{\nu_{\max}+\varepsilon\over 2}\Big],\quad
\nu_+\in \Big[-{\nu_{\min}-\varepsilon\over 2},C_0\Big].
\end{equation}
Let $F(\omega)$ be defined in~\eqref{e:S-omega} and put
\begin{equation}
  \label{e:trace-stuff}
\mathcal I^\pm_\chi:=(2\pi h)^{n-1}\int_{\Im\omega=h\nu_\pm} \tilde\chi(\omega){\partial_\omega F(\omega)\over F(\omega)}\,d\omega.
\end{equation}
Then, with $d\Vol_\sigma=\sigma_S^{n-1}/(n-1)!$ the symplectic volume form,
\begin{equation}
  \label{e:trace}
\mathcal I_\chi^--\mathcal I_\chi^+=2\pi i\int_K \chi(p)\,d\Vol_\sigma + \mathcal O(h).
\end{equation}
\end{prop}
\noindent\textbf{Remark}.
More precise trace formulas are possible; in particular, one can get
a full asymptotic expansion in $h$ of each of $\mathcal I^\pm_\chi$. For simplicity,
we prove here a less general version which suffices for the analysis of~\S\ref{s:weyl-law}.

\smallskip

The key feature of the expansions for the integrals~\eqref{e:trace-stuff}, which produces a nontrivial
asymptotics for resonances in Theorem~\ref{t:weyl-law}, is that the
principal part of $\mathcal I_\chi^\pm$ depends on the sign of $\pm$.
The reason for this dependence is the difference of directions for propagation in the resolvent approximation
$\mathcal R_\psi^\pm$ of Proposition~\ref{l:trace-approx} for the two cases; this in turn
is explained by the difference between~\eqref{e:res-mic-1} and~\eqref{e:res-mic-2},
which is due to the difference of the signs of the `commutator' $\mathcal Z_+$
between~\eqref{e:image-ineq-1} and~\eqref{e:image-ineq-2}.

We start the proof by using Proposition~\ref{l:res-mic}
to replace $\mathcal R(\omega)$ in the formula for $\partial_\omega F(\omega)/F(\omega)$ from
Proposition~\ref{l:grushin-ultimate}
by an operator $\mathcal R^\pm_\psi(\omega)$ obtained by integrating the Schr\"odinger propagator $e^{- it(P-\omega)/h}$
over a bounded range of times $t$.
\begin{prop}
  \label{l:trace-approx}
Fix $\psi\in C_0^\infty(\mathbb R)$ such that $\psi=1$ near zero.
For $\omega\in\mathbb C$,
define the operators $\mathcal R^\pm_\psi(\omega):L^2(X)\to L^2(X)$ by
\begin{gather}
  \label{e:rr-1}
\mathcal R^+_\psi(\omega):={i\over h}\int_{-\infty}^0 e^{is(P-\omega)/h}\psi(s)\,ds;\\
  \label{e:rr-2}
\mathcal R^-_\psi(\omega):=-{i\over h}\int_0^\infty e^{is(P-\omega)/h}\psi(s)\,ds;
\end{gather}
Then, if $\supp\psi$ is contained in a small enough neighborhood of zero,
\begin{equation}
  \label{e:trace-approx}
\begin{gathered}
\mathcal I_\chi^\pm=-(2\pi h)^{n-1}\Tr\int_{\Im\omega=h\nu_\pm}\tilde\chi(\omega)
(1-Q_1\Pi Q_2\\-Q_1\mathcal R_\psi^\pm(\omega)\Pi Q_1\Pi Q_2)\partial_\omega L_{22}(\omega)\,d\omega
+\mathcal O(h^\infty).
\end{gathered}
\end{equation}
\end{prop}
\begin{proof}
We concentrate on the case of $\mathcal I_\chi^+$,
the case of~$\mathcal I_\chi^-$ is handled similarly,
using~\eqref{e:res-mic-2} in place of~\eqref{e:res-mic-1}.
We denote $\omega=\alpha+ih\nu_+$, where $\alpha\in (\alpha_0,\alpha_1)$.
By part~3 of Proposition~\ref{l:grushin-ultimate}, it suffices to prove the trace norm bound
$\|\mathcal W\|_{\Tr}=\mathcal O(h^\infty)$, where
$$
\mathcal W:=\int_{\Im\omega=h\nu_+}\tilde\chi(\omega)
Q_1(\Pi\mathcal S(\omega)\mathcal R(\omega)\mathcal S(\omega)-\mathcal R_\psi^+(\omega)\Pi)Q_1\Pi Q_2\partial_\omega L_{22}(\omega)\,d\omega.
$$
Since $\mathcal W$ is compactly supported and microlocalized away from the fiber infinity,
we can write $\mathcal W=\mathcal Z\mathcal W+\mathcal O(h^\infty)_{\Psi^{-\infty}(X)}$
for some compactly supported $\mathcal Z\in\Psi^{\comp}(X)$. Since $\|\mathcal Z\|_{\Tr}$ is bounded
polynomially in $h$ (see for instance~\cite[Chapter~9]{d-sj}), and trace class operators
form an ideal in the algebra of bounded operators on $L^2$,
it suffices to prove the bound
$$
\|\mathcal W\|_{L^2\to L^2}=\mathcal O(h^\infty).
$$
Take arbitrary $h$-independent family $\tilde f=\tilde f(h)\in L^2(X)$ with $\|\tilde f\|_{L^2}\leq 1$ and put
$$
f(\alpha):=\tilde\chi(\omega) Q_1\Pi Q_2\partial_\omega L_{22}(\omega)\tilde f,\quad
u(\alpha):=\mathcal S(\omega)\mathcal R(\omega)\mathcal S(\omega)f(\alpha).
$$
Then $f(x,\alpha)$ is compactly supported in both
$x\in X$ and $\alpha\in (\alpha_0,\alpha_1)$, $\|f\|_{L^\infty_\alpha L^2_x}$ is polynomially bounded in $h$,
and $\WFh(f(\alpha))\subset \Gamma_+\cap W'$. Since $\mathcal R(\omega)_{\mathcal H_2\to \mathcal H_1}=\mathcal
O(h^{-2})$ by Theorem~\ref{t:gaps}, we see
that $u(\alpha)\in\mathcal H_2$ is compactly supported in $\alpha\in (\alpha_0,\alpha_1)$ and
the norm $\|u\|_{L^\infty_\alpha L^2_x}$ is bounded polynomially in $h$. Using Lemma~\ref{l:propagate-outgoing}
similarly to~\S\ref{s:estimate-full}, we see that $u,f$ satisfy~\eqref{e:assumption-general-1}--\eqref{e:assumption-general-2.5}, uniformly in $\alpha$.

It now suffices to prove that for each choice of $\tilde f$, independent of $\alpha$, we have
\begin{equation}
  \label{e:trace-int}
\int_{\alpha_0}^{\alpha_1}Q_1(\Pi u(\alpha)-R^+_\psi(\omega)\Pi f(\alpha))\,d\alpha=\mathcal O(h^\infty)_{L^2}.
\end{equation}
Define the semiclassical Fourier transforms~$\hat u(s),\hat f(s)$ by~\eqref{e:fourier-transform}.
Then~\eqref{e:trace-int} becomes
\begin{equation}
  \label{e:trace-int-2}
Q_1\bigg(\Pi\hat u(0)-{i\over h}\int_{-\infty}^0e^{is(P-ih\nu_+)/h}\psi(s)\Pi\hat f(s)\,ds\bigg)=\mathcal O(h^\infty)_{L^2}.
\end{equation}
By~\eqref{e:commutation-yayf} and Proposition~\ref{l:schrodinger}, we find
microlocally near $W'$,
\begin{equation}
  \label{e:bazooka}
\Pi\hat u(0)={i\over h}\int_{-\infty}^0e^{is(P-ih\nu_+)/h}(\psi(s)\Pi\hat f(s)-ih\psi'(s)\Pi\hat u(s))\,ds
+\mathcal O(h^\infty).
\end{equation}
Take $\tilde\varepsilon>0$ such that $\psi=1$ near $[-\tilde \varepsilon,\tilde\varepsilon]$,
so that $\psi'(s)$ is compactly supported in $\{|s|>\tilde \varepsilon\}$. 
Since $\chi(\omega)$ and $\partial_\omega L_{22}(\omega)$ depend smoothly on $\alpha$,
we see that $\|\partial^j_\alpha f(\alpha)\|_{L^\infty_\alpha L^2_x}=\mathcal O(h^{-1/2})$ for all $j$.
By repeated integration by parts, we get
$$
\|\hat f(s)\|_{L^2_s((-\infty,-\tilde\varepsilon])L^2_x}=\mathcal O(h^\infty).
$$
Then by~\eqref{e:res-mic-1}, $\Pi\hat u(s)=\mathcal O(h^\infty)$
microlocally near $W'$
locally uniformly in $s\in (-\infty,-\tilde\varepsilon]$,
and thus $Q_1e^{is(P-ih\nu_+)/h}\Pi\hat u(s)=\mathcal O(h^\infty)_{L^2}$
uniformly in $s\in (-\infty,0]\cap \supp \psi'$. By~\eqref{e:bazooka},
we now get~\eqref{e:trace-int-2}.
\end{proof}
Now, note that, since the expression under the integral in~\eqref{e:trace-approx}
is almost analytic in $\omega$, we can replace the integral over $\Im\omega=h\nu_\pm$
by the integral over the real line, with an $\mathcal O(h^\infty)$ error. Then
$$
\begin{gathered}
\mathcal I^-_\chi-\mathcal I^+_\chi=(2\pi h)^{n-1}\Tr\mathcal A_\chi
+\mathcal O(h^\infty),\\
\mathcal A_\chi:=\int_{\mathbb R}\chi(\alpha)\partial_\alpha L_{22}(\alpha)
Q_1(\mathcal R^-_\psi(\alpha)-\mathcal R^+_\psi(\alpha))\Pi Q_1\Pi Q_2\,d\alpha.
\end{gathered}
$$
Proposition~\ref{l:trace} now follows from Proposition~\ref{l:trace-basic},
the fact that $\WFh(\mathcal A_\chi)\subset \widehat W\times\widehat W$,
and the following
\begin{prop}
  \label{l:trace-kiwi}
The operator $\mathcal A_\chi$ lies in $\II$ and its principal symbol, as defined by~\eqref{e:symbol},
satisfies $\sigma_\Lambda(\mathcal A_\chi)\circ j_K=2\pi i \chi(p)$, with $j_K:K^\circ\to\Lambda^\circ$ defined in~\eqref{e:j-k}.
\end{prop}
\begin{proof}
Given the multiplication formula~\eqref{e:good-calc-1}, the fact that $\sigma(Q_1)=\sigma(Q_2)=1$ and
$\sigma_\Lambda(\Pi)\circ j_K=1$ on $K\cap p^{-1}([\alpha_0,\alpha_1])$
and $\supp\chi\subset (\alpha_0,\alpha_1)$, it is enough to prove the proposition with $\mathcal A_\chi$ replaced by
$$
\mathcal A'_\chi:=-{i\over h}\int_{\mathbb R^2}e^{-is\alpha/h}\chi(\alpha)\partial_\alpha L_{22}(\alpha)Q_1 e^{isP/h}\psi(s)\,dsd\alpha.
$$
Denote $\mathcal L(\alpha)=\chi(\alpha)\partial_\alpha L_{22}(\alpha)Q_1$; it is an operator
in $\II$. By applying a microlocal partition of unity to $\mathcal L(\alpha)$, we may reduce to the case
when both $\mathcal L(\alpha)$ and $e^{isP/h}$ have local parametrizations (see~\eqref{e:lagrangian}
for the first one and for example~\cite[Theorem~10.4]{e-z} for the second one)
$$
\begin{gathered}
\mathcal L(\alpha)u(x)=(2\pi h)^{-(N+n)/2}\int_{\mathbb R^{N+n}} e^{{i\over h}\Phi(x,y,\theta)}a(x,y,\theta,\alpha;h)u(y)\,dyd\theta,\\
e^{isP/h}u(y)=(2\pi h)^{-n} \int_{\mathbb R^{2n}} e^{{i\over h}(S(y,\zeta,s)-z\cdot\zeta)}b(y,\zeta,s;h)u(z)\,dzd\zeta.
\end{gathered}
$$
Here $S(y,\zeta,s)=y\cdot\zeta+sp(y,\zeta)+\mathcal O(s^2)$ and $b(y,\zeta,0;0)=1$.
Then $\mathcal A'_\chi$ takes the form
$$
\begin{gathered}
\mathcal A'_\chi u(x)=-ih^{-1}(2\pi h)^{-(N+3n)/2}\int_{\mathbb R^{N+3n}}
e^{{i\over h}(\Phi(x,y,\theta)+S(y,\zeta,s)-z\cdot\zeta-s\alpha)}\\
a(x,y,\theta,\alpha;h)
b(y,\zeta,s;h)\psi(s) u(z)\,dy d\theta dzd\zeta dsd\alpha.
\end{gathered}
$$
We now apply the method of stationary phase in the $y,\zeta,s,\alpha$ variables.
The stationary points are given by $s=0$, $\alpha=p(z,\zeta)$,
$y=z$, $\zeta=-\partial_z \Phi(x,z,\theta)$. We get
$$
\mathcal A'_\chi u(x)=-2\pi i (2\pi h)^{-(N+n)/2}\int_{\mathbb R^{N+n}}
e^{{i\over h}\Phi(x,z,\theta)}c(x,z,\theta;h)u(z)\,d\theta dz,
$$
where $c$ is a classical symbol and $c(x,z,\theta;0)=a(x,z,\theta,p(z,-\partial_z\Phi(x,z,\zeta));0)$.
It follows that $\mathcal A'_\chi\in \II$ and
$\sigma_\Lambda(\mathcal A'_\chi)(\rho_-,\rho_+)=-2\pi i\sigma_\Lambda(\mathcal L(p(\rho_-)))(\rho_-,\rho_+)$.
By~\eqref{e:l-22-symbol}, $\sigma_\Lambda(L_{22}(\alpha))(\rho,\rho)=p(\rho)-\alpha+1$ 
when $\rho\in K\cap p^{-1}([\alpha_0,\alpha_1])$,
and thus $\sigma_\Lambda(\partial_\alpha L_{22}(\alpha))(\rho,\rho)=-1$. Therefore, we find
$\sigma_\Lambda(\mathcal A'_\chi)(\rho,\rho)=2\pi i\chi(p(\rho))$ for $\rho\in K$.
\end{proof}

\section{Weyl law for resonances}
  \label{s:weyl-law}
  
In this section, we prove Theorem~\ref{t:weyl-law}, using the Grushin problem from~\S\ref{s:grushin},
the trace formula of~\S\ref{s:trace}, and several tools from complex analysis. The argument
below is quite standard, see for instance~\cite{markus,sj,sj-bottles}, and is simplified
by the fact that we do not aim for the optimal $\mathcal O(h)$ remainder in the Weyl
law, instead carrying out the argument in a rectangle of width $\sim 1$ and height $\sim h$.
For more sophisticated techniques needed to obtain the optimal remainder, see~\cite{sj-dwe}.

First of all, by Proposition~\ref{l:grushin-ultimate}, resonances in the region of interest
are (with multiplicities) the zeroes of the holomorphic function $F(\omega)$ defined in~\eqref{e:S-omega}.
Take $\alpha''_0\in (\alpha_0,\alpha'_0)$, $\alpha''_1\in (\alpha'_1,\alpha_1)$.
Fix $\nu_\pm$ satisfying~\eqref{e:nu-trace} and
let $\{\omega_j\}_{j=1}^{M(h)}$ denote the set of zeroes (counted with multiplicities) of $F(\omega)$ in
the region (see Figure~\ref{f:complex})
\begin{figure}
\includegraphics{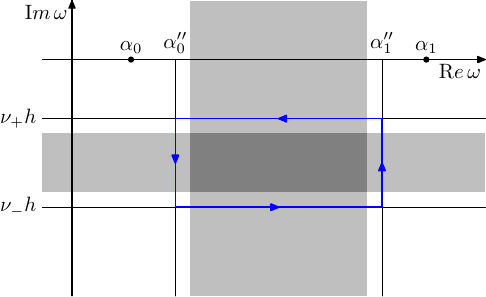}
\caption{The contour $\partial\Omega(h)$ (in blue). The horizontal shaded region is
$\{\Im\omega\in (-(\nu_{\max}+\varepsilon)h/2,-(\nu_{\min}-\varepsilon)h/2)\}$, where
Theorem~\ref{t:gaps} does not provide polynomial resolvent bounds; the vertical
shaded region is the support of $\tilde\chi$.}
\label{f:complex}
\end{figure}
$$
\Omega(h):=\{\Re\omega\in [\alpha''_0,\alpha''_1],\ \Im\omega\in [\nu_-h,\nu_+h]\}
$$
By part~2 of Proposition~\ref{l:grushin-ultimate} and Jensen's inequality, see
for example~\cite[\S 2]{fwl}, we have the polynomial bound, for some $N,C$,
\begin{equation}
  \label{e:rough-upper-bound}
M(h)\leq Ch^{-N}.
\end{equation}
By a standard argument approximating the indicator function of $[\alpha'_0,\alpha'_1]$ by
smooth functions from above and below, it is enough
to prove that for each $\chi\in C_0^\infty(\alpha_0,\alpha_1)$,
\begin{equation}
  \label{e:weyl-law-int}
(2\pi h)^{n-1}\sum_{j=1}^{M(h)} \chi(\Re\omega_j)=\int_K \chi(p)\,d\Vol_\sigma +\mathcal O(h).
\end{equation}
Let $\tilde\chi(\omega)$ be an almost analytic continuation of $\chi$, as discussed
in the beginning of~\S\ref{s:trace}. We may assume that $\supp\tilde\chi\subset \{\Re\omega\in (\alpha''_0,\alpha''_1)\}$.

By Proposition~\ref{l:trace}, we have (with the integral over the vertical parts of $\partial\Omega(h)$ vanishing
since $\tilde\chi=0$ there)
\begin{equation}
  \label{e:intint}
{(2\pi h)^{n-1}\over 2\pi i}\oint_{\partial\Omega(h)}\tilde\chi(\omega){\partial_\omega F(\omega)\over F(\omega)}
\,d\omega=\int_K\chi(p)\,d\Vol_\sigma+\mathcal O(h).
\end{equation}
By Lemma~$\alpha$ in~\cite[\S 3.9]{tit} and the exponential
estimates of part~2 of Proposition~\ref{l:grushin-ultimate} (splitting the region $\Omega(h)$ into boxes of size $h$
and applying Lemma~$\alpha$ to each of these boxes, transformed into the unit disk by the Riemann mapping theorem),
we have for some fixed $N$,
$$
{\partial_\omega F(\omega)\over F(\omega)}=\sum_{j=1}^{M(h)}{1\over \omega-\omega_j}+G(\omega);\quad
G(\omega)=\mathcal O(h^{-N}),\quad
\omega\in\Omega(h)\cap\supp\tilde\chi.
$$
Applying Stokes theorem to~\eqref{e:intint} (over the contour comprised of $\partial\Omega(h)$
minus the sum of circles of small radius $r$ centered at each $\omega_j$, and letting $r\to 0$)
we get
$$
(2\pi h)^{n-1}\sum_{j=1}^{M(h)}\tilde\chi(\omega_j)=\int_K\chi(p)\,d\Vol_\sigma-
{(2\pi h)^{n-1}\over 2\pi i}\int_{\Omega(h)}{\partial_\omega F(\omega)\over F(\omega)}
\partial_{\bar\omega}\tilde\chi(\omega)\,d\bar\omega\wedge d\omega+\mathcal O(h).
$$
Since $\tilde\chi$ is almost analytic and $\Omega(h)$ lies $\mathcal O(h)$ close
to the real line, we have $\partial_{\bar\omega}\tilde\chi(\omega)=\mathcal O(h^\infty)$
for $\omega\in\Omega(h)$. Therefore, the second integral on the right-hand side is $\mathcal O(h^\infty)$
and we get
$$
(2\pi h)^{n-1}\sum_{j=1}^{M(h)}\tilde\chi(\omega_j)=\int_K \chi(p)\,d\Vol_\sigma+\mathcal O(h).
$$
Since $\tilde\chi(\omega)=\chi(\Re\omega)+\mathcal O(h)$ for $\omega\in\Omega(h)$, we get
\begin{equation}
  \label{e:weyl-law-int-2}
(2\pi h)^{n-1}\sum_{j=1}^{M(h)}\chi(\Re\omega_j)=\int_K\chi(p)\,d\Vol_\sigma+\mathcal O(h(1+h^{n-1}M(h))).
\end{equation}
Since one can take $\chi$ to be any compactly supported function on $(\alpha_0,\alpha_1)$,
and $M(h)=\mathcal O(h^{-N})$ for some fixed $N$ and any choice of $(\alpha''_0,\alpha''_1)$,
by induction we see from~\eqref{e:weyl-law-int-2} that $M(h)=\mathcal O(h^{1-n})$.
Given this bound, \eqref{e:weyl-law-int-2} implies~\eqref{e:weyl-law-int}, which finishes
the proof.

\appendix
\section{Example of a manifold with\\ $r$-normally hyperbolic trapping}
\label{s:example}

In this Appendix, we provide a simple example of an even asymptotically
hyperbolic manifold (as defined in~\S\ref{s:framework-ah}) whose geodesic
flow satisfies the dynamical assumptions of~\S\ref{s:dynamics}
and the pinching condition~\eqref{e:pinching}, therefore our Theorems~\ref{t:gaps}--\ref{t:resonant-states} apply.
This example is a higher dimensional generalization of the hyperbolic cylinder,
considered for instance in~\cite[Appendix~B]{fwl}.

The resonances for the
provided example can be described explicitly via the eigenvalues of the
Laplacian on the underlying compact manifold $N$, using separation of variables.
However, our results apply to small perturbations of the metric (see~\S\ref{s:stability}),
as well as to subprincipal perturbations in the considered operator, when
separation of variables no longer takes place.

Let $(N,\tilde g)$ be a compact $n-1$ dimensional Riemannian manifold
(at the end of this appendix, we will impose further conditions on $\tilde g$).
We consider the manifold $M=\mathbb R_r\times N_\theta$ with the metric
$$
g=dr^2+\cosh^2 r\,\tilde g(\theta,d\theta).
$$
Then $M$ has two infinite ends $\{r=\pm\infty\}$; near each of these ends,
one can represent it as an even asymptotically hyperbolic manifold by taking
the boundary defining function $\tilde x=e^{\mp r}$:
$$
g={d\tilde x^2\over \tilde x^2}+{(1+\tilde x^2)^2\over 4\tilde x^2}\tilde g(\theta,d\theta).
$$
The resonances for the Laplace--Beltrami operator
on $M$ therefore fit into the framework of~\S\ref{s:framework-assumptions},
as demonstrated in~\S\ref{s:framework-ah}. The associated flow $e^{tH_p}$
is the geodesic flow on the unit cotangent $S^*M$, extended to a homogeneous
flow of degree zero on the complement of the zero section in $T^*M$.

We now verify the assumptions of~\S\ref{s:dynamics}.
If $\xi_r,\xi_\theta$ are the momenta dual to $r,\theta$, then
$$
p^2=\xi_r^2+\cosh^{-2}r\,\tilde g^{-1}(\theta,\xi_\theta),
$$
where $\tilde g^{-1}$ is the dual metric to $g$, defined on the fibers of $T^*N$.
We then have
$$
H_p r={\xi_r\over p},\quad
H_p\xi_r={p^2-\xi_r^2\over p}\tanh r.
$$
The trapped set $K$ and the incoming/outgoing tails $\Gamma_\pm$
are given by
$$
\Gamma_\pm=\{\xi_r=\pm p\tanh r\},\quad
K=\{r=0,\ \xi_r=0\},
$$
or strictly speaking, by the intersections of the sets above with the set $\overline{\mathcal U}$
from~\eqref{e:sets-ah}.

Consider the following defining functions of $\Gamma_\pm$:
$$
\varphi_\pm=\xi_r\mp p\tanh r,
$$
then $\{\varphi_+,\varphi_-\}|_K=2p$ and thus assumptions~\eqref{aa:basic} and~\eqref{aa:symplectic}
of~\S\ref{s:dynamics} are satisfied. Next,
$$
H_p\varphi_\pm=\mp c_\pm\varphi_\pm,\quad
c_\pm=1\pm {\xi_r\over p}\tanh r.
$$
In particular, $c_\pm|_K=1$ and, arguing as in the proof of Lemma~\ref{l:phi-pm}, we get
$$
\nu_{\min}=\nu_{\max}=1.
$$
In particular, the pinching condition~\eqref{e:pinching} is satisfied.

Finally, in order for the $r$-normal hyperbolicity assumption~\eqref{aa:r-nh} of~\S\ref{s:dynamics}
to be satisfied, we need to make $\mu_{\max}\ll 1$, with $\mu_{\max}$
defined in~\eqref{e:mu-max}. This is a condition on the underlying
compact Riemannian manifold $(N,\tilde g)$, since $\mu_{\max}$ is the maximal expansion
rate of the geodesic flow of $\tilde g$ on the unit cotangent bundle $S^*N$.
To satisfy this condition, we can start with an arbitrary compact Riemannian manifold
and multiply its metric by a large constant $C^2$; indeed, if
$\varphi_t$ is the geodesic flow on the original manifold, then
$\varphi_{C^{-1}t}$ is the geodesic flow on the rescaled manifold and
the resulting $\mu_{\max}$ is divided by $C$.

\smallsection{Acknowledgements} I would like to thank Maciej Zworski for
plenty of helpful advice and constant encouragement throughout this project,
and Andr\'as Vasy, St\'ephane Nonnenmacher, Charles Pugh, Colin Guillarmou,
and Michael Hitrik for many very helpful discussions. I am very grateful to
Fr\'ed\'eric Faure for a discussion of~\cite{f-t,f-t2,f-t3}. I am also
thankful to two anonymous referees for suggestions to improve the manuscript.
Part of this work
was completed while visiting \'Ecol\'e Normale Sup\'erieure in October 2012.
This work was also partially supported by the NSF grant DMS-1201417.


\def\arXiv#1{\href{http://arxiv.org/abs/#1}{arXiv:#1}}

\end{document}